\documentclass{amsart}

\usepackage{xcolor}
\usepackage{hyperref} 

\usepackage[doi=false,isbn=false,url=false,backend=bibtex]{biblatex}  
\usepackage{amsmath,amsfonts,amssymb,amsthm,mathrsfs,bbm,array} 
\usepackage{microtype} 
\usepackage{enumitem} 
\usepackage{multicol} 
\usepackage{bm} 

\numberwithin{equation}{section}  
\bibliography{biblio.bib}  

\DeclareFontFamily{OT1}{pzc}{}
\DeclareFontShape{OT1}{pzc}{m}{it}{<-> s * [1.10] pzcmi7t}{}
\DeclareMathAlphabet{\mathpzc}{OT1}{pzc}{m}{it}

\let\originalleft\left
\let\originalright\right
\renewcommand{\left}{\mathopen{}\mathclose\bgroup\originalleft}
\renewcommand{\right}{\aftergroup\egroup\originalright}

\theoremstyle{plain}
\newtheorem{thm}{Theorem}[section]
\newtheorem{lem}[thm]{Lemma}
\newtheorem{prop}[thm]{Proposition}

\theoremstyle{definition}
\newtheorem{rem}[thm]{Remark}

\newcommand{\<}{\langle}
\renewcommand{\>}{\rangle}
\renewcommand{\(}{\left(}
\renewcommand{\)}{\right)}
\renewcommand{\tilde}{\widetilde}
\renewcommand{\bar}[1]{\mkern 1.5mu\overline{\mkern-1.5mu#1\mkern-1.5mu}\mkern 1.5mu}
\newcommand{\bs}{\backslash}
\newcommand{\mm}[4]{\begin{pmatrix} #1 & #2 \\ #3 & #4 \end{pmatrix}} 

\renewcommand{\Re}{\mathrm{Re}}
\renewcommand{\Im}{\mathrm{Im}}
\newcommand{\SL}{\mathrm{SL}}
\newcommand{\sgn}{\mathrm{sgn}}
\newcommand{\loc}{\mathrm{loc}}
\renewcommand{\mod}{\mathrm{mod\,}}
\newcommand{\id}{\mathrm{id}}

\newcommand{\supp}{\mathrm{supp}}

\newcommand{\onebb}{\mathbbm{1}}  
\newcommand{\dse}{\mathbbm{e}}
\newcommand{\R}{\mathbb{R}}
\renewcommand{\P}{\mathbb{P}}
\newcommand{\C}{\mathbb{C}}
\newcommand{\Z}{\mathbb{Z}}
\newcommand{\Q}{\mathbb{Q}}
\newcommand{\dsE}{\mathbb{E}}
\newcommand{\dsH}{\mathbb{H}}

\newcommand{\fa}{\mathfrak{a}}

\newcommand{\scc}{\mathpzc{c}}
\newcommand{\scp}{\mathpzc{p}}
\newcommand{\scC}{\mathcal{C}}
\newcommand{\scG}{\mathcal{G}}
\newcommand{\scK}{\mathcal{K}}

\newcommand{\tla}{\tilde{a}}
\newcommand{\tlA}{\tilde{A}}
\newcommand{\tlb}{\tilde{b}}
\newcommand{\tlB}{\tilde{B}}
\newcommand{\tlE}{\tilde{E}}
\newcommand{\tlg}{\tilde{g}}
\newcommand{\tlgamma}{\tilde{\gamma}}
\newcommand{\tlGamma}{\tilde{\Gamma}}
\newcommand{\tlh}{\tilde{h}}
\newcommand{\tlk}{\tilde{k}}
\newcommand{\tlK}{\tilde{K}}
\newcommand{\tlm}{\tilde{m}}
\newcommand{\tlM}{\tilde{M}}
\newcommand{\tln}{\tilde{n}}
\newcommand{\tlN}{\tilde{N}}
\newcommand{\tls}{\tilde{s}}
\newcommand{\tlSL}{\tilde{\SL}}
\newcommand{\tlxi}{\tilde{\xi}}

\newcommand{\h}{\hspace{.1in}}
\newcommand{\hh}{\hspace{.2in}}
\newcommand{\hhh}{\hspace{.3in}}

\begin{document}
\title[Metaplectic Eisenstein distributions]{Metaplectic Eisenstein distributions}
\author{Brandon Bate}
\address{Department of Mathematics, Houghton College, 1 Willard Ave, Houghton, NY 14744}
\email{brandonbate@gmail.com}
\date{}
\subjclass[2010]{11F37}

\begin{abstract}
In this paper, we define distributional analogues of the metaplectic Eisenstein series for $\SL_2(\R)$, which we refer to as
\textit{metaplectic Eisenstein distributions}.  We prove that these distributions
have meromorphic continuation and give an explicit functional equation that they satisfy.
From these results, we deduce the
functional equation for the corresponding classical metaplectic Eisenstein series.
\end{abstract}

\dedicatory{To Wilfried Schmid}

\maketitle

\section{Introduction}
\label{sec:Introduction}

In 1967, Robert Langlands outlined a definition for $L$-functions associated to
automorphic representations.  He then conjectured that these $L$-functions, commonly
referred to as \textit{automorphic $L$-functions}, can be meromorphically continued
and satisfy a functional equation \cite{Langlands67}.  This conjecture, which falls under the purview of the Langlands program, 
has only been proved for a relatively small number of automorphic $L$-functions.
A principal method employed in these cases has been the Rankin-Selberg method \cite{Bump05}.
In the Rankin-Selberg method, an automorphic $L$-function for an
automorphic representation $(\pi,V)$ is computed by taking a sufficiently ``nice'' element of the representation space $V$,
and considering the integral of this function multiplied by an Eisenstein series and some other suitable automorphic function.
Simplifications of the resulting integral yield the desired automorphic $L$-function, along with an integral
which is believed to simplify to a ratio of Gamma functions.\footnote{In reality, such a simplification
into Gamma functions is extremely difficult and rarely accomplished.}
Since the analytic variable of the resulting $L$-function
corresponds to the analytic variable of the Eisenstein series, the meromorphic continuation and functional
equation for such $L$-functions can be established from that of the Eisenstein series \cite{Bump05}.

As a part of one of the most celebrated applications of the Rankin-Selberg method,
Shimura first established the meromorphic continuation 
of a metaplectic Eisenstein series based at the cusp at infinity \cite{Shimura75}.
By then unfolding a Rankin-Selberg integral composed from this metaplectic Eisenstein series,
a $\theta$-function, and a weight $k$ Hecke eigenform $f$, Shimura was able 
to prove the meromorphic continuation of $L(s, \text{Sym}^2 f)$, the symmetric square $L$-function for $f$.
Since then, metaplectic Eisenstein series and their generalizations have become a fixture in modern number theory,
being utilized to prove the meromorphic continuation and functional equations of other automorphic $L$-functions,
and utilized in the study of average values of $L$-functions and multiple Dirichlet series \cite{Bump12,Ginzburg93,Goldfeld85}.

One of the drawbacks to the Rankin-Selberg method, which we alluded to earlier, is the difficulty involved in computing
the aforementioned ratio of Gamma functions, especially in the archimedean places.
Fortunately, the recent adaptation of the Rankin-Selberg method
to automorphic distributions by Miller and Schmid \cite{Miller12, Miller08}, has in large part simplified the process of calculating
these Gamma factors.  The difficulty however, is that one must work with automorphic distributions, and in particular,
one must utilize Eisenstein distributions, the distributional analogues of Eisenstein series.
Thus in order to adapt the Rankin-Selberg integrals which utilize metaplectic Eisenstein series to this
distributional setting, we must first define the metaplectic Eisenstein distributions, prove their meromorphic continuation,
and compute an explicit functional equation. By doing this, it is believed that Rankin-Selberg integrals, such as
the one in \cite{Ginzburg93}, can be adapted to the distributional setting, and thereby provide a means for
the explicit computation of the corresponding Gamma factors.

As indicated in \cite{Miller11, Miller08}, Eisenstein distributions are identified as elements
of distributional principal series for the group in question.  Thus we begin Section \ref{sec:DblCvrsSL2}
with a review of the metaplectic group and its corresponding distributional principal series representations.
In Section \ref{sec:MetaEisenDist}, we continue our analysis of distributional principal series representations,
concluding with a definition of $\tlE_{\epsilon,\nu}^{(\infty)}$,
the metaplectic Eisenstein distribution based at the cusp at $\infty$.
In Section \ref{sec:GaussKloo}, we find that the Fourier coefficients of $\tlE_{\epsilon,\nu}^{(\infty)}$
will in part consist of Gauss and Kloosterman sums, and we conclude this section by proving explicit
formulas for these sums.
Then in Section \ref{sec:MeroContFourCoeff}, we use the results of Section \ref{sec:GaussKloo} to 
compute explicit formulas for the Fourier coefficients, which in turn show that the Fourier coefficients
of $\tlE_{\epsilon,\nu}^{(\infty)}$ have meromorphic continuation.
The exact formulas for the Fourier coefficients is stated in Theorem \ref{thm:SimplFourCoeff}.

In Section \ref{sec:MeroContEisenInf}, 
we show that $\tlE_{\epsilon,\nu}^{(\infty)}$ has holomorphic continuation to $\C$ except for simple poles on $\{0,\frac{1}{2}\}$.
We do this by utilizing the meromorphic continuation of its Fourier coefficients.
Establishing the meromorphic continuation of $\tlE_{\epsilon,\nu}^{(\infty)}$ in this manner is preferable since it
affords us sharper information about the meromorphic continuation compared to spectral theory.
Some care is required since the Fourier series for  $\tlE_{\epsilon,\nu}^{(\infty)}$
does not capture its behavior ``at infinity."  To handle this situation, we make reference to
the notion of a distribution ``vanishing'' to a given order at a point \cite{Miller04}.
In Proposition \ref{prop:EisenInfMeroCont} we give a series expansion for  $\tlE_{\epsilon,\nu}^{(\infty)}$
which accounts for this behavior ``at infinity", and which also has the property that it
can be meromorphically continued.

In Section \ref{sec:MetaAtZero}, we define $\tlE_{\epsilon,\nu}^{(0)}$, the metaplectic Eisenstein distribution based at the cusp $0$.
Most of the work done in proving the meromorphic continuation for $\tlE_{\epsilon,\nu}^{(\infty)}$ can be directly
applied to prove the meromorphic continuation of $\tlE_{\epsilon,\nu}^{(0)}$.  We conclude this section with an explicit
computation of the zeroeth Fourier coefficient of $\tlE_{\epsilon,\nu}^{(0)}$, which is a necessary ingredient for 
our proof of the functional equation.  In Section \ref{sec:Intertwine}, we prove some results concerning the
intertwining operators for the metaplectic group, and then use these results to establish series expansions for the intertwined
metaplectic Eisenstein distributions.
In Section \ref{sec:EisenDistFunctionalEq}, we prove a functional equation for the metaplectic
Eisenstein distributions, and then in Section \ref{sec:EisenSeriesFunctionalEq}, obtain from this a functional equation
for the corresponding classical metaplectic Eisenstein series, which were first introduced by Kubota \cite{Kubota69}.
To the best of the author's knowledge, no such explicit functional equation for the classical metaplectic
Eisenstein series is recorded in the literature,
although \cite{Moeglin95} gives a road map for proving the adelic functional equation.
Given that its derivation from the distributional
functional equation required only a little extra work, it seemed appropriate to take the opportunity
to record it here.

The author would like to thank Stephen D. Miller for his advice and assistance in preparing this
paper, Henryk Iwaniec for his insightful conversations on this subject, and the referees for their helpful comments.
Finally, the author would like to thank Wilfried Schmid, whose work in the
theory of automorphic distributions has served as a foundation for this article.

\section{Principal Series Representations for \texorpdfstring{$\tlSL_2$}{the Metaplectic Group}}
\label{sec:DblCvrsSL2}

In this section we define the metaplectic group and review some facts concerning its principal
series representations.  In Section \ref{sec:MetaEisenDist} we will define the metaplectic 
Eisenstein distributions to be elements of these principal series representation spaces.

As a set, let $\tlSL_2(\R) = \SL_2(\R) \times \{\pm 1\}$.  Recall that the Hilbert
symbol for $\R$ is given by the formula
\begin{equation}
(x,y)_H = \begin{cases}-1 &\text{if } x<0 \text{ and } y<0\\ 1 &\text{otherwise},\end{cases}
\end{equation}
where $x,y \in \R_{\neq 0}$. For $\mm{a}{b}{c}{d} \in \SL_2(\R)$, define
\begin{equation}
X\(\mm{a}{b}{c}{d}\) = \begin{cases}c & \text{if } c \neq 0\\d & \text{if } c = 0.\end{cases}
\end{equation}
For $g_1, g_2 \in \SL_2(\R)$, define
\begin{equation}
\alpha(g_1,g_2) = \(\frac{X(g_1 g_2)}{X(g_1)},\frac{X(g_1 g_2)}{X(g_2)}\)_H.
\end{equation}
One can show that $\alpha$ is a 2-cocycle \cite{Kazhdan84,Kubota67}, and thus we give $\tlSL_2(\R)$ the structure
of a group by defining the following multiplication law:
\begin{equation}
(g_1,\epsilon_1) \cdot (g_2,\epsilon_2) = (g_1 g_2, \alpha(g_1, g_2) \epsilon_1 \epsilon_2),
\end{equation}
where $g_1, g_2 \in \SL_2(\R)$ and $\epsilon_1, \epsilon_2 \in \{\pm 1 \}$.
In addition to identifying $\tlSL_2(\R)$ as a group, one can give $\tlSL_2(\R)$
the structure of a smooth manifold by requiring the map
\begin{equation}
\label{eq:covermap}
\tlg=(g,\epsilon) \mapsto g
\end{equation}
to be a smooth covering map from $\tlSL_2(\R)$ onto $\SL_2(\R)$.
With respect to this smooth structure, one can check that $\tlSL_2(\R)$ is a Lie group.
This group is commonly referred to as \textit{the metaplectic group}.

Let
\begin{align}
\label{eq:sltl2elems}
&\tlm = \tlm_{\epsilon_1,\epsilon_2} = \(\mm{\epsilon_1}{0}{0}{\epsilon_1}, \epsilon_2\),&  &\tla = \tla_u  = \(\mm{u}{0}{0}{u^{-1}},1\),& \\
&\tln = \tln_x = \(\mm{1}{x}{0}{1},1\),& &\tln_- = \tln_{-,x} = \(\mm{1}{0}{x}{1},1\),&\notag\\
&\tlk = \tlk_\theta = \(\mm{\cos(\theta)}{-\sin(\theta)}{\sin(\theta)}{\cos(\theta)}, \varepsilon(\theta)\),\notag
\end{align}
where $\epsilon_1, \epsilon_2 \in \{\pm 1\}$, $u \in \R_{>0}$, $x \in \R$, $\theta \in [-2\pi,2\pi)$, and
\begin{equation}
\label{eq:epsDef}
\varepsilon(\theta) = \begin{cases} 1 &\text{if } -\pi \leq \theta < \pi \\ -1 &\text{if } -2 \pi \leq \theta < -\pi \text{ or } \pi \leq \theta < 2 \pi. \end{cases}
\end{equation}
We extend the definition of $\tlk_\theta$ periodically to all $\theta \in \R$.
As indicated in the above definitions, sometimes we will suppress the variables
$\epsilon_1$, $\epsilon_2$, $u$, $x$, and $\theta$ in our notation.  Occasionally we will write $\tlm(\epsilon_1,\epsilon_2)$,
$\tla(u)$, $\tln(x)$, $\tln_-(x)$, and $\tlk(\theta)$ for $\tlm_{\epsilon_1,\epsilon_2}$, $\tla_u$, $\tln_x$, $\tln_{-,x}$, and
$\tlk_\theta$ (respectively).
Let
\begin{align}
\label{eq:tlSL2Subgroup}
&\tlM = \left\{ \tlm_{\epsilon_1,\epsilon_2}: \epsilon_1,\epsilon_2 \in \{\pm 1\}\right\},& &\tlA = \left\{ \tla_u : u \in \R_{>0} \right\},&\\
&\tlN = \left\{ \tln_{x}: x \in \R \right\},& &\tlN_{-} = \left\{ \tln_{-,x}: x \in \R \right\},&\notag\\
&\tlK = \{ k_\theta: \theta \in [-2\pi,2\pi) \},&   &\tlB = \tlM \tlA \tlN_-.\notag& 
\end{align}
One can check that the sets defined in \eqref{eq:tlSL2Subgroup} are subgroups of $\tlSL_2(\R)$. One can also verify that
$\tlM$ is the center $\tlSL_2(\R)$ and that
$\theta \mapsto \tlk_\theta$ is a continuous isomorphism from $\R / [-2\pi,2\pi)$ onto $\tlK$.

Let $\omega$ denote a quasi-character of $\tlB$, and let
\begin{align}
\label{eq:indRep}
&V_\omega^\infty = \{f \in C^\infty(\tlSL_2(\R)):f(\tlg \tlb)=\omega(\tlb^{-1})f(\tlg) \text{ for all } \tlg \in \tlSL_2(\R), \tlb \in \tlB\},\\
&V_\omega = \{f \in L_\loc^2(\tlSL_2(\R)):f(\tlg \tlb)=\omega(\tlb^{-1})f(\tlg) \text{ for all } \tlg \in \tlSL_2(\R), \tlb \in \tlB\},\notag\\
&V_\omega^{-\infty} = \{f \in C^{-\infty}(\tlSL_2(\R)):f(\tlg \tlb)=\omega(\tlb^{-1})f(\tlg) \text{ for all } \tlg \in \tlSL_2(\R), \tlb \in \tlB\},\notag
\end{align}
where $C^\infty(\tlSL_2(\R))$ denotes the spaces of smooth functions on $\tlSL_2(\R)$, $L_\loc^2(\tlSL_2(\R))$ denotes
the space of locally $L^2$-functions  on $\tlSL_2(\R)$ (modulo the space of functions which vanish almost everywhere),
and $C^{-\infty}(\tlSL_2(\R))$ denotes the space of distributions on $\tlSL_2(\R)$.
The equality specified in the definition of $V_\omega^{-\infty}$ is interpreted as an equality between distributions on $\tlSL_2(\R)$. 
We refer to the spaces in \eqref{eq:indRep} as \textit{smooth (resp. $L^2$, distributional) principal series representations spaces}.
Each of these spaces comes equipped with the action of left inverse multiplication by elements of $\tlSL_2(\R)$.
We denote these actions by $\pi = \pi_\omega$.
It is clear from the definitions in \eqref{eq:indRep} that
\begin{equation}
\label{eq:indRepIncls}
V^\infty_\omega \subset V_\omega \subset V_\omega^{-\infty}.
\end{equation}

We define the quasi-characters of $\tlB$ which will be of interest to us in this paper.
For $\nu \in \C$, let 
\begin{equation}
\label{eq:mudef}
\mu_\nu(\tla_u) = u^{\nu-1},
\end{equation}
which is a quasi-character of $\tlA$ (the $-1$ in the exponent in \eqref{eq:mudef} is a normalizing factor
common in representation theory).  For $\epsilon \in \{\pm 1\}$, let $\sigma_\epsilon:\tlM \to \C^*$ denote the character of $\tlM$
defined by the following identities:
\begin{align}
\label{eq:sigmaDef}
&\sigma_\epsilon\( \mm{1}{0}{0}{1}, \pm 1 \) = \pm 1,&
&\sigma_\epsilon\( \mm{-1}{0}{0}{-1}, \pm 1 \) = \mp \epsilon \, i.&
\end{align}
For $\tlb = \tlm \tla \tln_-$, we define
\begin{equation}
\label{eq:omegaepsnu}
\omega_{\epsilon,\nu}(\tlb) = \sigma_{\epsilon}(\tlm) \mu_\nu(\tla),
\end{equation}
which is a quasi-character of $\tlB$.
In order to simplify notation, we will write 
$V_{\epsilon,\nu}^{\infty}$ for $V_{\omega_{\epsilon,\nu}}^{\infty}$.  A similar notation will also be used
for the $L^2$ and distributional principal series representation spaces.

For $f \in V_{\epsilon,\nu}^\infty$, define $f_0, f_\infty \in C^\infty(\R)$ by the equalities
\begin{equation}
\label{eq:Sub0SubInfDef}
f_0(x) = f(\tln_x) \text{ and } f_\infty(x) = (\pi(\tls)f)(\tln_x) = f(\tls^{-1} \tln_x),
\end{equation}
where
\begin{equation}
\label{eq:tlsdef}
\tls  =  \tlk_{\frac{\pi}{2}} = \( \mm{0}{-1}{1}{0}, 1\).
\end{equation}
We will now explain how $f_0$ and $f_\infty$ are related to each other.  To do so, let
\begin{align}
\label{eq:branchCut}
&\log(z) \text{ denote the branch cut of the logarithm whose imaginary} \\
&\text{part lies in } (-\pi,\pi], \text{ and let } z^t = \exp(\log(z) t)  \text{ where } t \in \C, z \in \C^*.\notag
\end{align}
Since for $x \neq 0$,
\begin{align}
\label{eq:tlsAction}
\tls^{-1} \tln_x = \tln_{-x^{-1}} \(\tla_{|x|} \tlm_{\sgn(-x),\sgn(-x)} \tln_{-,-x}\)^{-1},
\end{align}
we find for $f \in V_{\epsilon,\nu}^{\infty}$,
\begin{equation}
\label{eq:SmoothSigmaIndRepUnbddEq}
f_\infty(x) = \sgn(-x)^{\epsilon/2} |x|^{\nu-1} f_0(-x^{-1})
\end{equation}
for $x \neq 0$.
Conversely, when given $f_0,f_\infty \in C^{\infty}(\R)$ which satisfy \eqref{eq:SmoothSigmaIndRepUnbddEq}, 
it follows that one can define a unique element $f \in V_{\epsilon,\nu}^{\infty}$ 
such that \eqref{eq:Sub0SubInfDef} holds since $\tlN$ and $\tls^{-1} \tlN$ cover $\tlSL_2(\R) / \tlB$.
Thus
\begin{align}
\label{eq:SmoothSigmaIndRepUnbddIso}
&V_{\epsilon,\nu}^{\infty} \cong \{(f_0,f_\infty) \in C^{\infty}(\R)^2 : f_0 \text{ and } f_\infty \text{ satisfy \eqref{eq:SmoothSigmaIndRepUnbddEq} }\}.
\end{align}

For $f_1 \in V_{\epsilon,\nu}$ and $f_2 \in V_{-\epsilon,-\nu}$
we define the following pairing:
\begin{equation}
\label{eq:pairingDef}
\<f_1,f_2\>_{\epsilon,\nu} = \int_{-\pi/2}^{\pi/2} f_1(\tlk_\theta) f_2(\tlk_\theta) \, d\theta.
\end{equation}
We claim that the integration is in fact over the compact group $\tlK / \tlM$.
To see this, first recall that $\theta \mapsto \tlk_\theta$ is a continuous isomorphism from $\R / [-2\pi,2\pi)$ onto  $\tlK$.
Therefore a Haar measure on $\tlK$ is given by $d\theta$ under the parameterization $\theta \mapsto \tlk_\theta$.
Since $\tlM = \{\tlk_{-2\pi}, \tlk_{-\pi}, \tlk_{0}, \tlk_{\pi}\}$, it follows that 
$d\theta$ likewise is a Haar measure for $\tlK / \tlM$, allowing any interval of the form $[a,a+\pi)$ to serve as a 
fundamental domain for $\tlK / \tlM$.
In \eqref{eq:pairingDef}, we chose the fundamental domain $[-\frac{\pi}{2},\frac{\pi}{2})$.
Since the integrand is $\tlM$-invariant and since $\tlK / \tlM$ is compact,
it follows that the integral converges and corresponds to integration over $\tlK / \tlM$.
Since $\tlSL_2(\R) / \tlB \cong \tlK / \tlM$, the pairing is clearly non-degenerate.
Furthermore, it can be shown that the dual representation of
$(\pi,V_{\epsilon,\nu})$ under this pairing agrees with the representation $(\pi, V_{-\epsilon,-\nu})$.
This allows us to identify $(\pi, V_{-\epsilon,-\nu})$ with the dual of $(\pi,V_{\epsilon,\nu})$.

The pairing \eqref{eq:pairingDef} extends (on the left) and restricts (on the right) to
$V_{\epsilon,\nu}^{-\infty} \times V_{-\epsilon,-\nu}^{\infty}$;
we can extend to $V_{\epsilon,\nu}^{-\infty}$ since distributions are given locally by integration against
derivatives of locally $L^2$ functions. 
Via this pairing we can identify $V_{\epsilon,\nu}^{-\infty}$ with the dual of 
$V_{-\epsilon,-\nu}^{\infty}$, where we equip $V_{-\epsilon,-\nu}^{\infty}$ with the usual test function topology.
Relative to this topology, we equip $V_{\epsilon,\nu}^{-\infty}$ with the strong topology
(see \cite[\S 19]{Treves06} for the definition of the strong topology for topological vector spaces in general).
This topology on $V_{\epsilon,\nu}^{-\infty}$ is also known as the \textit{strong distribution topology}.
With this topology, one can check that $V_{\epsilon,\nu}^\infty$ is dense in $V_{\epsilon,\nu}^{-\infty}$.
The pairing $\<\cdot,\cdot\>_{\epsilon,\nu}$ on $V_{\epsilon,\nu}^{-\infty} \times V_{-\epsilon,-\nu}^\infty$
is not continuous.
However, it is separately continuous \cite[\S 41]{Treves06}, which means that
\begin{align}
\label{eq:sepCts}
&v \mapsto \<v',v\> \text{ is continuous for fixed } v' \in V_{\epsilon,\nu}^{-\infty}, \text{ and}\\
&v' \mapsto \<v',v\> \text{ is continuous for fixed } v \in V_{-\epsilon,-\nu}^\infty.\notag
\end{align}
It should also be noted that since $V_{\epsilon,\nu}^{-\infty}$ is the dual of a Montel space, it follows that 
sequential convergence in the strong distribution topology is equivalent to sequential convergence in the weak topology
\cite[\S 34.4]{Treves06}.

We wish to speak of the restriction of $f \in V_{\epsilon,\nu}^{-\infty}$ to $\tlN \cong \R$.
To do this in a meaningful way, we select $\phi \in C^\infty_c(\R)$ and define $h_\phi \in V_{-\epsilon,-\nu}^\infty$
(via the isomorphism in \eqref{eq:SmoothSigmaIndRepUnbddIso}) by the equation
\begin{equation}
(h_\phi)_0(x) = \phi(x).
\end{equation}
Although we have not specified $(h_\phi)_\infty(x)$, its precise definition is easily deduced from \eqref{eq:SmoothSigmaIndRepUnbddEq} 
since $\phi$ has compact support.
We define $f_0 \in C^{-\infty}(\R)$ by the equation
\begin{equation}
\label{eq:f0Def}
f_0(\phi) = \<f,h_\phi\>_{\epsilon,\nu}.
\end{equation}
We refer to $f_0$ as \textit{the restriction of $f$ to $\tlN$}. This definition
for the restriction of $f$ to $\tlN$ is appropriate since it agrees with the usual notion of restriction of
$f$ to $\tlN$ when $f$ is a function.  Indeed, if $f \in V_{\epsilon,\nu}$ then since 
\begin{equation}
\label{eq:tlkTotlnId}
\tlk_\theta = \tln_{-\!\tan(\theta)} \(\tla_{|\cos(\theta)|} \tlm_{\sgn(\cos(\theta)),1} \tln_{-,-\!\sin(\theta)\cos(\theta)}\)^{-1}
\end{equation}
for $\theta \in (-\frac{\pi}{2},\frac{\pi}{2})$, then
\begin{align}
\label{eq:argCmptToUnbdd}
&\<f,h_\phi\>_{\epsilon,\nu} = \int_{-\pi/2}^{\pi/2} f(\tlk_\theta) h_\phi(\tlk_\theta) d \theta\\
&= \int_{-\pi/2}^{\pi/2} f(\tln_{-\!\tan(\theta)}) |\cos(\theta)|^{\nu-1} \sigma_{\epsilon}(\tlm_{\sgn(\cos(\theta)),1})\notag\\
&\hh\hh \cdot h_\phi(\tln_{-\!\tan(\theta)})  |\cos(\theta)|^{-\nu-1} \sigma_{-\epsilon}(\tlm_{\sgn(\cos(\theta)),1}) d \theta\notag\\
&= \int_{-\pi/2}^{\pi/2} f_0(-\tan(\theta)) \phi(-\tan(\theta))  |\cos(\theta)|^{-2} d \theta\notag\\
&=\int_{-\infty}^\infty f_0(-t) \phi(-t) (\sqrt{1+t^2})^2 \frac{dt}{1+t^2} = \int_{-\infty}^\infty f_0(t) \phi(t) dt.\notag
\end{align}
In the last line we performed the change of variables $\theta \mapsto \arctan(t)$.

We can likewise define the restriction of $f \in V_{\epsilon,\nu}^{-\infty}$ to $\tls^{-1} N \cong \R$.
We define $f_\infty \in C^{-\infty}(\R)$ by the equation
\begin{equation}
f_\infty(\phi) = \<\pi(\tls)f,h_\phi\>_{\epsilon,\nu}.
\end{equation}
We refer to $f_\infty$ as the \textit{the restriction of $f$ to $\tls^{-1} \tlN$}.  
As before, one can check that the restriction
of $f$ to $\tls^{-1} \tlN$ is appropriate since it agrees with our usual notion of restriction of $f$
to $\tls^{-1} \tlN$ when $f$ is a function.
Since $\tlN$ and $\tls^{-1} \tlN$ cover $\tlSL_2(\R) / \tlB$,  it follows that $f$ is completely determined by $f_0$ and $f_\infty$.
Just as in the case of smooth $f$, we have that for $f \in V_{\epsilon,\nu}^{-\infty}$,
\begin{equation}
\label{eq:sigmaIndRepUnbddEq}
f_\infty(x) = \sgn(-x)^{\epsilon/2} |x|^{\nu-1} f_0(-x^{-1})
\end{equation}
as an equality between distributions on $\R_{\neq 0}$.
This can be deduced directly from the analogous statement for smooth functions in \eqref{eq:SmoothSigmaIndRepUnbddEq}
and the density of $V_{\epsilon,\nu}^\infty$ in $V_{\epsilon,\nu}^{-\infty}$.
Conversely, when given $f_0,f_\infty \in C^{-\infty}(\R)$ which satisfy \eqref{eq:sigmaIndRepUnbddEq}, 
it follows that one can define a unique element
$f \in V_{\epsilon,\nu}^{-\infty}$. Thus
\begin{align}
\label{eq:sigmaIndRepUnbddIso}
&V_{\epsilon,\nu}^{-\infty} \cong \{(f_0,f_\infty) \in C^{-\infty}(\R)^2 : f_0 \text{ and } f_\infty \text{ satisfy \eqref{eq:sigmaIndRepUnbddEq}} \}.
\end{align}

Sections 3-7 will rely heavily upon \eqref{eq:sigmaIndRepUnbddIso}, with little need to explicitly reference
back to the pairing $\<\cdot,\cdot\>_{\epsilon,\nu}$.
In Sections 8-10, we will again make explicit reference to $\<\cdot,\cdot\>_{\epsilon,\nu}$,
and in particular, need some of the following facts.

Notice that \eqref{eq:argCmptToUnbdd} applies not only to $f \in V_{\epsilon,\nu}$, but also to $f \in V_{\epsilon,\nu}^{-\infty}$.
By changing notation slightly,%
\footnote{The change in notation is slightly confusing: we replaced $h_\phi$ with $\phi$ and $\phi$ with $\phi_0$.}
we then have that
\begin{equation}
\label{eq:cmpctToUnbdd}
\<f,\phi\>_{\epsilon,\nu} = \int_{-\infty}^\infty f(\tln_t) \phi(\tln_t) \, dt = \int_{-\infty}^\infty f_0(t) \phi_0(t) \, dt,
\end{equation}
for $f \in V_{\epsilon,\nu}^{-\infty}$ and $\phi \in V_{-\epsilon,-\nu}^\infty$, provided $\supp(\phi) \subset \tlN \tlB$
or $\supp(f) \subset \tlN \tlB$.
Similarly, 
\begin{equation}
\label{eq:cmpctToUnbdd2}
\<f,\phi\>_{\epsilon,\nu} = \int_{-\infty}^\infty f(\tls^{-1} \tln_t) \phi(\tls^{-1} \tln_t) \, dt 
= \int_{-\infty}^\infty f_\infty(t) \phi_\infty(t) \, dt,
\end{equation}
provided $\supp(\phi) \subset \tls^{-1}\tlN\tlB$ or $\supp(f) \subset \tls^{-1}\tlN\tlB$.
For arbitrary $\phi \in V_{-\epsilon,-\nu}^\infty$, 
there exists $\phi_1, \phi_2 \in V_{-\epsilon,-\nu}^\infty$ such that
$\phi = \phi_1 + \phi_2$, $\supp(\phi_1) \subset \tlN \tlB$, and $\supp(\phi_2) \subset \tls^{-1} \tlN \tlB$.
Thus by \eqref{eq:cmpctToUnbdd} and \eqref{eq:cmpctToUnbdd2},
\begin{align}
\label{eq:phi1phi2breakup}
&\<f,\phi\>_{\epsilon,\nu} = \<f,\phi_1\>_{\epsilon,\nu} + \<f,\phi_2\>_{\epsilon,\nu}\\
&= \int_{-\infty}^\infty f_0(t) \(\phi_1\)_0(t) \, dt + \int_{-\infty}^\infty f_\infty(t) \(\phi_2\)_\infty(t) \, dt. \notag
\end{align}
The following proposition provides another convenient means by which to compute $\<f,\phi\>_{\epsilon,\nu}$ in certain cases.

\begin{prop}
\label{prop:pairingToCondConvIntegral}
Suppose $\Re(\nu) > 0$.
Let $f \in V_{\epsilon,\nu}^{-\infty}$ such that $f_0 \in L^\infty(\R)$ and such that 
\eqref{eq:sigmaIndRepUnbddEq} holds as an equality
between distributions on $\R$ (as opposed to an equality between distributions on $\R_{\neq 0}$).
Then for $\phi \in V_{-\epsilon,-\nu}^\infty$,
\begin{equation}
\label{eq:pairingToCondConvIntegral}
\<f,\phi\>_{\epsilon,\nu} = \lim_{m_1,m_2 \to \infty} \int_{-m_1}^{m_2} f_0(x) \phi_0(x) \, dx.
\end{equation}
\end{prop}

\begin{proof}
Let $\omega_0$ denote the trivial character on $\tlB$.
By \eqref{eq:indRep},
\begin{equation*}
V^\infty_{\omega_0} = \{h \in C^\infty(\tlSL_2(\R)): h(\tlg \tlb) = h(\tlg) \text{ for all } \tlg \in \tlSL_2(\R), \tlb \in \tlB\}.
\end{equation*}
As we did in \eqref{eq:Sub0SubInfDef}, let $h_0(x) = h(\tln_x)$ and $h_\infty(x) = h(\tls^{-1} \tln_x)$
for $h \in V_{\omega_0}^\infty$.
By \eqref{eq:tlsAction}, we see that for $h \in V^\infty_{\omega_0}$,
\begin{equation}
\label{eq:cutOffUnbddEq1}
h_\infty(x) = h_0(-x^{-1}).
\end{equation}
Conversely, since $\tlN$ and $\tls^{-1} \tlN$ cover $\tlSL_2(\R) / \tlB$, it follows that we can define a unique element
$h \in V^\infty_{\omega_0}$ for given $h_0$ and $h_\infty$ which satisfy \eqref{eq:cutOffUnbddEq1}.
With that in mind, we will now define an element of $\psi_t \in V_{\omega_0}^\infty$ for $t \in \R_{>0}$.
Let $\psi^*: \R \to [0,1]$ be a smooth cutoff function such that $\psi^*(x) \equiv 1$ near the origin.
Let $\(\psi_t\)_0(x) = \psi^*(x/t)$ and
\begin{equation}
\label{eq:cutOffUnbddEq2}
\(\psi_t\)_\infty(x) = \(\psi_t\)_0(-x^{-1}),
\end{equation}
for $x \in \R_{\neq 0}$, with $\(\psi_t\)_\infty(0) = 0$.
Since $\(\psi_t\)_\infty(x)$ is a smooth function and since $\(\psi_t\)_0$ and $\(\psi_t\)_\infty$ satisfy \eqref{eq:cutOffUnbddEq1},
they together define a unique element $\psi_t \in V_{\omega_0}^\infty$.

Let $\bm{1}$ denote the function $(\tlg \mapsto 1) \in V_{\omega_0}^\infty$.
Notice that for $\phi \in V_{-\epsilon,-\nu}^\infty$,
\begin{equation}
\label{eq:cutOffSupp}
\supp(\phi \cdot \psi_t) \subset \tlN \tlB \text{ and } \supp(\phi \cdot (\bm{1}-\psi_t)) \subset \tls^{-1}\tlN \tlB.
\end{equation}
Since
\begin{equation*}
\<f,\phi\>_{\epsilon,\nu} = \<f,\phi \cdot \psi_t\>_{\epsilon,\nu} + \<f,\phi \cdot (\boldsymbol{1}-\psi_t)\>_{\epsilon,\nu},
\end{equation*}
it follows then by \eqref{eq:cmpctToUnbdd} and \eqref{eq:cmpctToUnbdd2},
\begin{align}
\label{eq:cutOffStep1}
&\<f,\phi\>_{\epsilon,\nu} = \int_{-\infty}^\infty f_0(x) \phi_0(x) \(\psi_t\)_0(x) \, dx + \int_{-\infty}^\infty f_\infty(x) \phi_\infty(x) (\boldsymbol{1}-\psi_t)_\infty(x) \, dx.
\end{align}
The second integral on the right-hand side vanishes as $t \to \infty$ by \cite[Corollary 3.12]{Miller04}.%
\footnote{Corollary 3.12 is stated in terms of a distribution $\sigma$ on $\R\P^1$ and a connected open
neighborhood $I$ containing $\infty$.
In our application, $I = \R\P^1 - \{0\} \cong \R$ and $\sigma(x) = \sgn(-x)^{\epsilon/2}|x|^{-\nu+1}f_0(-x)$.
Thus by \eqref{eq:sigmaIndRepUnbddEq},
\begin{equation*}
\sigma(1/x) = \sgn(-x)^{\epsilon/2} |x|^{\nu-1} f_0(-x^{-1}) = f_\infty(x).
\end{equation*}
Since $f_\infty \in L^1_\loc(\R)$, it follows from \cite[Lemma 3.1]{Miller04} that $\sigma$ vanishes 
to non-negative order at $\infty$.
Applying \cite[(3.23)]{Miller04} to the conclusion of Corollary 3.12 yields the desired result.}
Hence
\begin{equation}
\label{eq:cutOffStep2}
\<f,\phi\>_{\epsilon,\nu} =  \lim_{t \to \infty} \int_{-\infty}^\infty f_0(x) \phi_0(x) \(\psi_t\)_0(x) \, dx.
\end{equation}
If $\phi_0 \in L^1(\R)$, which is guaranteed to occur when $\Re(\nu)>0$,
then the integrand in \eqref{eq:cutOffStep2} is also in $L^1(\R)$; hence by dominated
convergence we can replace the smooth cutoff function in \eqref{eq:cutOffStep2} with a hard cutoff.
Indeed, since
\begin{equation}
\tlk(-\arctan(t)) = \tln_t \( \tla\( (1+t^2)^{-1/2} \) \tln_-\(\frac{t}{1+t^2}\) \)^{-1},
\end{equation}
it then follows from \eqref{eq:indRep} and \eqref{eq:omegaepsnu} that
\begin{equation}
\label{eq:phi0Bnd}
|\phi_0(t)| = |\phi(\tln_t)| \ll (1+t^2)^{-\frac{1}{2}(\nu+1)},
\end{equation}
with the implied constant dependent upon $\sup_{\tlk \in \tlK} |\phi(\tlk)|$.
\end{proof}

\section{The Metaplectic Eisenstein Distribution at \texorpdfstring{$\infty$}{Infinity}}
\label{sec:MetaEisenDist}

In this section we define the metaplectic Eisenstein distribution based at the cusp $\infty$.
However before doing so, we prove some lemmas that will be helpful for computing the Fourier series expansion
for this metaplectic Eisenstein distribution. The proofs will make implicit reference to the following
transformation laws for $f \in V_{\epsilon,\nu}^{-\infty}$:
\begin{align}
\label{eq:TransLawsB}
&f(\tlg \tla_u \tlm_{-1,\kappa} \tln_-) = \epsilon \kappa i u^{-\nu+1} f(\tlg)\\
&f(\tlg \tla_u \tlm_{1,\kappa} \tln_-) = \kappa u^{-\nu+1} f(\tlg),\notag
\end{align}
where $\kappa, \epsilon \in \{\pm 1\}$, $\nu \in \C$, and $u \in \R_{>0}$.
One can check that \eqref{eq:TransLawsB} follows from \eqref{eq:indRep} and \eqref{eq:omegaepsnu}.

In the following lemma, we give explicit formulas for
$(\pi(\tlg)f)_0$ and $(\pi(\tlg)f)_\infty$ in terms of $f_0$ and $f_\infty$ for $f \in V^{-\infty}_{\epsilon,\nu}$.
In the statement of this lemma, we will use the following notational convention:
for $a,b,c,d \in \R$ and $\kappa \in \{\pm 1\}$, let $\kappa' \in \{\pm 1\}$ such that
\begin{equation}
\label{eq:tlsComm}
\(\mm{a}{b}{c}{d},\kappa\) \tls^{-1} = \tls^{-1} \(\mm{d}{-c}{-b}{a}, \kappa'\),
\end{equation}
with $\tls$ as defined in \eqref{eq:tlsdef}.

\begin{lem}
\label{lem:MetaUnbddAct}
Let $f \in V_{\epsilon,\nu}^{-\infty}$ and $\tilde{g}^{-1}=\left( \mm{a}{b}{c}{d}, \kappa \right) \in \tlSL_2(\R)$.

\begin{itemize}[leftmargin=.3in,font=\normalfont\textbf]

\item[(a)] If $c \neq 0$ then
\begin{equation*}
(\pi( \tilde{g}) f)_0(x) = \kappa (-c,cx+d)_H |cx+d|^{\nu-1} \sgn(cx+d)^{\epsilon/2} f_0 \left( \frac{ax + b}{cx + d} \right),
\end{equation*}
as an equality between distributions on $\R_{\neq \frac{-d}{c}}$.

\item[(b)] If $c = 0$ then
\begin{equation*}
(\pi(\tlg) f)_0(x) = \kappa |d|^{\nu - 1} \sgn(d)^{\epsilon/2} f_0\(\frac{x}{d^2} + \frac{b}{d}\),
\end{equation*}
as an equality between distributions on $\R$.

\item[(c)] If $b \neq 0$ then
\begin{align*}
&(\pi( \tilde{g}) f)_\infty(x)\\
&= \kappa' (b,-bx+a)_H |-bx+a|^{\nu-1} \sgn(-bx + a)^{\epsilon/2} f_\infty \left( \frac{dx -c}{-bx + a} \right),
\end{align*}
as an equality between distributions on $\R_{\neq \frac{a}{b}}$.

\item[(d)] If $b = 0$ then
\begin{equation*}
(\pi( \tilde{g}) f)_\infty(x) = \kappa' |a|^{\nu-1} \sgn(a)^{\epsilon/2} f_\infty \left( \frac{x}{a^2} - \frac{c}{a} \right),
\end{equation*}
as an equality between distributions on $\R$. 
\end{itemize}
Recall that $\kappa'$ is defined according to \eqref{eq:tlsComm}.
\end{lem}

\begin{proof}
For $c \neq 0$ and $x \neq \frac{-d}{c}$,
\begin{equation}
\tlg^{-1} \tln_x = \tln\(\frac{ax+b}{cx+d}\)  \tla(|cx+d|^{-1}) \tlm(\sgn(cx+d),\kappa (-c,cx+d)_H) \tln_-\(\frac{c}{cx+d}\)
\end{equation}
By utilizing \eqref{eq:TransLawsB},  we find that
\begin{equation*}
(\pi(\tlg) f)_0(x) = \kappa (-c,cx+d)_H |cx+d|^{\nu-1} \sgn(cx+d)^{\epsilon/2} f_0 \left( \frac{ax + b}{cx + d} \right),
\end{equation*}
as an equality between distributions on $\R_{\neq \frac{-d}{c}}$.  Similarly, when $c=0$ (which implies that 
$a = d^{-1} \neq 0$) we see that
\begin{equation}
\tlg^{-1} \tln_x = \tln\(\frac{x}{d^2} + \frac{b}{d}\) \tla(|d|^{-1}) \tlm_{\sgn(d),\kappa} 
\end{equation}
Thus
\begin{equation*}
(\pi(\tlg) f)_0(x) = \kappa |d|^{\nu - 1} \sgn(d)^{\epsilon/2} f_0\(\frac{x}{d^2} + \frac{b}{d}\),
\end{equation*}
as an equality between distributions on $\R$.  This proves parts (a) and (b).

By \eqref{eq:tlsComm} we have
\begin{align*}
&\Big(\pi\(\mm{a}{b}{c}{d},\kappa\)^{-1} f\Big)_\infty(x) = f\(\(\mm{a}{b}{c}{d},\kappa\) \tls^{-1} \tln_x \)\\
&= f\(\tls^{-1} \(\mm{d}{-c}{-b}{a},\kappa'\) \tln_x \) = \pi\(\mm{d}{-c}{-b}{a},\kappa'\)^{-1} (\pi(\tls)f)(\tln_x).
\end{align*}
We repeat the argument in the prior paragraph, except with $f$ replaced by $\pi(\tls) f$ and
with $\(\mm{a}{b}{c}{d},\kappa\)$ replaced by $\(\mm{d}{-c}{-b}{a}, \kappa'\)$.  
The conclusion of this repeated argument then yields parts (c) and (d) since $(\pi(\tls)f)_0 = f_\infty$.
\end{proof}

Define
\begin{equation}
\label{eq:tlGammaDef}
\tlGamma_1(4) = \left\{ \(\mm{a}{b}{c}{d},\(\frac{c}{d}\)\) \in \tlSL_2(\R): \begin{array}{l} a,b,c,d \in \Z \\ a \equiv d \equiv 1 (\mod 4) \\ c \equiv 0 (\mod 4) \end{array} \right\},
\end{equation}
where $\(\frac{\cdot}{\cdot}\)$ is the Kronecker symbol.  One can show
that $\tlGamma_1(4)$ is a well-defined subgroup of $\tlSL_2(\R)$.  Since we will be performing several computations involving the
Kronecker symbol, we state (without proof) some of its important properties. More information concerning the Kronecker symbol can be found
in \cite[\S 3.4.3]{Buchmann07}, \cite[\S 5]{Davenport00}.

\begin{prop}[Properties of the Kronecker Symbol]
\label{prop:KronProps}
Let $a,b,m,n \in \Z$.

\begin{itemize}[leftmargin=.3in,font=\normalfont\textbf]
\item[(a)] If $p$ is an odd prime then 
\begin{equation*}
\(\frac{a}{p}\) = \begin{cases}0 &\text{if } a \equiv 0(\mod p)\\ 1 &\text{if there exists } x\in\Z \text{ such that } x^2 \equiv a \not\equiv 0 (\mod p)\\ -1 &\text{if there are no solutions to } x^2 \equiv a (\mod p).\end{cases}
\end{equation*}

\item[(b)] $\displaystyle \(\frac{a}{2}\) = \begin{cases}0  &\text{if } a \equiv 0 (\mod 2)\\ 1 &\text{if } a \equiv 1, 7(\mod 8)\\ -1 &\text{if } a \equiv 3, 5(\mod 8). \end{cases}$

\item[(c)] $\displaystyle \(\frac{a}{1}\) = 1$,\h $\displaystyle \(\frac{a}{-1}\) = \begin{cases}-1 &\text{if } a<0\\ 1 &\text{if } a \geq 0,\end{cases}$ \h and \h $\displaystyle \(\frac{a}{0}\) = \begin{cases}1 &\text{if } a = \pm 1\\ 0 &\text{otherwise}.\end{cases}$

\item[(d)] If $n = \pm p_1^{e_1} \cdots p_k^{e_k}$ is the prime factorization of $n$ then 
\begin{equation*}
\(\frac{a}{n}\) = \(\frac{a}{\pm 1}\) \(\frac{a}{p_1}\)^{e_1} \cdots \(\frac{a}{p_k}\)^{e_k}.
\end{equation*}

\item[(e)] If $ab \neq 0$ then $\displaystyle \(\frac{a}{n}\) \(\frac{b}{n}\) = \(\frac{ab}{n}\)$.

\item[(f)] If $mn \neq 0$ then $\displaystyle \(\frac{a}{m}\) \(\frac{a}{n}\) = \(\frac{a}{mn}\)$.

\item[(g)] If $n>0$ then $\displaystyle \(\frac{a}{n}\) = \(\frac{b}{n}\)$ if $a \equiv b (\mod m)$ where
\begin{equation*}
m = \begin{cases}4n &\text{if } n \equiv 2 (\mod 4)\\ n &\text{otherwise}.\end{cases}
\end{equation*}

\item[(h)] If $a \neq 0$ and $a \equiv 0, 1(\mod 4)$ then $\displaystyle \(\frac{a}{m}\) = \(\frac{a}{n}\)$ if $m \equiv n (\mod a)$.

\item[(i)] If $m, n$ are positive odd integers then
\begin{equation*}
\(\frac{-1}{n}\) = (-1)^{(n-1)/2}, \h \(\frac{2}{n}\) = (-1)^{(n^2-1)/8}, \h \(\frac{m}{n}\) \(\frac{n}{m}\) = (-1)^{\frac{(m-1)(n-1)}{4}},
\end{equation*}
for $\gcd(m,n)=1$.

\end{itemize}
\end{prop}

Let $\delta_0$ denote the Dirac distribution on $\R$ centered at zero.
By \eqref{eq:sigmaIndRepUnbddIso}, the pair $(\delta_0,0)$ defines an element of
$V_{\epsilon,\nu}^{-\infty}$.
We denote this element of $V_{\epsilon,\nu}^{-\infty}$ by $\delta_{\epsilon,\nu,0}$.
Let $\delta_{\epsilon,\nu,\infty} = \pi(\tls) \delta_{\epsilon,\nu,0}$, with $\tls$ as defined in \eqref{eq:tlsdef}.
By \eqref{eq:TransLawsB},
\begin{align}
\label{eq:tlsdeltprop}
&\(\delta_{\epsilon,\nu,\infty}\)_0(x) = 0 \text{ and }\\
&\(\delta_{\epsilon,\nu,\infty}\)_\infty(x) = \delta_{\epsilon,\nu,0}(\tlm_{-1,1}\tln_x) = \sigma_\epsilon(\tlm_{-1,-1}) \delta_0(x) = \epsilon i \delta_0(x).\notag
\end{align}

\begin{lem}
\label{lem:deltaInfProp}
Let $\tlgamma^{-1} = \( \mm{a}{b}{c}{d}, \(\frac{c}{d}\) \) \in \tlGamma_1(4)$.

\begin{itemize}[leftmargin=.3in,font=\normalfont\textbf]

\item[(a)]
We have
\begin{equation*}
(\pi(\tlgamma)\delta_{\epsilon,\nu,\infty})_0 = \begin{cases} \(\frac{c}{d}\) |c|^{-\nu-1} \sgn(c)^{\epsilon/2} \delta_{\frac{-d}{c}} &\text{if } c \neq 0\\ 0 &\text{if } c=0, \end{cases}
\end{equation*}
as an equality between distributions on $\R$. 

\item[(b)] We have
\begin{equation*}
(\pi(\tlgamma) \delta_{\epsilon,\nu,\infty})_\infty = \begin{cases} \epsilon i (-c,d)_H \(\frac{c}{d}\)|d|^{-\nu-1} \sgn\(d\)^{\epsilon/2} \delta_{\frac{c}{d}} & \text{if } bc \neq 0\\
 \epsilon i \delta_0 & \text{if } b \neq 0, c=0\\
 \epsilon i \delta_{c} & \text{if } b = 0, \end{cases}
\end{equation*}
as an equality between distributions on $\R$.

\end{itemize}
\end{lem}

\begin{proof}
By Lemma \ref{lem:MetaUnbddAct}(a,b),
\begin{equation}
\label{eq:tlgmdelinfeq0}
(\pi(\tlgamma) \delta_{\epsilon,\nu,\infty})_0(x) = 0,
\end{equation}
as an equality between distributions on $\R_{\neq \frac{-d}{c}}$ when $c \neq 0$, and as an equality
between distributions on $\R$ when $c = 0$.
Thus for part (a) it remains to describe $(\pi(\tlgamma) \delta_{\epsilon,\nu,\infty})_0$ about the point $\frac{-d}{c}$ when $c \neq 0$.
To do this, first observe that $a, d \neq 0$ since $\tlgamma \in \tlGamma_1(4)$. Next, observe
\begin{equation}
\label{eq:tlgammatls}
\tlgamma \tls = \(\mm{-b}{-d}{a}{c},\(\frac{c}{d}\)(c,a)_H\) = \(\mm{c}{d}{-a}{-b},\(\frac{c}{d}\)(c,a)_H\)^{-1}
\end{equation}
(for $c \neq 0$). Thus by Lemma \ref{lem:MetaUnbddAct}(a), 
\begin{align}
\label{eq:pitlgmdeltinfcompl}
&(\pi(\tlgamma) \delta_{\epsilon,\nu,\infty})_0(x) = (\pi(\tlgamma \tls) \delta_{\epsilon,\nu,0})_0(x)\\
&= \(\frac{c}{d}\) (c,a)_H (a,-ax-b)_H |ax+b|^{\nu-1} \sgn(-ax-b)^{\epsilon/2} \delta_0\(\frac{cx+d}{-ax-b}\),\notag
\end{align}
as an equality between distributions on $\R_{\neq \frac{-b}{a}}$.

We can further simplify \eqref{eq:pitlgmdeltinfcompl}.  To do so, consider $\phi$ a test function on $\R_{\neq \frac{-b}{a}}$.
By performing various changes of  variables, we have
\begin{align*}
&\int_{\R_{\neq \frac{-b}{a}}} (\pi(\tlgamma) \delta_{\epsilon,\nu,\infty})_0(x) \phi(x)\, dx\\
&= \int_{\R_{\neq \frac{-b}{a}}} \(\frac{c}{d}\) (c,a)_H (a,-ax-b)_H |ax+b|^{\nu-1} \sgn(-ax-b)^{\epsilon/2}\notag\\
&\hhh \cdot \delta_0\(\frac{cx+d}{-ax-b}\) \phi(x) \, dx\notag\\
&= \int_{\R_{\neq 0}} \(\frac{c}{d}\) (c,a)_H (a,-ax)_H |ax|^{\nu-1} \sgn(-ax)^{\epsilon/2} \delta_0\(\frac{cx-\frac{bc}{a}+d}{-ax}\)\notag\\
&\hhh \cdot \phi\(x-\frac{b}{a}\) \, dx\notag\\
&= \int_{\R_{\neq 0}} \(\frac{c}{d}\) (c,a)_H (a,-ax)_H |ax|^{\nu-1} \sgn(-ax)^{\epsilon/2} \delta_0\( \frac{-c}{a} - \frac{1}{a^2 x} \)\notag\\
&\hhh \cdot \phi\(x-\frac{b}{a}\) \, dx\notag\\
&= \int_{\R_{\neq 0}} x^{-2} \(\frac{c}{d}\) (c,a)_H (a,ax^{-1})_H |ax^{-1}|^{\nu-1} \sgn(ax^{-1})^{\epsilon/2} \delta_0\( \frac{-c}{a} + \frac{x}{a^2} \) \notag\\
&\hhh \cdot \phi\(-x^{-1}-\frac{b}{a}\) \, dx\notag\\
&= \int_{\R_{\neq 0}} (ax)^{-2} \(\frac{c}{d}\) (c,a)_H (a,(ax)^{-1})_H |(ax)^{-1}|^{\nu-1} \sgn((ax)^{-1})^{\epsilon/2}\notag\\
&\hhh \cdot \delta_0\( \frac{-c}{a} + x\) \cdot \phi\(-(a^2 x)^{-1}-\frac{b}{a}\) \, dx\notag\\
&= c^{-2} \(\frac{c}{d}\) (c,a)_H (a,c^{-1})_H |c^{-1}|^{\nu-1} \sgn(c^{-1})^{\epsilon/2} \phi\(-\frac{1}{ac}-\frac{b}{a}\)\notag\\
&=  \(\frac{c}{d}\) |c|^{-\nu-1} \sgn(c)^{\epsilon/2} \phi\(-\frac{d}{c}\)\notag\\
&= \(\frac{c}{d}\) |c|^{-\nu-1} \sgn(c)^{\epsilon/2} \int_{\R_{\neq \frac{-b}{a}}} \delta_{\frac{-d}{c}}(x) \phi(x) \, dx.\notag
\end{align*}
In this last equality we have used the fact that $\frac{-b}{a} \neq \frac{-d}{c}$ (since $ad-bc=1$).
Thus for $c \neq 0$,
\begin{equation}
\label{eq:tlgammaDeltaInf}
(\pi(\tlgamma) \delta_{\epsilon,\nu,\infty})_0 = \(\frac{c}{d}\) |c|^{-\nu-1} \sgn(c)^{\epsilon/2} \delta_{\frac{-d}{c}},
\end{equation}
as an equality between distributions on $\R_{\neq \frac{-b}{a}}$.
Since $\frac{-b}{a} \neq \frac{-d}{c}$ then it follows from \eqref{eq:tlgmdelinfeq0} that $(\pi(\tlgamma) \delta_\infty)_0$
vanishes about the point $\frac{-b}{a}$.  Thus we conclude that \eqref{eq:tlgammaDeltaInf} holds as an equality between
distributions on $\R$.  This proves part (a).

In order to prove part (b), we follow the same approach we used for part (a).
If $c \neq 0$ then it follows from \eqref{eq:tlgammatls} and Lemma \ref{lem:MetaUnbddAct}(c) that
\begin{equation}
\label{eq:tlgmdelinfinfeq0}
(\pi(\tlgamma) \delta_{\epsilon,\nu,\infty})_\infty = (\pi(\tlgamma \tls) \delta_{\epsilon,\nu,0})_\infty = 0 \hh \text{(for $c \neq 0$)},
\end{equation}
as an equality between distributions on $\R_{\neq \frac{c}{d}}$.
If instead $c=0$, then we must have $a=d=1$ since $\tlgamma \in \tlGamma_1(4)$, and hence
\begin{equation*}
\tlgamma \tls = \(\mm{-b}{-1}{1}{0},1\) = \(\mm{0}{1}{-1}{-b},1\)^{-1}.
\end{equation*}
Therefore when $c=0$, it follows from this equality and Lemma \ref{lem:MetaUnbddAct}(c) that
\begin{equation}
\label{eq:tlgmdelinfinfeq0a}
(\pi(\tlgamma) \delta_{\epsilon,\nu,\infty})_\infty = (\pi(\tlgamma \tls) \delta_{\epsilon,\nu,0})_\infty = 0  \hh \text{(for $c = 0$)},
\end{equation}
as an equality between distributions on $\R_{\neq 0}$.
Thus it remains to describe $(\pi(\tlgamma) \delta_{\epsilon,\nu,\infty})_\infty$ about the point $\frac{c}{d}$, 
both for the case of $c=0$ and the case of $c \neq 0$.
We proceed by considering the following cases for $\tlgamma^{-1}$:
\begin{multicols}{3}
\begin{itemize}
\item[(i)] $bc \neq 0$,
\item[(ii)] $b \neq 0$ and $c=0$,
\item[(iii)] $b = 0$.
\end{itemize}
\end{multicols}

First we consider case (i) ($bc \neq 0$).
If $bc > 0$, then $ad = bc + 1 > 1$, which implies either $a,d > 0$ or $a,d < 0$.
If instead $bc < 0$, then $ad = bc + 1 < 1$, which implies either $a > 0$ and $d < 0$, or $a < 0$ and $d > 0$
(since $ad \neq 0$ because $\tlgamma \in \tlGamma_1(4)$).
By computing $\tls \tlgamma^{-1} \tls^{-1}$ for each of these possibilities, one can show that 
\begin{equation}
\label{eq:tlgmtlsbcneq0}
\tlgamma^{-1} \tls^{-1} = \tls^{-1} \(\mm{d}{-c}{-b}{a}, (a,d)_H \(\frac{c}{d}\)\).
\end{equation}
For case (ii) ($b \neq 0$ and $c = 0$), we find that \eqref{eq:tlgmtlsbcneq0} also holds (recall that if
$c = 0$ then $a = d = 1$ since $\tlgamma \in \tlGamma_1(4)$).
Thus in what follows, we consider cases (i) and (ii) simultaneously.

Notice that \eqref{eq:tlgmtlsbcneq0} is of the form \eqref{eq:tlsComm}.
Thus by Lemma \ref{lem:MetaUnbddAct}(c) and \eqref{eq:tlsdeltprop}
it follows that
\begin{align}
\label{eq:pitlgmdeltinfinfcompl}
&(\pi(\tlgamma) \delta_{\epsilon,\nu,\infty})_\infty(x)\\
&= (a,d)_H \(\frac{c}{d}\) (b, -bx +a)_H |-bx+a|^{\nu-1} \sgn(-bx+a)^{\epsilon/2} (\delta_{\epsilon,\nu,\infty})_\infty\(\frac{dx - c}{-bx + a}\)\notag\\
&= \epsilon i (a,d)_H \(\frac{c}{d}\) (b, -bx +a)_H |-bx+a|^{\nu-1} \sgn(-bx+a)^{\epsilon/2} \delta_0\(\frac{dx - c}{-bx + a}\),\notag
\end{align}
as an equality between distributions on $\R_{\neq \frac{a}{b}}$.
We simplify \eqref{eq:pitlgmdeltinfinfcompl} by integrating it against a test function $\phi$  on $\R_{\neq \frac{a}{b}}$
and then performing various changes of variables:
{\allowdisplaybreaks
\begin{align*}
&\int_{\R_{\neq \frac{a}{b}}} (\pi(\tlgamma) \delta_{\epsilon,\nu,\infty})_\infty(x) \phi(x) \, dx\\
&= \epsilon  i \int_{\R_{\neq \frac{a}{b}}} (a,d)_H \(\frac{c}{d}\) (b, -bx +a)_H |-bx+a|^{\nu-1} \sgn(-bx+a)^{\epsilon/2}\notag\\
&\hhh \cdot \delta_0\(\frac{dx - c}{-bx + a}\) \phi(x) \, dx\notag\\
&= \epsilon  i \int_{\R_{\neq 0}} (a,d)_H \(\frac{c}{d}\) (b, -bx)_H |-bx|^{\nu-1} \sgn(-bx)^{\epsilon/2} \delta_0\(\frac{dx + \frac{ad}{b} - c}{-bx}\)\notag\\
&\hhh \cdot \phi\(x + \frac{a}{b}\) \, dx\notag\\
&= \epsilon  i \int_{\R_{\neq 0}} (a,d)_H \(\frac{c}{d}\) (b, -bx)_H |bx|^{\nu-1} \sgn(-bx)^{\epsilon/2} \delta_0\(\frac{-d}{b}- \frac{1}{b^2 x}\)\notag\\
&\hhh \cdot \phi\(x + \frac{a}{b}\) \, dx\notag\\
&= \epsilon  i \int_{\R_{\neq 0}} (a,d)_H \(\frac{c}{d}\) \(b, \frac{-x}{b}\)_H |b|^{-\nu-1} |x|^{\nu-1} \sgn\(\frac{-x}{b}\)^{\epsilon/2} \delta_0\(\frac{-d}{b}- \frac{1}{x}\) \notag\\
&\hhh \cdot \phi\(\frac{x}{b^2} + \frac{a}{b}\) \, dx\notag\\
&= \epsilon  i \int_{\R_{\neq 0}} (a,d)_H \(\frac{c}{d}\) \(b, \frac{1}{bx}\)_H |b|^{-\nu-1} |x|^{-\nu-1} \sgn\(\frac{1}{bx}\)^{\epsilon/2} \delta_0\(\frac{-d}{b}+x\)\notag\\
&\hhh \cdot \phi\(\frac{-1}{b^2 x} + \frac{a}{b}\) \, dx\notag\\
&= \epsilon  i (a,d)_H \(\frac{c}{d}\) \(b, \frac{1}{d}\)_H |d|^{-\nu-1} \sgn\(\frac{1}{d}\)^{\epsilon/2} \phi\(\frac{-1}{bd} + \frac{a}{b}\)\notag\\
&= \epsilon  i (a,d)_H  \(b, d\)_H \(\frac{c}{d}\)|d|^{-\nu-1} \sgn\(d\)^{\epsilon/2} \int_{\R_{\neq \frac{a}{b}}} \delta_{\frac{c}{d}}(x) \phi(x) \, dx.\notag
\end{align*}
}%
In this last equality we once more use the fact that $\frac{a}{b} \neq \frac{c}{d}$. 
Thus for cases (i) and (ii) we have
\begin{equation}
\label{eq:tlgammaDeltaInf2}
(\pi(\tlgamma) \delta_{\epsilon,\nu,\infty})_\infty = \epsilon i (a,d)_H  \(b, d\)_H \(\frac{c}{d}\)|d|^{-\nu-1} \sgn\(d\)^{\epsilon/2} \delta_{\frac{c}{d}},
\end{equation}
as an equality between distributions on $\R_{\neq \frac{a}{b}}$.  Since $\frac{a}{b} \neq \frac{c}{d}$ it follows from 
\eqref{eq:tlgmdelinfinfeq0} (and \eqref{eq:tlgmdelinfinfeq0a}) that
$(\pi(\tlgamma) \delta_{\epsilon,\nu,\infty})_\infty$ vanishes about the point $\frac{a}{b}$.
Thus we conclude that \eqref{eq:tlgammaDeltaInf2}
holds as an equality between distributions on $\R$.

For case (ii), \eqref{eq:tlgammaDeltaInf2} simplifies to
\begin{equation}
\label{eq:tlgammaDeltaInf3}
(\pi(\tlgamma) \delta_{\epsilon,\nu,\infty})_\infty = \epsilon i \delta_{0},
\end{equation}
as an equality between distributions on $\R$. For case (i), we can also further simplify \eqref{eq:tlgammaDeltaInf2};
in particular, $(a,d)_H (b,d)_H$ can be written more concisely.
Observe that if $d>0$ then it follows that $(a,d)_H (b,d)_H$ $= (-c,d)_H$ since $(x,d)_H=1$ whenever $d>0$.
We wish to prove that this is also the case for $d<0$.  
To do so, recall that $(x,z)_H (y,z)_H = (xy,z)_H$ for $x,y,z \in \R_{\neq 0}$. 
Therefore 
\begin{equation}
\label{eq:hilId}
(a,d)_H (b,d)_H (-c,d)_H = (a,d)_H (-bc,d)_H.
\end{equation}
If $a > 0$ then $ad < 0$, which implies $-bc = 1 - ad > 0$, and hence $(a,d)_H (-bc,d)_H = 1$.
If $a < 0$ then $ad \geq 9$ since $a \equiv d \equiv 1 (\mod 4)$, which implies $-bc = 1 - ad < 0$, and hence $(a,d)_H (-bc,d)_H = 1$.
Thus by \eqref{eq:hilId}, $(a,d)_H (b,d)_H = (-c,d)_H$ (both for $d>0$ and $d<0$).
Hence for case (i),
\begin{equation}
\label{eq:tlgammaDeltaInf4}
(\pi(\tlgamma) \delta_{\epsilon,\nu,\infty})_\infty =  \epsilon i  (-c,d)_H \(\frac{c}{d}\)|d|^{-\nu-1} \sgn\(d\)^{\epsilon/2} \delta_{\frac{c}{d}}
\end{equation}
as an equality between distributions on $\R$.

For case (iii), we have
\begin{equation}
\label{eq:tlgmtlsbeq0}
\tlgamma^{-1} \tls^{-1} = \tls^{-1} \(\mm{1}{-c}{0}{1}, 1 \),
\end{equation}
since if $b = 0$ then $a = d = 1$ since $\tlgamma \in \tlGamma_1(4)$.
Notice that \eqref{eq:tlgmtlsbeq0} is also an equation of the form \eqref{eq:tlsComm}.
Thus by Lemma \ref{lem:MetaUnbddAct}(d) and \eqref{eq:tlsdeltprop},
\begin{equation}
\label{eq:tlgammaDeltaInf5}
(\pi(\tlgamma) \delta_{\epsilon,\nu,\infty})_\infty(x) = (\delta_{\epsilon,\nu,\infty})_\infty(x-c) = \epsilon i \delta_c(x).
\end{equation}
as an equality between distributions on $\R$.
Taken together, \eqref{eq:tlgammaDeltaInf3}, \eqref{eq:tlgammaDeltaInf4}, and \eqref{eq:tlgammaDeltaInf5} prove part (b).
\end{proof}

Let
\begin{equation}
\tlGamma_\infty = \left\{ \( \mm{1}{n}{0}{1}, 1 \): n \in \Z \right\}.
\end{equation}
For $\Re(\nu) > 1$, we define the \textit{metaplectic Eisenstein distribution at $\infty$} to be the following distribution in $V_{\epsilon,\nu}^{-\infty}$:
\begin{equation}
\label{eq:EisenInfDist}
\tlE^{(\infty)}_{\epsilon,\nu}(\tilde{g}) = \zeta_2(2\nu+1) \sum_{\tlgamma \in \tlGamma_1(4) / \tlGamma_\infty} \pi(\tlgamma) \delta_{\epsilon,\nu,\infty}(\tilde{g}),
\end{equation}
where 
\begin{equation}
\zeta_2(\nu) = (1-2^{-\nu}) \zeta(\nu) = \prod_{p \text{ odd prime}} (1-p^{-\nu})^{-1}.
\end{equation}
The summation over $\tlGamma_1(4) / \tlGamma_\infty$ is justified since by Lemma \ref{lem:deltaInfProp}
and \eqref{eq:tlsdeltprop}, we have that $\delta_{\epsilon,\nu,\infty}$ is 
$\tlGamma_\infty$-invariant under left inverse translation.
Thus $\tlE_{\epsilon,\nu}^{(\infty)}$ is at least formally $\tlGamma_1(4)$-invariant.
We shall justify the convergence of this series momentarily.

For $\tlgamma^{-1}=\(\mm{a}{b}{c}{d}, \(\frac{c}{d}\)\)$, observe
\begin{align}
&\tlgamma \cdot \(\mm{1}{n}{0}{1},1 \) = \(\mm{d}{-b}{-c}{a}, \(\frac{c}{d}\)\) \cdot \(\mm{1}{n}{0}{1},1 \)\\
&= \(\mm{d}{dn-b}{-c}{-cn+a},\(\frac{c}{d}\)\). \notag
\end{align}
Thus to each coset $\tlgamma \tlGamma_\infty$ of $\tlGamma_1(4) / \tlGamma_\infty$, there corresponds $(c,d) \in \Z^2$
such that $\gcd(c,d)=1$, $c \equiv 0(\mod 4)$, and $d \equiv 1 (\mod 4)$.  This correspondence is unique, for if
\begin{equation*}
\tlgamma' = \(\mm{d}{*}{-c}{*},*\) \in \tlGamma_1(4),
\end{equation*}
as must be the case if $\tlgamma' \tlGamma_\infty$ and $\tlgamma \tlGamma_\infty$ corresponded to the same $(c,d) \in \Z^2$,
then
\begin{equation*}
(\tlgamma')^{-1} \tlgamma = \(\mm{*}{*}{c}{d},*\) \(\mm{d}{*}{-c}{*},*\) = \(\mm{*}{*}{0}{*},*\).
\end{equation*}
Since the matrix coordinate of this element must have determinant $1$ and since we require $a \equiv d \equiv 1 (\mod 4)$,
it follows that $(\tlgamma')^{-1} \tlgamma \in \tlGamma_\infty$.
This shows that the correspondence of cosets of 
$\tlGamma / \tlGamma_\infty$  to such $(c,d) \in \Z^2$ is in fact unique.
Conversely, when given $(c,d) \in \Z^2$ such that $\gcd(c,d)=1$, $c \equiv 0(\mod 4)$,
and $d \equiv 1 (\mod 4)$, it follows that there exists $a,b \in \Z$ such that $ad - bc =1$.  Since $c \equiv 0 (\mod 4)$ and $d \equiv 1 (\mod 4)$
it follows that $a \equiv 1 (\mod 4)$.  Thus we are able to construct $\tlgamma$ which corresponds to such $(c,d)$.
Therefore
\begin{equation}
\tlGamma_1(4) / \tlGamma_\infty \cong \{(c,d) \in \Z^2: \gcd(c,d)=1, c \equiv 0 (\mod 4), d \equiv 1 (\mod 4)\}.
\end{equation}
Thus by Lemma \ref{lem:deltaInfProp},
\begin{equation}
\label{eq:EisenInfDistZero}
\(\tlE^{(\infty)}_{\epsilon,\nu}\)_0(x) = \zeta_2(2\nu+1) \sum_{\substack{(c,d)\in \Z_{\neq 0} \times \Z \\ \gcd(c,d)=1 \\ d \equiv 1 (\mod 4) \\ c \equiv 0 (\mod 4)}} \(\frac{c}{d}\) |c|^{-\nu - 1} \sgn(c)^{\epsilon/2} \delta_{-\frac{d}{c}}(x),
\end{equation}
and
\begin{align}
\label{eq:EisenInfDistInf}
&\(\tlE^{(\infty)}_{\epsilon,\nu}\)_\infty(x)\\
&= \epsilon i  \zeta_2(2\nu+1) \Bigg(\delta_0(x) + \sum_{\substack{(c,d)\in \Z_{\neq 0} \times \Z \\ \gcd(c,d)=1 \\ d \equiv 1 (\mod 4) \\ c \equiv 0 (\mod 4)}} (-c,d)_H \(\frac{c}{d}\)|d|^{-\nu-1} \sgn\(d\)^{\epsilon/2} \delta_{\frac{c}{d}}(x) \Bigg).\notag
\end{align}
Notice that for $\Re(\nu) > 1$, the integrals of \eqref{eq:EisenInfDistZero} and \eqref{eq:EisenInfDistInf} converge uniformly and absolutely
against compactly supported test functions on $\R$.  Therefore since
$\tlE^{(\infty)}_{\epsilon,\nu}$ is determined completely by $\(\tlE^{(\infty)}_{\epsilon,\nu}\)_0$ and $\(\tlE^{(\infty)}_{\epsilon,\nu}\)_\infty$, it follows that
$\tlE^{(\infty)}_{\epsilon,\nu}$ depends holomorphically on $\nu$ for $\Re(\nu) > 1$.  Furthermore, it follows
 that our series expansion for $\tlE^{(\infty)}_{\epsilon,\nu}$ converges
in the strong distribution topology.

\section{Gauss and Kloosterman Sums}
\label{sec:GaussKloo}

We will assume throughout this section that $\Re(\nu) > 1$.
Since $\(\tlE^{(\infty)}_{\epsilon,\nu}\)_0$ is periodic, it has a Fourier series expansion
\begin{equation}
\label{eq:tlEInfFourSeries}
\(\tlE^{(\infty)}_{\epsilon,\nu}\)_0(x) = \sum_{n \in \Z} a_{\epsilon,\nu}(n) e(n x),
\end{equation}
where
\begin{align}
&a_{\epsilon,\nu}(n) = \int^1_0 \(\tlE^{(\infty)}_{\epsilon,\nu}\)_0(x) e(-nx) \, dx \, \text{ and } e(x) = e^{2 \pi i x}.
\end{align}
By computing explicit formulas for $a_{\epsilon,\nu}(n)$, we will be able to show in Section \ref{sec:MeroContEisenInf} that 
$\tlE^{(\infty)}_{\epsilon,\nu}$ has holomorphic continuation to $\C$ except for simple poles on $\{0,\frac{1}{2}\}$.
By \eqref{eq:EisenInfDistZero}, observe
{\allowdisplaybreaks
\begin{align*}
&\frac{a_{\epsilon,\nu}(n)}{\zeta_2(2\nu+1)} = \sum_{\substack{(c,d)\in \Z_{\neq 0} \times \Z \\ \gcd(c,d)=1 \\ d \equiv 1 (\mod 4) \\ c \equiv 0 (\mod 4)}} \int^1_0 \(\frac{c}{d}\) |c|^{-\nu-1} \sgn(c)^{\epsilon/2} \delta_{-\frac{d}{c}}(x) e(-nx) \, dx\\
&= \sum_{\substack{(c,d) \in \Z_{>0}\times\Z \\ 0 \leq -d < c \\ d \equiv 1 (\mod 4) \\ c \equiv 0 (\mod 4)}} \(\frac{c}{d}\) |c|^{-\nu-1} e\(\frac{nd}{c}\) + \epsilon i \sum_{\substack{(c,d) \in \Z_{<0}\times\Z \\ 0 \geq -d > c \\ d \equiv 1 (\mod 4) \\ c \equiv 0 (\mod 4)}} \(\frac{c}{d}\) |c|^{-\nu-1} e\(\frac{nd}{c}\)\notag\\
&= \sum_{\substack{(c,d) \in \Z_{>0}\times\Z \\ 0 \leq -d < c \\ d \equiv 1 (\mod 4) \\ c \equiv 0 (\mod 4)}} \(\frac{c}{d}\) |c|^{-\nu-1} e\(\frac{nd}{c}\)
+ \epsilon i \sum_{\substack{(c,d) \in \Z_{>0}\times\Z \\ 0 \geq -d > -c \\ d \equiv 1 (\mod 4) \\ c \equiv 0 (\mod 4)}} \(\frac{-c}{d}\) |c|^{-\nu-1} e\(\frac{nd}{-c}\)\notag\\
&= \sum_{\substack{(c,d) \in \Z_{>0}\times\Z \\ 0 \leq d < c \\ d \equiv 3 (\mod 4) \\ c \equiv 0 (\mod 4)}} \(\frac{c}{-d}\) |c|^{-\nu-1} e\(\frac{-nd}{c}\)
+ \epsilon i \sum_{\substack{(c,d) \in \Z_{>0}\times\Z \\ 0 \leq d < c \\ d \equiv 1 (\mod 4) \\ c \equiv 0 (\mod 4)}} \(\frac{-c}{d}\) |c|^{-\nu-1} e\(\frac{nd}{-c}\).\notag
\end{align*}
}%
For $c>0$, we have $\(\frac{c}{d}\)=\(\frac{c}{-d}\)$, and for $d \equiv 1(\mod 4)$ we have $\(\frac{-c}{d}\) = \(\frac{c}{d}\)$.  Thus
\begin{align*}
&a_{\epsilon,\nu}(n) = \epsilon i \zeta_2(2\nu+1) \sum_{\substack{(c,d)\in \Z_{>0} \times \Z \\ 0 \leq d < c \\ c \equiv 0 (\mod 4)}} \Delta_d^{-\epsilon} \(\frac{c}{d}\) c^{-\nu-1} e\(-\frac{nd}{c}\)\\
&= \epsilon i \zeta_2(2\nu+1) \sum_{\substack{(c,d)\in \Z_{>0} \times \Z \\ 0 \leq d < 4c}} \Delta_d^{-\epsilon} \(\frac{4 c}{d}\) (4c)^{-\nu-1} e\(-\frac{nd}{4c}\),\notag
\end{align*}
where
\begin{equation}
\label{eq:DeltaDef}
\Delta_d = \begin{cases}1 &\text{if } d \equiv 1 (\mod 4)\\ i &\text{if } d \equiv 3 (\mod 4)\\ 0 &\text{otherwise}. \end{cases}
\end{equation}

Let
\begin{align}
\label{eq:KlooDef}
&K_\kappa(n;4c) = \sum_{d \in \Z/(4c)\Z} \Delta_d^{-\kappa} \(\frac{4c}{d}\) e\(\frac{nd}{4c}\),\\
\label{eq:GaussDef}
&G(n;c) = \sum_{x \in \Z/c\Z} \(\frac{x}{c}\)e\(\frac{nx}{c}\),
\end{align}
where $\kappa \equiv 1 (\mod 2)$.  The expressions in \eqref{eq:KlooDef} and \eqref{eq:GaussDef} are
referred to as \textit{Kloosterman Sums} and \textit{Gauss Sums} (respectively). Observe
\begin{equation*}
a_{\epsilon,\nu}(n) = \epsilon i 4^{-\nu-1} \zeta_2(2\nu+1) \sum_{c \in \Z_{>0}} c^{-\nu-1} K_{\epsilon}(-n;4c).
\end{equation*}
Since $\bar{K_\kappa(n;c)} = K_{-\kappa}(-n;c)$, it follows that
\begin{equation*}
a_{\epsilon,\nu}(n) = \epsilon i 4^{-\nu-1} \zeta_2(2\nu+1) \sum_{c \in \Z_{>0}} c^{-\nu-1} \scC_{-\epsilon} \( K_{-1}(\epsilon \, n;4c)\),
\end{equation*}
where
\begin{equation}
\label{eq:scCdef}
\scC_\epsilon(z) = \begin{cases}z &\text{if } \epsilon=1,\\ \bar{z} &\text{if } \epsilon=-1.\end{cases}
\end{equation}
For $c = 2^k c'$ where $\gcd(2,c')=1$ and $k \in \Z_{\geq 0}$, it follows from \cite[Lemma 2]{Iwaniec87} that
\begin{equation}
K_{-1}(n;4c)=K_{-1}(n;2^{k+2}c')=K_{-c'}(n\bar{c'};2^{k+2})G(n\bar{2^{k+2}};c'),
\end{equation}
where $\bar{c'}$ and $\bar{2^{k+2}}$ are integers such that 
\begin{equation}
\label{eq:WeirdBarCondition}
c' \bar{c'} \equiv 1(\mod 2^{k+2}) \text{ and } 2^{k+2}\bar{2^{k+2}} \equiv 1 (\mod c').
\end{equation} 
Thus 
\begin{align}
\label{eq:EisenInfFourCoeff}
&a_{\epsilon,\nu}(n) = \epsilon i 4^{-\nu-1} \zeta_2(2\nu+1) \sum_{c \in \Z_{>0}} c^{-\nu-1} \scC_{-\epsilon} \(K_{-1}(\epsilon n;4c)\)\\
&= \epsilon i 4^{-\nu-1} \zeta_2(2\nu+1) \sum_{\substack{c' \in \Z_{>0}\\c' \text{ odd}}} \sum_{k \in \Z_{\geq 0}} (2^k c')^{-\nu-1} \scC_{-\epsilon}\( K_{-c'}(\epsilon n \bar{c'};2^{k+2})G(\epsilon n\bar{2^{k+2}};c') \).\notag
\end{align}
In light of the comments following \eqref{eq:EisenInfDistInf}, we know that
this series expansion converges uniformly and absolutely for $\Re(\nu)>1$.

The following lemmas give explicit formulas for the Gauss and Kloosterman sums that occur in \eqref{eq:EisenInfFourCoeff}.
By utilizing these formulas, we will show in Section \ref{sec:MeroContFourCoeff}
that the Fourier coefficients $a_{\epsilon,\nu}(n)$ can be expressed in terms of Dirichlet $L$-functions.

\begin{lem} 
\label{lem:GaussSums}
Let $c>0$ and odd.
\begin{itemize}[leftmargin=.3in,font=\normalfont\textbf]
	
\item[(a)] If $\gcd(n,c)=1$ and $c$ is prime then $G(n;c) = \(\frac{n}{c}\)\Delta_c c^{1/2}$.

\item[(b)] If $c = p q$ and $\gcd(p,q)=1$ then $G(n;c) = (-1)^{\frac{(p-1)(q-1)}{4}} G(n;p) G(n;q)$.  Consequently,
$\Delta_c G(n;c)$ is multiplicative since $\Delta_p \Delta_q = (-1)^{\frac{(p-1)(q-1)}{4}} \Delta_{pq}$.

\item[(c)] If $c$ is square-free and $\gcd(n,c)=1$ then $G(n;c) = \(\frac{n}{c}\) \Delta_c c^{1/2}$.

\item[(d)] If $k \geq 2$, $p$ is prime, and $n = p^\ell n'$ where $\gcd(n',p)=1$ and $\ell \in \Z_{\geq 0}$, then
\begin{equation*}
G(n;p^k)=\begin{cases} 0 &\text{if } \ell < k-1\\ -p^{k-1} &\text{if } \ell = k-1, k \text{ even}\\ \(\frac{n'}{p}\) \Delta_p p^{k-1/2} &\text{if } \ell = k-1, k \text{ odd}\\ p^k - p^{k-1} & \text{if } \ell \geq k, k \text{ even}\\ 0 &\text{if } \ell \geq k, k \text{ odd}. \end{cases}
\end{equation*}

\item[(e)] If $c$ is not square-free then $G(1;c)=0$.

\end{itemize}
\end{lem}

The results in Lemma \ref{lem:GaussSums} are well-known and are proved using the usual techniques \cite{Berndt98,Ireland84,Iwaniec04}.
Let
\begin{equation}
\label{eq:chiDef}
\chi_4(c) = \begin{cases}1 &\text{if } c \equiv 1(\mod 4)\\ -1 &\text{if } c \equiv 3 (\mod 4)\\ 0 & \text{otherwise}\end{cases} \hh \text{and} \hh \chi_8(c) = \begin{cases}1 &\text{if } c \equiv 1(\mod 8)\\ 1 &\text{if } c \equiv 3 (\mod 8)\\ -1 &\text{if } c \equiv 5 (\mod 8)\\ -1 &\text{if } c \equiv 7 (\mod 8)\\ 0 & \text{otherwise}.\end{cases}
\end{equation}

\begin{lem}
\label{lem:KlooSums}
Let $c > 0$ and odd.
\begin{itemize}[leftmargin=.3in,font=\normalfont\textbf]
	
\item[(a)]  If $n \equiv 1 (\mod 4)$
\begin{align*}
&\hhh\hhh \Delta_c^{-1} K_{-c}(n \bar{c};4) = (1+i)\chi_4(c), \text{ where } \bar{c} \in \Z \text{ such that } c \bar{c} \equiv 1(\mod 4),\\
&\hhh\hhh \Delta_c^{-1} K_{-c}(n \bar{c};8) = 2^{3/2}(1+i)\chi_8(c), \text{ where } \bar{c} \in \Z \text{ such that } c \bar{c} \equiv 1(\mod 8),\\
&\hhh\hhh \Delta_c^{-1} K_{-c}(n \bar{c};2^{k+2}) = 0, \text{ where } \bar{c} \in \Z \text{ such that } c \bar{c} \equiv 1(\mod 2^{k+2}),
\end{align*}
for $k \geq 2$.

\item[(b)] If $n \equiv 3 (\mod 4)$ then
\begin{align*}
&\hhh\hh \Delta_c^{-1} K_{-c}(n \bar{c};4) = -(1+i)\chi_4(c), \text{ where } \bar{c} \in \Z \text{ such that } c \bar{c} \equiv 1(\mod 4),\\
&\hhh\hh \Delta_c^{-1} K_{-c}(n \bar{c};8) = 0, \text{ where } \bar{c} \in \Z \text{ such that } c \bar{c} \equiv 1(\mod 8),\\
&\hhh\hh \Delta_c^{-1} K_{-c}(n \bar{c};2^{k+2}) = 0, \text{ where } \bar{c} \in \Z \text{ such that } c \bar{c} \equiv 1(\mod 2^{k+2}),
\end{align*}
for $k \geq 2$.

\item[(c)] If $n=2^\ell n'$, $\ell>0$, and $\gcd(n',2)=1$, then
\begin{align*}
&\Delta_c^{-1} K_{-c}(n \bar{c};4) = \begin{cases}(1+i)\chi_4(c) &\text{if } \ell>1 \\ -(1+i)\chi_4(c) &\text{if } \ell=1, \end{cases}\\
&\hhh\hhh \text{ where } \bar{c} \in \Z \text{ such that } c \bar{c} \equiv 1(\mod 4),\\
&\Delta_c^{-1} K_{-c}(n \bar{c};8) = 0,\\
&\hhh\hhh \text{ where } \bar{c} \in \Z \text{ such that } c \bar{c} \equiv 1(\mod 8),
\end{align*}
and for $k \geq 2$ and $\bar{c} \in \Z$ such that $c \bar{c} \equiv 1(\mod 2^{k+2})$, we have:

\begin{itemize}[leftmargin=.25in,font=\normalfont\textbf]
\item[(1)] if $\ell \geq k+2$ then
\begin{equation*}
\Delta_c^{-1} K_{-c}(n \bar{c};2^{k+2}) = \begin{cases}(1+i) 2^k\chi_4(c) &\text{if } k \text{ even} \\ 0 &\text{if } k \text{ odd}, \end{cases}
\end{equation*}

\item[(2)] if $\ell = k+1$ then
\begin{equation*}
\Delta_c^{-1} K_{-c}(n \bar{c};2^{k+2}) = \begin{cases}-(1+i) 2^k\chi_4(c) &\text{if } k \text{ even} \\ 0 &\text{if } k \text{ odd}, \end{cases}
\end{equation*}

\item[(3)] if $\ell = k$ then
\begin{equation*}
\Delta_c^{-1} K_{-c}(n \bar{c};2^{k+2}) = \begin{cases} (1+i) 2^k\chi_4(n' c) & \text{if } k \text{ even} \\ 0 & \text{if } k \text{ odd}, \end{cases}
\end{equation*}

\item[(4)] if $\ell = k-1$ then
\begin{equation*}
\Delta_c^{-1} K_{-c}(n \bar{c};2^{k+2}) = \begin{cases} 0 & \text{if } \begin{array}{l}k \text{ even or }\\[-.15cm] n' \equiv 3(\mod 4)\end{array} \\  \frac{1+i}{\sqrt{2}} 2^{k+1} \chi_8(n'c') & \text{if } \begin{array}{l}k \text{ odd and }\\[-.15cm] n' \equiv 1 (\mod 4),\end{array} \end{cases}
\end{equation*}

\item[(5)] if $\ell \leq k-2$ then
\begin{equation*}
\Delta_c^{-1} K_{-c}(n \bar{c};2^{k+2}) = 0.
\end{equation*}
\end{itemize}

\end{itemize}
\end{lem}

\begin{proof}
Observe that for $k \geq 0$ we have
\begin{equation}
\Delta_{c}^{-1} K_{-c}(n \bar{c};2^{k+2}) = \Delta_{b}^{-1} K_{-b}(m \bar{b};2^{k+2})
\end{equation}
if $c \equiv b (\mod 2^{k+2})$ and $n \equiv m (\mod 2^{k+2})$.
Thus the evaluation of $\Delta_c^{-1} K_{-c}(n \bar{c};4)$ and $\Delta_c^{-1} K_{-c}(n \bar{c};8)$ in parts (a), (b), and (c) follows from 
a finite number of computations which are easy to perform. The evaluation of  $\Delta_c^{-1} K_{-c}(n \bar{c};2^{k+2})$ for $k \geq 2$ in
parts (a) and (b) will follow from our proof of part (c.5).

If $k \geq 1$ then
\begin{align*}
&K_{-c}(n\bar{c};2^{k+2})\\
&= \sum_{d (\mod 2^{k+2})} \Delta_d^{c} \(\frac{2^{k+2}}{d}\) e\(\frac{n\bar{c}d}{2^{k+2}}\) =\sum_{\substack{d (\mod 2^{k+2})\\d \text{ odd}}} \Delta_d^{c} \(\frac{2^{k+2}}{d}\) e\(\frac{n\bar{c}d}{2^{k+2}}\)\notag\\
&=\sum_{\substack{d (\mod 2^{k+2})\\ d \equiv 1(\mod 8)}} e\(\frac{n\bar{c}d}{2^{k+2}}\)+ (-1)^k i^{c} \sum_{\substack{d (\mod 2^{k+2})\\ d \equiv 3(\mod 8)}} e\(\frac{n\bar{c}d}{2^{k+2}}\)\notag\\
&\hh +(-1)^k\sum_{\substack{d (\mod 2^{k+2})\\ d \equiv 5(\mod 8)}} e\(\frac{n\bar{c}d}{2^{k+2}}\) + i^{c} \sum_{\substack{d (\mod 2^{k+2})\\ d \equiv 7(\mod 8)}} e\(\frac{n\bar{c}d}{2^{k+2}}\)\notag\\
&=\sum_{\substack{d (\mod 2^{k+2})\\ d \equiv 1(\mod 8)}} e\(\frac{n\bar{c}d}{2^{k+2}}\)+ (-1)^k i^{c} \sum_{\substack{d (\mod 2^{k+2})\\ d \equiv 1(\mod 8)}} e\(\frac{n\bar{c}(d+2)}{2^{k+2}}\)\notag\\
&\hh +(-1)^k\sum_{\substack{d (\mod 2^{k+2})\\ d \equiv 1(\mod 8)}} e\(\frac{n\bar{c}(d+4)}{2^{k+2}}\) + i^{c} \sum_{\substack{d (\mod 2^{k+2})\\ d \equiv 1(\mod 8)}} e\(\frac{n\bar{c}(d+6)}{2^{k+2}}\)\notag\\
&=\( 1 + (-1)^k i^{c} e\( \frac{2n \bar{c}}{2^{k+2}} \) + (-1)^k e\( \frac{4 n \bar{c}}{2^{k+2}} \) + i^{c} e\( \frac{6 n \bar{c}}{2^{k+2}} \) \) \notag\\
&\hhh \cdot \sum_{\substack{d (\mod 2^{k+2})\\ d \equiv 1(\mod 8)}} e\(\frac{n\bar{c}d}{2^{k+2}}\).\notag
\end{align*}
Since we can write $n = 2^\ell n'$ where $\gcd(2,n')=1$, this equation can also be written as
\begin{align}
\label{eq:KlooId}
&K_{-c}(n\bar{c};2^{k+2})\\
&=\(1 + (-1)^k i^c e\(\frac{2^{\ell+1}n'\bar{c}}{2^{k+2}}\) + (-1)^k e\(\frac{2^{\ell+2} n' \bar{c}}{2^{k+2}}\) + i^c e\(\frac{2^{\ell+1} \cdot 3 n' \bar{c}}{2^{k+2}}\) \) \notag \\
&\hh \cdot \sum_{\substack{d (\mod 2^{k+2})\\ d \equiv 1(\mod 8)}} e\(\frac{2^\ell n' \bar{c}d}{2^{k+2}}\).\notag
\end{align}
Along with this equation, we shall often use the identity $\Delta_{c}^{-1}(1+i^c) = (1+i) \chi_4(c)$ in
evaluating $K_{-c}(n\bar{c};2^{k+2})$.

Suppose that $k \geq 2$.
\begin{itemize}
	
\item If $\ell \geq k+2$ then by \eqref{eq:KlooId} we have
\begin{align*}
&\Delta_c^{-1} K_{-c}(n \bar{c};2^{k+2}) = \Delta_c^{-1} (1 + (-1)^k i^c + (-1)^k + i^c) 2^{k-1}\\
&=\begin{cases} \Delta_c^{-1} (1+i^{c}) 2^k & \text{if } k \text{ even} \\ 0 &\text{if } k \text{ odd} \end{cases} \h =\begin{cases} (1+i) 2^k \chi_4(c) &\text{if } k \text{ even} \\ 0 &\text{if } k \text{ odd}.\end{cases}\notag
\end{align*}
This proves part (c.1).

\item If $\ell = k+1$ then by \eqref{eq:KlooId} we have
\begin{align*}
&\Delta_c^{-1} K_{-c}(n\bar{c};2^{k+2})\\
&= \Delta_c^{-1} (1 + (-1)^k i^c + (-1)^k + i^c) \sum_{\substack{d (\mod 2^{k+2})\\ d \equiv 1(\mod 8)}} e\(\frac{n' \bar{c}d}{2}\)\notag\\
&=\begin{cases} -\Delta_c^{-1} (1+i^{c}) 2^k &\text{if } k \text{ even} \\ 0 &\text{if } k \text{ odd} \end{cases} \h =\begin{cases} -(1+i)  2^k \chi_4(c) &\text{if } k \text{ even} \\ 0 &\text{if } k \text{ odd}. \end{cases}\notag
\end{align*}
This proves part (c.2).

\item If $\ell = k$ then by \eqref{eq:KlooId} we have
\begin{align*}
&\Delta_c^{-1} K_{-c}(n\bar{c};2^{k+2})\\
&= \Delta_c^{-1} (1 - (-1)^k i^c + (-1)^k - i^c) \sum_{\substack{d (\mod 2^{k+2})\\ d \equiv 1(\mod 8)}} e\(\frac{n' \bar{c}d}{4}\).\notag
\end{align*}
Observe
\begin{equation*}
\Delta_c^{-1}(1-(-1)^k i^c + (-1)^k - i^c) = \begin{cases}2(1-i) &\text{if } k \text{ even}\\ 0 &\text{if } k \text{ odd}, \end{cases}
\end{equation*}
since $\Delta_c^{-1}(1-i^c) = 1-i$ regardless of the value of $c$.
Also observe
\begin{align*}
&\sum_{\substack{d (\mod 2^{k+2})\\ d \equiv 1(\mod 8)}} e\(\frac{n' \bar{c}d}{4}\)=\begin{cases}2^{k-1} i &\text{if } c \equiv n' \equiv 1 (\mod 4)\\ -2^{k-1} i &\text{if } -c \equiv n' \equiv 1 (\mod 4)\\ -2^{k-1} i &\text{if } c \equiv -n' \equiv 1 (\mod 4)\\ 2^{k-1} i &\text{if } c \equiv n' \equiv 3 (\mod 4)\end{cases}\\
&= i 2^{k-1} \chi_4(n'c).\notag
\end{align*}
Thus for $k$ even,
\begin{equation*}
\Delta_c^{-1} K_{-c}(n\bar{c};2^{k+2}) = (1+i) 2^k \chi_4(n'c).
\end{equation*}
This proves part (c.3)

\item If $\ell = k-1$ then by \eqref{eq:KlooId} we have
\begin{align*}
&K_{-c}(n\bar{c};2^{k+2})\\
&= \(1 + (-1)^k i^c e\(\frac{n'\bar{c}}{4}\) + (-1)^k e\(\frac{n'\bar{c}}{2}\) + i^{c}e\(\frac{3n'\bar{c}}{4}\)\)\notag\\
&\hhh \cdot \sum_{\substack{d(\mod 2^{k+2})\\d \equiv 1 (\mod 8)}} e\(\frac{n'\bar{c}d}{8}\)\notag\\
&= \(1 + (-1)^k i^c e\(\frac{n'\bar{c}}{4}\) + (-1)^k e\(\frac{n'\bar{c}}{2}\) + i^{c}e\(\frac{3n'\bar{c}}{4}\)\)\notag\\
&\hhh \cdot \sum_{\substack{d(\mod 2^{k+2})\\d \equiv 1 (\mod 8)}}e\(\frac{n'\bar{c}}{8}\)^d\notag\\
&= 2^{k-1} \(1 + (-1)^k i^c e\(\frac{n'\bar{c}}{4}\) + (-1)^k e\(\frac{n'\bar{c}}{2}\) + i^{c}e\(\frac{3n'\bar{c}}{4}\)\)\\
&\hhh \cdot e\(\frac{n'\bar{c}}{8}\)\notag\\
&= 2^{k-1} \(1 + (-1)^k i^{c+n'\bar{c}} - (-1)^k + i^{c+3n'\bar{c}}\) e\(\frac{n'\bar{c}}{8}\).\notag
\end{align*}
If $n'\bar{c} \equiv 1 (\mod 4)$ then
\begin{equation*}
i^{c+3n'\bar{c}} = i^{c+3} = -i^{c+1}=-i^{c+n'\bar{c}}.
\end{equation*}
If $n'\bar{c} \equiv 3 (\mod 4)$ then
\begin{equation*}
i^{c+3n'\bar{c}}=i^{c+1}=-i^{c+3}=-i^{c+n'\bar{c}}.
\end{equation*}
Therefore, regardless of the values of $n'$ and $c$, we have
\begin{align*}
&1+(-1)^k i^{c+n'\bar{c}}-(-1)^k+i^{c+3n'\bar{c}}\\
&= 1 + (-1)^ki^{c+n'\bar{c}}-(-1)^k-i^{c+n'\bar{c}}\notag\\
&= \begin{cases}0 &\text{if } k \text{ even}\\ 2(1-i^{c+n'\bar{c}}), &\text{if } k \text{ odd}.\end{cases}\notag
\end{align*}
Thus $K_{-c}(n\bar{c};2^{k+2})=0$ if $k$ is even, and if $k$ is odd then
\begin{equation*}
K_{-c}(n\bar{c};2^{k+2})=2^k (1-i^{c+n'\bar{c}})e\(\frac{n'\bar{c}}{8}\).
\end{equation*}
If $n' \equiv 3 (\mod 4)$ then $1-i^{c+n'\bar{c}}=0$, and hence $K_{-c}(n\bar{c};2^{k+2})=0$.
Thus it remains to evaluate $K_{-c}(n\bar{c};2^{k+2})$ for $n' \equiv 1 (\mod 4)$.
Towards this end, observe
\begin{equation*}
(1-i^{c+n'\bar{c}})e\(\frac{n'\bar{c}}{8}\) = (1-i^{c'+n''\bar{c'}})e\(\frac{n''\bar{c'}}{8}\)
\end{equation*}
if $c \equiv c' (\mod 8)$ and $n' \equiv n'' (\mod 8)$.
In light of this, we give the following table for values of 
$\Delta_c^{-1} K_{-c}(n \bar{c};2^{k+2}) = \Delta_c^{-1} 2^k (1-i^{c+n'\bar{c}})e\(\frac{n'\bar{c}}{8}\)$:
\begin{center}
\begin{tabular}{|c|c|c|c|c|}
\hline
$n'\bs c$ & $1$ & $3$ & $5$ & $7$ \\ \hline
$1$ & $\frac{1+i}{\sqrt{2}} 2^{k+1}$ & $\frac{1+i}{\sqrt{2}} 2^{k+1}$ & $-\frac{1+i}{\sqrt{2}} 2^{k+1}$ & $-\frac{1+i}{\sqrt{2}} 2^{k+1}$ \\ \hline
$5$ & $-\frac{1+i}{\sqrt{2}} 2^{k+1}$ & $-\frac{1+i}{\sqrt{2}} 2^{k+1}$  & $\frac{1+i}{\sqrt{2}} 2^{k+1}$ & $\frac{1+i}{\sqrt{2}} 2^{k+1}$ \\ \hline
\end{tabular}
\end{center}
From this we see
\begin{equation*}
\Delta_c^{-1}K_{-c}(n\bar{c};2^{k+2}) = \frac{1+i}{\sqrt{2}} 2^{k+1} \chi_8(n' c).
\end{equation*}
This proves part (c.4).

\item If $0 \leq \ell \leq k-2$ then
\begin{equation*}
\sum_{\substack{d (\mod 2^{k+2})\\ d \equiv 1 (\mod 8)}} e\(\frac{2^\ell n' \bar{c} d}{2^{k+2}}\) = \sum_{\substack{d (\mod 2^{k+2})\\ d \equiv 1 (\mod 8)}} e\(\frac{n' \bar{c}}{2^{k-\ell+2}}\)^d.
\end{equation*}
If $\xi = e\(\frac{1}{2^{k-\ell+2}}\)$ and $m$ is an odd integer then since $1 \leq k-\ell-1$, we have that $\xi^{8 m} \neq 1$.
Furthermore, since
\begin{equation*}
\xi^{8m}\(\sum_{\substack{d (2^{k+2})\\d \equiv 1 (\mod 8)}} \xi^{m d}\) = \sum_{\substack{d (\mod 2^{k+2})\\d \equiv 1 (\mod 8)}} \xi^{m(d+8)} = \sum_{\substack{d (\mod 2^{k+2})\\ d \equiv 1 (\mod 8)}} \xi^{m d},
\end{equation*}
it follows
\begin{equation*}
\sum_{\substack{d (\mod 2^{k+2})\\ d \equiv 1 (\mod 8)}} e\(\frac{2^{\ell}n' \bar{c}d}{2^{k+2}}\) = 0.
\end{equation*}
Hence by \eqref{eq:KlooId}, we have that $\Delta_c^{-1} K_{-c}(n'\bar{c'};2^{k+2}) = 0$.
This proves part (c.5).
Since we have allowed for $\ell = 0$, we have also proven the remaining statements of parts (a) and (b) for $k\geq 2$.\qedhere
\end{itemize}
\end{proof}

\section{The Meromorphic Continuation of Fourier Coefficients}
\label{sec:MeroContFourCoeff}

In this section, we further simplify the Fourier coefficients $a_{\epsilon,\nu}(n)$ (given in \eqref{eq:EisenInfFourCoeff}).
By doing so, we shall see that these Fourier coefficients have meromorphic continuation to $\C$.

Let
\begin{align}
\label{eq:scGdef}
\scG_r(\epsilon,n,2^k;\nu) &= \(1-\(\frac{\epsilon n}{2}\) 2^{-\nu-\frac{1}{2}}\) \prod_{\substack{p \text{ odd prime} \\ p|n}} \(1-p^{-2\nu-1}\)^{-1}\\
&\hhh \cdot \sum_{j \in \Z_{\geq 0}} \chi_r(p^j) \scC_{-\epsilon}\(\Delta_{p^j} G(\epsilon n \bar{2^k};p^j)\) (p^j)^{-\nu-1},\notag
\end{align}
and
\begin{equation}
L\(\nu,\(\frac{n}{\cdot}\)\) = \sum_{d \in \Z_{>0}} \(\frac{n}{d}\) d^{-\nu},
\end{equation}
where $r=4, 8$, $\scC_\epsilon$ is defined in \eqref{eq:scCdef}, 
$\chi_r$ is defined in \eqref{eq:chiDef},
and $\(\frac{\cdot}{\cdot}\)$ is the Kronecker symbol
(whose properties are stated in Proposition \ref{prop:KronProps}).
We assume throughout that $n \neq 0$. 
Recall that $L\(\nu,\(\frac{n}{\cdot}\)\)$ converges for $\Re(\nu)>1$.
We will establish the same domain of convergence for $\scG_4$ and $\scG_8$ momentarily.
In Proposition \ref{prop:SimpFourCoeff}, we will see that 
$a_{\epsilon,\nu}(n)$ can be expressed in terms of $\scG_4$, $\scG_8$, and $L\(\nu,\(\frac{n}{\cdot}\)\)$.
Also in Proposition \ref{prop:SimpFourCoeff}, we will show that
\begin{align}
\label{eq:ComplSer}
&\sum_{\substack{c' \in \Z_{>0} \\ c' \text{ odd}}} \chi_r(c') \scC_{-\epsilon}\(\Delta_{c'} G(\epsilon n \bar{2^k};c')\) (c')^{-\nu-1}
\end{align}
converges absolutely for $\Re(\nu)>1$, where either $r=4$ and $k$ is even, or $r=8$ and $k$ is odd.
But since we have not yet proved this to be the case, the following lemma is given as a conditional statement.

\begin{lem} If \eqref{eq:ComplSer} converges absolutely for $\Re(\nu) > 1$,
then $\scG_4$ and $\scG_8$ converge absolutely for $\Re(\nu)>1$, and
\label{lem:serToLFunc}
\begin{align*}
&\zeta_2(2\nu+1) \sum_{\substack{c' \in \Z_{>0} \\ c' odd}} \chi_r(c') \scC_{-\epsilon}\(\Delta_{c'} G(\epsilon n \bar{2^k};c')\) (c')^{-\nu-1}\\
&= L\(\nu+\frac{1}{2},\(\frac{\epsilon n}{\cdot}\)\) \scG_r(\epsilon, n, 2^k; \nu),
\end{align*}
for either $r=4$ and $k$ even, or $r=8$ and $k$ odd.
\end{lem}

\begin{proof}
Throughout the proof, suppose $\Re(\nu)>1$.
By Lemma \ref{lem:GaussSums}(b), it follows that $\chi_r(c') \scC_{-\epsilon}\(\Delta_{c'} G(\epsilon n \bar{2^k};c')\)$ is multiplicative.
Thus
\begin{align*}
&\zeta_2(2\nu+1) \sum_{\substack{c' \in \Z_{>0} \\ c' \text{ odd}}} \chi_r(c') \scC_{-\epsilon}\(\Delta_{c'} G(\epsilon n \bar{2^k};c')\) (c')^{-\nu-1} \\
&=\zeta_2(2\nu+1) \prod_{p \text{ odd prime}} \sum_{j \in \Z_{\geq 0}} \chi_r(p^j) \scC_{-\epsilon}\(\Delta_{p^j} G(\epsilon n \bar{2^k};p^j)\) (p^j)^{-\nu-1}.\notag
\end{align*}
By Lemma \ref{lem:GaussSums}(d), if $p \nmid n$ then $\chi_r(p^j) \scC_{-\epsilon}\(\Delta_{p^j} G(\epsilon n \bar{2^k};p^j)\) (p^j)^{-\nu-1} = 0$
for $j \geq 2$.
By Lemma \ref{lem:GaussSums}(a) and \eqref{eq:WeirdBarCondition}, if $p \nmid n$ then
\begin{equation*}
\scC_{-\epsilon}\(\Delta_{p} G(\epsilon n \bar{2^k};p)\) p^{-\nu-1} = \scC_{-\epsilon}\(\Delta_{p}^2 \(\frac{\epsilon n 2^k}{p}\)\) p^{-\nu-\frac{1}{2}}.
\end{equation*}
Therefore
\begin{align}
\label{eq:ComplSumStep}
&\zeta_2(2\nu+1) \sum_{\substack{c' \in \Z_{>0} \\ c' \text{ odd}}} \chi_r(c') \scC_{-\epsilon}\(\Delta_{c'} G(\epsilon n \bar{2^k};c')\) (c')^{-\nu-1} \\
&=\zeta_2(2\nu+1) \prod_{\substack{p \text{ odd prime}\\ p \nmid n}} \(1+ \chi_r(p) \scC_{-\epsilon}\(\Delta_p^2 \(\frac{\epsilon n 2^k}{p}\)\) p^{-\nu-\frac{1}{2}}\)\notag\\ 
&\hhh \cdot \prod_{\substack{p \text{ odd prime} \\ p|n}} \sum_{j \in \Z_{\geq 0}} \chi_r(p^j) \scC_{-\epsilon} \(\Delta_{p^j} G(\epsilon n \bar{2^k};p^j)\) (p^j)^{-\nu-1}.\notag
\end{align}
In the case where $r = 4$ and $k$ is even, it follows that
$\chi_4(p) \Delta_p^2 = 1$ for all odd primes $p$, and that $\(\frac{2^k}{p}\)=1$.
Thus by \eqref{eq:ComplSumStep},
\begin{align*}
&\zeta_2(2\nu+1) \sum_{\substack{c' \in \Z_{>0} \\ c' odd}} \chi_4(c') \scC_{-\epsilon}\(\Delta_{c'} G(\epsilon n \bar{2^k};c')\) (c')^{-\nu-1} \\
&=\zeta_2(2\nu+1) \prod_{\substack{p \text{ odd prime}\\ p \nmid n}} \(1+\(\frac{\epsilon n}{p}\)p^{-\nu-\frac{1}{2}}\)\notag\\
&\hhh \cdot \prod_{\substack{p \text{ odd prime} \\ p|n}} \sum_{j \in \Z_{\geq 0}} \chi_4(p^j) \scC_{-\epsilon}\(\Delta_{p^j} G(\epsilon n \bar{2^k};p^j)\) (p^j)^{-\nu-1}.\notag
\end{align*}
Since  for $p \nmid n$,
\begin{equation}
\label{eq:DirZetaId}
\(1+\(\frac{\epsilon n}{p}\)p^{-\nu-\frac{1}{2}}\) = \(1-p^{-2\nu-1}\) \cdot \(1-\(\frac{\epsilon n}{p}\)p^{-\nu-\frac{1}{2}}\)^{-1},
\end{equation}
and since 
\begin{equation}
\label{eq:zeta2EulerProd}
\zeta_2(2\nu+1) = \prod_{p \text{ odd prime}} \(1-p^{-2\nu-1}\)^{-1},
\end{equation}
it follows that
\begin{align*}
&\zeta_2(2\nu+1) \sum_{\substack{c' \in \Z_{>0} \\ c' odd}} \chi_4(c') \scC_{-\epsilon}\(\Delta_{c'} G(\epsilon n \bar{2^k};c')\) (c')^{-\nu-1} \\
&= \prod_{\substack{p \text{ odd prime}\\ p \mid n}} \(1-p^{-2\nu-1}\)^{-1} \prod_{\substack{p \text{ odd prime}\\ p \nmid n}} \(1-\(\frac{\epsilon n}{p}\)p^{-\nu-\frac{1}{2}}\)^{-1}\notag\\
&\hhh \cdot \prod_{\substack{p \text{ odd prime} \\ p|n}} \sum_{j \in \Z_{\geq 0}} \chi_4(p^j) \scC_{-\epsilon}\(\Delta_{p^j} G(\epsilon n \bar{2^k};p^j)\) (p^j)^{-\nu-1}\notag\\
&=L\(\nu+\frac{1}{2},\(\frac{\epsilon n}{\cdot}\)\) \(1- \(\frac{\epsilon n}{2}\)2^{-\nu-\frac{1}{2}}\) \prod_{\substack{p \text{ odd prime}\\ p \mid n}} \(1-p^{-2\nu-1}\)^{-1}\notag\\
&\hhh \cdot \prod_{\substack{p \text{ odd prime} \\ p|n}} \sum_{j \in \Z_{\geq 0}} \chi_4(p^j) \scC_{-\epsilon}\(\Delta_{p^j} G(\epsilon n \bar{2^k};p^j)\) (p^j)^{-\nu-1}.\notag
\end{align*}
This proves the formula stated in the lemma for $r=4$ and $k$ even.
We also see that if \eqref{eq:ComplSer} converges absolutely for $\Re(\nu)>1$ then $\scG_4(\epsilon, n, 2^k; \nu)$ converges
absolutely for $\Re(\nu)>1$ and $k$ even.

For the case of $r=8$ and $k$ odd, observe that $\chi_8(p) \Delta_p^2 \(\frac{2^k}{p}\) = \chi_8(p) \Delta_p^2 \(\frac{2}{p}\) = 1$ for all odd primes $p$;
this follows from the fact that $\(\frac{2^k}{p}\) = \(\frac{2}{p}\)$ since $k$ is odd, and the fact that
\begin{equation*}
\(\frac{2}{p}\) = \begin{cases}1 &\text{if } p \equiv 1,7 (\mod 8)\\ -1 &\text{if } p \equiv 3,5 (\mod 8).\end{cases}
\end{equation*}
Thus by \eqref{eq:ComplSumStep},
\begin{align*}
&\zeta_2(2\nu+1) \sum_{\substack{c' \in \Z_{>0} \\ c' \text{ odd}}} \chi_8(c') \scC_{-\epsilon}\(\Delta_{c'} G(\epsilon n \bar{2^k};c')\) (c')^{-\nu-1} \\
&=\zeta_2(2\nu+1) \prod_{\substack{p \text{ odd prime}\\ p \nmid n}} \(1+\(\frac{\epsilon n}{p}\)p^{-\nu-\frac{1}{2}}\)\notag\\
&\hhh \cdot \prod_{\substack{p \text{ odd prime} \\ p|n}} \sum_{j \in \Z_{\geq 0}} \chi_8(p^j) \scC_{-\epsilon}\( \Delta_{p^j} G(\epsilon n \bar{2^k};p^j) \) (p^j)^{-\nu-1}.\notag
\end{align*}
The portion of the proof for $r=4$ and $k$ even, subsequent to \eqref{eq:DirZetaId}, can be reapplied to
complete the proof for $r=8$ and $k$ odd.
\end{proof}

In \eqref{eq:EisenInfFourCoeff}, we showed for $\Re(\nu)>1$,
\begin{align*}
&a_{\epsilon,\nu}(n)\\
&= \epsilon i 4^{-\nu-1} \zeta_2(2\nu+1) \sum_{\substack{c' \in \Z_{>0}\\c' \text{ odd}}} \sum_{k \in \Z_{\geq 0}} (2^k c')^{-\nu-1} \scC_{-\epsilon}\( K_{-c'}(\epsilon n \bar{c'};2^{k+2})G(\epsilon n\bar{2^{k+2}};c') \),
\end{align*}
where $a_{\epsilon,\nu}(n)$ is the $n$-th Fourier coefficient of $\(\tlE_{\epsilon,\nu}^{(\infty)}\)_0$. 
Recall that this series converges absolutely.
By utilizing Lemmas \ref{lem:KlooSums} and \ref{lem:serToLFunc}, we are able to prove the following more detailed formulas for $a_{\epsilon,\nu}(n)$
when $n \neq 0$.

\begin{prop}
For $\Re(\nu)>1$, \eqref{eq:ComplSer} converges absolutely.
Furthermore, for $\Re(\nu)>1$, we have the following formulas for the Fourier coefficients $a_{\epsilon,\nu}(n)$ of $\(\tlE_{\epsilon,\nu}^{(\infty)}\)_0\mathrm{:}$
\label{prop:SimpFourCoeff}

\begin{itemize}[leftmargin=.3in,font=\normalfont\textbf]

\item[(a)]  If $\epsilon n \equiv 3 (\mod 4)$ then
\begin{equation*}
a_{\epsilon,\nu}(n) = -(1+\epsilon i) 4^{-\nu-1} L\(\nu+\frac{1}{2},\(\frac{\epsilon n}{\cdot}\)\) \scG_4(\epsilon, n,4;\nu).
\end{equation*}

\item[(b)] If $\epsilon n \equiv 1 (\mod 4)$ then
\begin{align*}
&a_{\epsilon,\nu}(n) = (1 + \epsilon i) 4^{-\nu-1} L\(\nu+\frac{1}{2},\(\frac{\epsilon n}{\cdot}\)\) \scG_4(\epsilon, n,4;\nu)\\
&\hhh + (1 + \epsilon i) 8^{-\nu-\frac{1}{2}} L\(\nu+\frac{1}{2},\(\frac{\epsilon n}{\cdot}\)\) \scG_8(\epsilon, n,8;\nu).
\end{align*}

\item[(c)] If $n=2 n'$ where $\gcd(n',2)=1$ then
\begin{equation*}
a_{\epsilon,\nu}(n) = - (1+\epsilon i) 4^{-\nu-1} L\(\nu+\frac{1}{2},\(\frac{\epsilon n}{\cdot}\)\) \scG_4(\epsilon, n,4;\nu).
\end{equation*}

\item[(d)] If $n=2^\ell n'$ where $\ell > 1$ and $\gcd(n',2)=1$, then
\begin{align*}
&a_{\epsilon,\nu}(n) = (1+\epsilon i) \sum_{\substack{k=0\\ k \text{ even}}}^{\ell-2} 2^{k} \(2^{k+2}\)^{-\nu-1} L\(\nu+\frac{1}{2},\(\frac{\epsilon n}{\cdot}\)\) \scG_4(\epsilon, n,2^{k+2};\nu)\\
&\hh - \delta_{\ell \equiv 1 (2)} (1+\epsilon i) 2^{\ell-1} \(2^{\ell+1}\)^{-\nu-1} L\(\nu+\frac{1}{2},\(\frac{\epsilon n}{\cdot}\)\) \scG_4(\epsilon, n,2^{\ell+1};\nu) \\
&\hh + \delta_{\ell \equiv 0 (2)} (1+\epsilon i) 2^\ell \(2^{\ell+2}\)^{-\nu-1} \chi_4(\epsilon n') L\(\nu+\frac{1}{2},\(\frac{\epsilon n}{\cdot}\)\) \scG_4(\epsilon, n,2^{\ell+2};\nu) \\
&\hh + \delta_{\ell \equiv 0 (2)} \delta_{n' \equiv 1 (4)} (1+\epsilon i) 2^{\ell+\frac{3}{2}} \(2^{\ell+3}\)^{-\nu-1} \chi_8(\epsilon n') L\(\nu + \frac{1}{2},\(\frac{\epsilon n}{\cdot}\)\)\\
&\hh\hh \cdot \scG_8(\epsilon, n,2^{\ell+3};\nu),
\end{align*}
where $\delta_{\ell \equiv j(m)} = \begin{cases}1 &\text{if } \ell \equiv j (\mod m),\\ 0 &\text{otherwise},\end{cases}$ \h for  $j, \ell, m \in \Z$.
\end{itemize}
\end{prop}

\begin{proof}
Suppose $\epsilon n \equiv 3 (\mod 4)$.
By \eqref{eq:EisenInfFourCoeff} and Lemma \ref{lem:KlooSums}(b),
\begin{align*}
a_{\epsilon,\nu}(n) &=\epsilon i 4^{-\nu-1} \zeta_2(2\nu+1) \sum_{\substack{c' \in \Z_{>0}\\c' \text{ odd}}} \sum_{k \in \Z_{\geq 0}} (2^k c')^{-\nu-1}\\
&\hh\hh \cdot \scC_{-\epsilon}\( \Delta_{c'}^{-1}  K_{-c'}(\epsilon n \bar{c'};2^{k+2}) \Delta_{c'} G(\epsilon n\bar{2^{k+2}};c')\)\notag\\
&= \epsilon i 4^{-\nu-1} \zeta_2(2\nu+1) \sum_{\substack{c' \in \Z_{>0} \\ c' odd}} (c')^{-\nu-1} \scC_{-\epsilon}\( \Delta_{c'}^{-1} K_{-c'}(\epsilon n \bar{c'};4) \Delta_{c'} G(\epsilon n \bar{4};c') \)\notag\\
&= \epsilon i 4^{-\nu-1} \zeta_2(2\nu+1) \sum_{\substack{c' \in \Z_{>0} \\ c' odd}} (c')^{-\nu-1} \scC_{-\epsilon}\((-1-i) \chi_4(c') \) \scC_{-\epsilon}\(\Delta_{c'} G(\epsilon n \bar{4};c')\)\notag\\
&= \epsilon i 4^{-\nu-1} \zeta_2(2\nu+1) \sum_{\substack{c' \in \Z_{>0} \\ c' odd}} (c')^{-\nu-1} (-1 + \epsilon i) \chi_4(c') \scC_{-\epsilon}\( \Delta_{c'} G(\epsilon n \bar{4};c') \)\notag\\
&= -(1+\epsilon i)4^{-\nu-1} \zeta_2(2\nu+1) \sum_{\substack{c' \in \Z_{>0} \\ c' odd}} \chi_4(c') \scC_{-\epsilon}\( \Delta_{c'} G(\epsilon n \bar{4};c') \) (c')^{-\nu-1}.
\end{align*}
We now see that $a_{\epsilon,\nu}(n)$ has a series of the form \eqref{eq:ComplSer}.
Since $a_{\epsilon,\nu}(n)$ converges absolutely for $\Re(\nu)>1$, this series of the form \eqref{eq:ComplSer}
must also converges absolutely for $\Re(\nu)>1$. Part (a) then follows from Lemma \ref{lem:serToLFunc}.

Suppose $\epsilon n \equiv 1 (\mod 4)$.
By \eqref{eq:EisenInfFourCoeff} and Lemma \ref{lem:KlooSums}(a),
{\allowdisplaybreaks
\begin{align*}
&a_{\epsilon,\nu}(n)\\
&= \epsilon i 4^{-\nu-1} \zeta_2(2\nu+1) \sum_{\substack{c' \in \Z_{>0}\\c' \text{ odd}}} (c')^{-\nu-1} \scC_{-\epsilon}\( \Delta_{c'}^{-1} K_{-c'}(\epsilon n \bar{c'};4) \Delta_{c'} G(\epsilon n\bar{4};c') \)\\*
&\hh + \epsilon i 4^{-\nu-1} \zeta_2(2\nu+1) \sum_{\substack{c' \in \Z_{>0}\\c' \text{ odd}}} (2c')^{-\nu-1} \scC_{-\epsilon}\(\Delta_{c'}^{-1} K_{-c'}(\epsilon n \bar{c'};8) \Delta_{c'} G(\epsilon n\bar{8};c') \)\\
&= \epsilon i 4^{-\nu-1} \zeta_2(2\nu+1) \sum_{\substack{c' \in \Z_{>0}\\c' \text{ odd}}} (c')^{-\nu-1} \scC_{-\epsilon}\((1+i) \chi_4(c')\) \scC_{-\epsilon}\(\Delta_{c'} G(\epsilon n\bar{4};c')\) \\*
&\hh + \epsilon i 4^{-\nu-1} \zeta_2(2\nu+1) \sum_{\substack{c' \in \Z_{>0}\\c' \text{ odd}}} (2c')^{-\nu-1} \scC_{-\epsilon}\(2^{\frac{3}{2}}(1+i)\chi_8(c')\) \scC_{-\epsilon}\(\Delta_{c'} G(\epsilon n\bar{8};c')\)\\
&= \epsilon i 4^{-\nu-1} \zeta_2(2\nu+1) \sum_{\substack{c' \in \Z_{>0}\\c' \text{ odd}}} (c')^{-\nu-1} (1-\epsilon i) \chi_4(c') \scC_{-\epsilon}\(\Delta_{c'} G(\epsilon n\bar{4};c')\) \\*
&\hh + \epsilon i 4^{-\nu-1} \zeta_2(2\nu+1) \sum_{\substack{c' \in \Z_{>0}\\c' \text{ odd}}} (2c')^{-\nu-1} 2^{\frac{3}{2}}(1-\epsilon i)\chi_8(c') \scC_{-\epsilon}\(\Delta_{c'} G(\epsilon n\bar{8};c')\)\\
&= (1+\epsilon i)4^{-\nu-1} \zeta_2(2\nu+1) \sum_{\substack{c' \in \Z_{>0} \\ c' odd}} \chi_4(c') \scC_{-\epsilon}\(\Delta_{c'} G(\epsilon n \bar{4};c')\) (c')^{-\nu-1} \\*
&\hh + (1+\epsilon i) 2^{\frac{3}{2}} 8^{-\nu-1} \zeta_2(2\nu+1) \sum_{\substack{c' \in \Z_{>0}\\c' \text{ odd}}} \chi_8(c') \scC_{-\epsilon}\(\Delta_{c'} G(\epsilon n\bar{8};c')\) (c')^{-\nu-1}.
\end{align*}
}%
Once again, we see that $a_{\epsilon,\nu}(n)$ contains series of the form \eqref{eq:ComplSer}.
Since $a_{\epsilon,\nu}(n)$ converges absolutely for $\Re(\nu)>1$, it follows then that these series of the form \eqref{eq:ComplSer}
must also converges absolutely for $\Re(\nu)>1$. Part (b) then follows from Lemma \ref{lem:serToLFunc}.

Suppose $n = 2 n'$ with $\gcd(n',2)=1$.
By \eqref{eq:EisenInfFourCoeff} and Lemma \ref{lem:KlooSums}(c),
\begin{align*}
a_{\epsilon,\nu}(n) &= \epsilon i 4^{-\nu-1} \zeta_2(2\nu+1) \sum_{\substack{c' \in \Z_{>0}\\c' \text{ odd}}} (c')^{-\nu-1} \scC_{-\epsilon}\(\Delta_{c'}^{-1} K_{-c'}(\epsilon n \bar{c'};4) \Delta_{c'} G(\epsilon n\bar{4};c') \)\\
&= \epsilon i 4^{-\nu-1} \zeta_2(2\nu+1) \sum_{\substack{c' \in \Z_{>0}\\c' \text{ odd}}} (c')^{-\nu-1} \scC_{-\epsilon}\((-1-i) \chi_4(c')\) \scC_{-\epsilon}\(\Delta_{c'} G(\epsilon n\bar{4};c')\)\\
&= -(1+\epsilon i) 4^{-\nu-1} \zeta_2(2\nu+1) \sum_{\substack{c' \in \Z_{>0}\\c' \text{ odd}}} \chi_4(c') \scC_{-\epsilon}\(\Delta_{c'} G(\epsilon n\bar{4};c')\) (c')^{-\nu-1}.
\end{align*}
As was the case with parts (a) and (b), $a_{\epsilon,\nu}(n)$ consists of a series of the form \eqref{eq:ComplSer}.
Since $a_{\epsilon,\nu}(n)$ converges absolutely for $\Re(\nu)>1$, it follows then that this series of the form \eqref{eq:ComplSer}
must also converges absolutely for $\Re(\nu)>1$. Part (c) then follows from Lemma \ref{lem:serToLFunc}.

Suppose $n = 2^\ell n'$ with $\gcd(n',2)=1$ and $\ell > 1$.
By \eqref{eq:EisenInfFourCoeff} and Lemma \ref{lem:KlooSums}(c),
{\small \allowdisplaybreaks
\begin{align*}
a_{\epsilon,\nu}(n) &= \epsilon i 4^{-\nu-1} \zeta_2(2\nu+1) \sum_{\substack{c' \in \Z_{>0}\\c' \text{ odd}}} \sum_{k \in \Z_{\geq 0}} (2^k c')^{-\nu-1}\\*
&\hh\hh \cdot \scC_{-\epsilon}\( \Delta_{c'}^{-1} K_{-c'}(\epsilon n \bar{c'};2^{k+2}) \Delta_{c'} G(\epsilon n \bar{2^{k+2}};c') \)\notag\\
&= \epsilon i 4^{-\nu-1} \zeta_2(2\nu+1) \sum_{\substack{c' \in \Z_{>0}\\c' \text{ odd}}} (c')^{-\nu -1} \scC_{-\epsilon}\((1+i) \chi_4(c')\) \scC_{-\epsilon}\(\Delta_{c'} G(\epsilon n\bar{4};c')\)\notag\\*
&\hh + \epsilon i 4^{-\nu-1} \zeta_2(2\nu+1) \sum_{\substack{c' \in \Z_{>0}\\c' \text{ odd}}} \sum_{\substack{k=2\\ k \text{ even}}}^{\ell-2} (2^k c')^{-\nu-1} \scC_{-\epsilon}\( (1+i) \chi_4(c') 2^k\)\notag\\*
&\hh\hh\cdot \scC_{-\epsilon}\(\Delta_{c'} G(\epsilon n\bar{2^{k+2}};c') \)\notag\\*
&\hh +\delta_{\ell \equiv 1 (2)} \epsilon i 4^{-\nu-1} \zeta_2(2\nu+1) \sum_{\substack{c' \in \Z_{>0}\\c' \text{ odd}}} (2^{\ell-1} c')^{-\nu-1} \scC_{-\epsilon}\( (-1-i)  \chi_4(c') 2^{\ell-1}\)\\*
&\hh\hh \cdot \scC_{-\epsilon}\(\Delta_{c'} G(\epsilon n\bar{2^{\ell+1}};c') \)\notag\\*
&\hh +\delta_{\ell \equiv 0 (2)} \epsilon i 4^{-\nu-1} \zeta_2(2\nu+1) \sum_{\substack{c' \in \Z_{>0}\\c' \text{ odd}}} (2^\ell c')^{-\nu-1} \scC_{-\epsilon}\( (1+i)  \chi_4(\epsilon n'c') 2^\ell\)\notag\\*
&\hh\hh \cdot \scC_{-\epsilon}\(\Delta_{c'} G(\epsilon n\bar{2^{\ell+2}};c') \)\notag\\*
&\hh +\delta_{\ell \equiv 0 (2)} \delta_{n' \equiv 1 (4)} \epsilon i 4^{-\nu-1} \zeta_2(2\nu+1) \sum_{\substack{c' \in \Z_{>0}\\c' \text{ odd}}} (2^{\ell+1} c')^{-\nu-1}\\*
&\hh\hh \cdot \scC_{-\epsilon}\( \frac{1+i}{\sqrt{2}} 2^{\ell+2} \chi_8(\epsilon n'c')\) \scC_{-\epsilon}\(\Delta_{c'} G(\epsilon n \bar{2^{\ell+3}};c') \)\notag\\
&= (1+\epsilon i) 4^{-\nu-1} \zeta_2(2\nu+1) \sum_{\substack{c' \in \Z_{>0}\\c' \text{ odd}}} \chi_4(c') \scC_{-\epsilon}\( \Delta_{c'} G(\epsilon n\bar{4};c')\) (c')^{-\nu-1}\\*
&\hh + (1+\epsilon i) \sum_{\substack{k=2\\ k \text{ even}}}^{\ell-2} 2^k  \(2^{k+2}\)^{-\nu-1} \zeta_2(2\nu+1)\\*
&\hh\hh \cdot \sum_{\substack{c' \in \Z_{>0}\\c' \text{ odd}}} \chi_4(c') \scC_{-\epsilon}\( \Delta_{c'} G(\epsilon n\bar{2^{k+2}};c') \)  (c')^{-\nu-1}\\*
&\hh - \delta_{\ell \equiv 1 (2)} (1+\epsilon i)  2^{\ell-1} \(2^{\ell+1}\)^{-\nu-1} \zeta_2(2\nu+1)\\*
&\hh\hh \cdot \sum_{\substack{c' \in \Z_{>0}\\c' \text{ odd}}}  \chi_4(c') \scC_{-\epsilon}\( \Delta_{c'} G(\epsilon n\bar{2^{\ell+1}};c') \) (c')^{-\nu-1}\\*
&\hh + \delta_{\ell \equiv 0 (2)} (1+\epsilon i) 2^\ell \(2^{\ell+2}\)^{-\nu-1} \chi_4(\epsilon n') \zeta_2(2\nu+1)\\*
&\hh\hh \cdot \sum_{\substack{c' \in \Z_{>0}\\c' \text{ odd}}} \chi_4(c') \scC_{-\epsilon}\( \Delta_{c'} G(\epsilon n\bar{2^{\ell+2}};c') \) (c')^{-\nu-1}\\*
&\hh + \delta_{\ell \equiv 0 (2)} \delta_{n' \equiv 1 (4)} \frac{1+\epsilon i}{\sqrt{2}}  2^{\ell+2} \(2^{\ell+3}\)^{-\nu-1} \chi_8(\epsilon n') \zeta_2(2\nu+1)\\*
&\hh\hh \cdot \sum_{\substack{c' \in \Z_{>0}\\c' \text{ odd}}} \chi_8(c') \scC_{-\epsilon}\( \Delta_{c'} G(\epsilon n\bar{2^{\ell+3}};c') \) (c')^{-\nu-1}.
\end{align*}}%
Notice that $a_{\epsilon,\nu}(n)$ has several series of the form \eqref{eq:ComplSer}.
Since $a_{\epsilon,\nu}(n)$ converges absolutely for $\Re(\nu)>1$, it follows then that these series of the form \eqref{eq:ComplSer}
must also converges absolutely for $\Re(\nu)>1$. Part (d) then follows from Lemma \ref{lem:serToLFunc}.
\end{proof}

For $n \neq 0$, we wish to show that $a_{\epsilon,\nu}(n)$ has a meromorphic continuation to $\C$.  
Proposition \ref{prop:SimpFourCoeff} does not show this as yet, since $\scG_r(\epsilon,n,2^{k+2};\nu)$,
for $r=4,8$, is not seen from its definition (given in \eqref{eq:scGdef}) to have meromorphic continuation.
In what follows, we will analyze the product $L\(\nu+\frac{1}{2},\(\frac{\epsilon n}{\cdot}\)\)\scG_r(\epsilon,n,2^{k+2};\nu)$
for even $k$ when $r = 4$ and for odd $k$ when $r=8$.  We will find that such products have meromorphic continuation, and therefore
by Proposition \ref{prop:SimpFourCoeff}, we will have that $a_{\epsilon,\nu}(n)$ has a meromorphic continuation to $\C$ for $n \neq 0$.

To begin, we write 
\begin{equation}
\label{eq:sandtDef}
\epsilon n = s t \text{ where } s = \prod_{\substack{p^e || n\\ e > 0 \text{ even}}} p^e \text{ and } t = \epsilon \, \sgn(n) \prod_{\substack{p^e || n\\ e \text{ odd}}} p^e.
\end{equation}
For $\Re(\nu)>1$,
\begin{align*}
&L\(\nu + \frac{1}{2}, \(\frac{t}{\cdot}\)\) = \prod_{p \text{ prime}} \(1 - \(\frac{t}{p}\) p^{-\nu-\frac{1}{2}}\)^{-1}\\
&= \prod_{\substack{p \text{ prime}\\ p \mid s}} \(1 - \(\frac{t}{p}\) p^{-\nu-\frac{1}{2}}\)^{-1} \prod_{\substack{p \text{ prime} \\ p \nmid s}} \(1 - \(\frac{t}{p}\) p^{-\nu-\frac{1}{2}}\)^{-1}\\
&= \prod_{\substack{p \text{ prime}\\ p \mid s}} \(1 - \(\frac{t}{p}\) p^{-\nu-\frac{1}{2}}\)^{-1} L\(\nu + \frac{1}{2}, \(\frac{\epsilon n}{\cdot}\)\),
\end{align*}
since $\(\frac{\epsilon n}{p}\) = 0$ for $p|s$ and $\(\frac{\epsilon n}{p}\) = \(\frac{s t}{p}\) = \(\frac{t}{p}\)$
for $p \nmid s$.
Thus
\begin{equation}
\label{eq:LNonPrimtoPrim}
L\(\nu + \frac{1}{2}, \(\frac{\epsilon n}{\cdot}\)\) =  \prod_{\substack{p \text{ prime}\\ p \mid s}} \(1 - \(\frac{t}{p}\) p^{-\nu-\frac{1}{2}}\) L\(\nu + \frac{1}{2}, \(\frac{t}{\cdot}\)\).
\end{equation}
Therefore by \eqref{eq:LNonPrimtoPrim} and \eqref{eq:scGdef},
\begin{align}
\label{eq:Lnu1}
&L\(\nu+\frac{1}{2},\(\frac{\epsilon n}{\cdot}\)\) \scG_r(\epsilon, n,2^k;\nu)\\
&= \scp_{2,\epsilon n} \cdot \(1-\(\frac{\epsilon n}{2}\) 2^{-\nu-\frac{1}{2}}\) L\(\nu + \frac{1}{2}, \(\frac{t}{\cdot}\)\) \( \prod_{\substack{p \text{ odd prime}\\ p \mid s}} \(1 - \(\frac{t}{p}\) p^{-\nu-\frac{1}{2}}\)  \)\notag\\
&\hh \cdot  \(\prod_{\substack{p \text{ odd prime} \\ p|n}} \(1-p^{-2\nu-1}\)^{-1} \sum_{j \in \Z_{\geq 0}} \chi_r(p^j) \scC_{-\epsilon}\(\Delta_{p^j} G(\epsilon n \bar{2^k};p^j)\) (p^j)^{-\nu-1}\),\notag
\end{align}
where $r=4,8$ and
\begin{equation}
\label{eq:scp2nDef}
\scp_{2,\epsilon n} = \begin{cases}\(1 - \(\frac{t}{2}\) 2^{-\nu-\frac{1}{2}}\) &\text{if } 2|s,\\ 1 &\text{otherwise}, \end{cases}
\end{equation}
with $s$ and $t$ as defined in \eqref{eq:sandtDef}.
Since 
\begin{equation}
\(1-\(\frac{t}{p}\) p^{-\nu -\frac{1}{2}} \) \(1 + \(\frac{t}{p}\) p^{-\nu-\frac{1}{2}} \) = 1 - p^{-2\nu - 1}
\end{equation}
it follows 
from \eqref{eq:Lnu1} that
{\allowdisplaybreaks
\begin{align}
\label{eq:LGr}
&L\(\nu+\frac{1}{2},\(\frac{\epsilon n}{\cdot}\)\) \scG_r(\epsilon, n,2^k;\nu)\\
&= \scp_{2,\epsilon n} \cdot  \(1-\(\frac{\epsilon n}{2}\) 2^{-\nu-\frac{1}{2}}\) L\(\nu + \frac{1}{2}, \(\frac{t}{\cdot}\)\) \notag\\
&\hh \cdot \( \prod_{\substack{p \text{ odd prime}\\ p \mid s}} \(\frac{1 - \(\frac{t}{p}\) p^{-\nu-\frac{1}{2}}}{1-p^{-2\nu-1}}\) \sum_{j \in \Z_{\geq 0}} \chi_r(p^j) \scC_{-\epsilon}\(\Delta_{p^j} G(\epsilon n \bar{2^k};p^j)\) (p^j)^{-\nu-1} \)\notag\\*
&\hh \cdot \( \prod_{\substack{p \text{ odd prime} \\ p|t}} \(1-p^{-2\nu-1}\)^{-1} \sum_{j \in \Z_{\geq 0}} \chi_r(p^j) \scC_{-\epsilon}\(\Delta_{p^j} G(\epsilon n \bar{2^k};p^j)\) (p^j)^{-\nu-1} \)\notag\\
&=  \scp_{2,\epsilon n} \cdot \(1-\(\frac{\epsilon n}{2}\) 2^{-\nu-\frac{1}{2}}\) L\(\nu + \frac{1}{2}, \(\frac{t}{\cdot}\)\)\notag\\*
&\hh \cdot \( \prod_{\substack{p \text{ odd prime}\\ p \mid s}} \(1 + \(\frac{t}{p}\) p^{-\nu-\frac{1}{2}}\)^{-1} \sum_{j \in \Z_{\geq 0}} \chi_r(p^j) \scC_{-\epsilon}\(\Delta_{p^j} G(\epsilon n \bar{2^k};p^j)\) (p^j)^{-\nu-1} \)\notag\\*
&\hh \cdot \( \prod_{\substack{p \text{ odd prime} \\ p|t}} \(1-p^{-2\nu-1}\)^{-1} \sum_{j \in \Z_{\geq 0}} \chi_r(p^j) \scC_{-\epsilon}\(\Delta_{p^j} G(\epsilon n \bar{2^k};p^j)\) (p^j)^{-\nu-1} \).\notag
\end{align}
}%
The following lemma allows us to further simplify \eqref{eq:LGr}.

\begin{lem} 
\label{lem:scGsimpl}
Suppose $\Re(\nu)>1$, $n \neq 0$, and $\epsilon = \pm 1$, with $\epsilon n = s t$ as in \eqref{eq:sandtDef}.
\begin{itemize}[leftmargin=.3in,font=\normalfont\textbf]
	
\item[(a)] If $p$ is an odd prime such that $p^e || t$, $e \in \Z_{>0}$, then
\begin{equation*}
\sum_{j \in \Z_{\geq 0}} \chi_r(p^j) \scC_{-\epsilon}\( \Delta_{p^j} G(\epsilon n \bar{2^k};p^j) \) (p^j)^{-\nu-1} = (1 - p^{-2\nu-1}) \sum_{\substack{j=0\\ j \text{ even}}}^{e-1} p^{-j \nu},
\end{equation*}
where $r = 4, 8$.

\item[(b)] Let $p$ be an odd prime such that $p^e || s$, $e \in \Z_{>0}$.
If $r=4$ and $k$ is even, or if $r=8$ and $k$ is odd, then
\begin{align*}
&\sum_{j \in \Z_{\geq 0}} \chi_r(p^j) \scC_{-\epsilon}\( \Delta_{p^j} G(\epsilon n \bar{2^k};p^j) \) (p^j)^{-\nu-1}\\
&= \(1 + \(\frac{t}{p}\) p^{-\nu -\frac{1}{2}}\) \( \(1 - \(\frac{t}{p}\) p^{-\nu -\frac{1}{2}}\) \sum_{\substack{j=0\\ j \text{ even}}}^{e-2} p^{-j \nu} + p^{-e \nu}\).
\end{align*}

\end{itemize}
\end{lem}

\begin{proof}
First we prove part (a). Since $p^e || t$, by \eqref{eq:sandtDef}, we know that $e$ is odd.
Also by \eqref{eq:sandtDef}, we have that $\epsilon n = p^e n''$ for some $n'' \in \Z$ such that $\gcd(n'',p)=1$.
Since $p$ is an odd prime and since $\bar{2^k}$ is an even integer, we also have that $\epsilon n \bar{2^k} = p^e n'$ for some $n' \in \Z$ 
such that $\gcd(n',p) = 1$.
Thus by Lemma \ref{lem:GaussSums}(d),
\begin{align*}
&\sum_{j \in \Z_{\geq 0}} \chi_r(p^j) \scC_{-\epsilon}\( \Delta_{p^j} G(\epsilon n \bar{2^k};p^j) \) (p^j)^{-\nu-1}\\
&=\sum_{\substack{j=0\\ j \text{ even}}}^{e-1} \chi_r(p^j) \scC_{-\epsilon}\( \Delta_{p^j} G(\epsilon n \bar{2^k};p^j) \) (p^j)^{-\nu-1}\\
&\hh + \chi_r(p^{e+1}) \scC_{-\epsilon}\( \Delta_{p^{e+1}} G(\epsilon n \bar{2^k};p^{e+1}) \) (p^{e+1})^{-\nu-1}\\
&=1 + \sum_{\substack{j=2\\ j \text{ even}}}^{e-1} (p^j - p^{j-1}) (p^j)^{-\nu-1} - p^{e} (p^{e+1})^{-\nu-1}.
\end{align*}
In the last equality, we also used the fact that $p^j \equiv 1 (\mod 8)$ for $j$ even, hence
$\chi_r(p^j)=1$ for $j$ even.
Further simplifications show that
\begin{align*}
&\sum_{j \in \Z_{\geq 0}} \chi_r(p^j) \scC_{-\epsilon}\( \Delta_{p^j} G(\epsilon n \bar{2^k};p^j) \) (p^j)^{-\nu-1}\\
&=1 + \sum_{\substack{j=2\\ j \text{ even}}}^{e-1} (p^{-j \nu} - p^{-j \nu-1}) - p^{-(e+1)\nu-1}\\
&=1 - p^{-2\nu-1} + \sum_{\substack{j=2\\ j \text{ even}}}^{e-3} (p^{-j \nu} - p^{-(j+2) \nu-1}) + p^{-(e-1) \nu} - p^{-(e+1)\nu-1}\\
&=(1 - p^{-2\nu-1}) + (1 - p^{-2\nu-1}) \sum_{\substack{j=2\\ j \text{ even}}}^{e-3} p^{-j \nu} + (1 - p^{-2\nu-1}) p^{-(e-1) \nu}\\
&=(1 - p^{-2\nu-1}) \sum_{\substack{j=0\\ j \text{ even}}}^{e-1} p^{-j \nu}.
\end{align*}
This proves part (a).

Next we prove part (b). Since $p^e || s$, by \eqref{eq:sandtDef}, we know that $e$ is even.
Thus, as we argued for in part (a), we have that $\epsilon n \bar{2^k}= p^e n'$ for some $n' \in \Z$ such that $\gcd(n',p)=1$.
Therefore by Lemma \ref{lem:GaussSums}(d),
\begin{align*}
&\sum_{j \in \Z_{\geq 0}} \chi_{r}(p^j) \scC_{-\epsilon}\(\Delta_{p^j} G(\epsilon n \bar{2^k};p^j)\) (p^j)^{-\nu-1}\\
&=\sum_{\substack{j=0\\ j \text{ even}}}^{e} \chi_{r}(p^j) \scC_{-\epsilon}\(\Delta_{p^j} G(\epsilon n \bar{2^k};p^j)\) (p^j)^{-\nu-1}\notag\\
&\hh + \chi_{r}(p^{e+1}) \scC_{-\epsilon}\( \Delta_{p^{e+1}} G(\epsilon n \bar{2^k};p^{e+1})\) (p^{e+1})^{-\nu-1}\notag\\
&=1 + \sum_{\substack{j=2\\ j \text{ even}}}^{e} (p^j - p^{j-1}) (p^j)^{-\nu-1}\notag\\
&\hh + \chi_{r}(p^{e+1}) \scC_{-\epsilon}\( \Delta_{p^{e+1}} \(\frac{2^k n'}{p}\) \Delta_p p^{e + \frac{1}{2}}\) (p^{e+1})^{-\nu-1}.\notag
\end{align*}
For $r=4$ and $k$ even, $\chi_r(p^{e+1}) \Delta_{p^{e+1}} \Delta_p = 1$ for all odd primes $p$ and $\(\frac{2^k}{p}\) = 1$.
For $r=8$ and $k$ odd, $\(\frac{2^k}{p}\) = \(\frac{2}{p}\)$ and $\chi_r(p^{e+1}) \Delta_{p^{e+1}} \(\frac{2}{p}\) \Delta_p = 1$ 
for all odd primes $p$.
Thus for $r=4$ and $k$ even, or $r=8$ and $k$ odd,
\begin{align*}
&\sum_{j \in \Z_{\geq 0}} \chi_{r}(p^j) \scC_{-\epsilon}\(\Delta_{p^j} G(\epsilon n \bar{2^k};p^j)\) (p^j)^{-\nu-1}\\
&=1 + \sum_{\substack{j=2\\ j \text{ even}}}^{e} (p^j - p^{j-1}) (p^j)^{-\nu-1} + \(\frac{n'}{p}\) p^{e + \frac{1}{2}} (p^{e+1})^{-\nu-1}\\
&=1 + \sum_{\substack{j=2\\ j \text{ even}}}^{e} (p^{-j \nu} - p^{-j \nu -1}) + \(\frac{n'}{p}\) p^{-(e+1)\nu -\frac{1}{2}}\\
&=1 - p^{-2 \nu -1} + \sum_{\substack{j=2\\ j \text{ even}}}^{e-2} (p^{-j \nu} - p^{-(j+2) \nu -1}) + p^{-e \nu} + \(\frac{n'}{p}\) p^{-(e+1)\nu -\frac{1}{2}}\\
&= (1 - p^{-2 \nu -1}) \sum_{\substack{j=0\\ j \text{ even}}}^{e-2} p^{-j \nu} + p^{-e \nu}\(1 + \(\frac{n'}{p}\) p^{-\nu -\frac{1}{2}}\)\\
&= \(1 + \(\frac{n'}{p}\) p^{-\nu -\frac{1}{2}}\) \( \(1 - \(\frac{n'}{p}\) p^{-\nu -\frac{1}{2}}\) \sum_{\substack{j=0\\ j \text{ even}}}^{e-2} p^{-j \nu} + p^{-e \nu}\).
\end{align*}
Part (b) now follows since $\(\frac{n'}{p}\) = \(\frac{t}{p}\)$.
\end{proof}

Let
\begin{align}
\label{eq:scLDef0}
&\mathcal{L}(\epsilon n, \nu) = L\(\nu+\frac{1}{2},\(\frac{\epsilon n}{\cdot}\)\) \scG_r(\epsilon, n,2^k;\nu),
\end{align}
where $r=4$ and $k$ is even, or $r=8$ and $k$ is odd.
Notice that in our notation for $\mathcal{L}(\epsilon n, \nu)$ we fail to account for $r$ and $k$.
This is justified since by Lemma \ref{lem:scGsimpl} and \eqref{eq:LGr} we have that
\begin{align}
\label{eq:scLDef}
&\mathcal{L}(\epsilon n, \nu)\\
&=  \scp_{2,\epsilon n} \cdot \(1-\(\frac{\epsilon n}{2}\) 2^{-\nu-\frac{1}{2}}\) L\(\nu + \frac{1}{2}, \(\frac{t}{\cdot}\)\) \prod_{\substack{p \text{ odd prime} \\ p^e||t,\, e>0}} \(\sum_{\substack{j=0\\ j \text{ even}}}^{e-1} p^{-j \nu} \) \notag\\*
&\hh \cdot \prod_{\substack{p \text{ odd prime}\\ p^e || s, \, e>0 }} \( \(1 - \(\frac{t}{p}\) p^{-\nu -\frac{1}{2}}\) \sum_{\substack{j=0\\ j \text{ even}}}^{e-2} p^{-j \nu} + p^{-e \nu} \),\notag
\end{align}
for $\Re(\nu)>1$
(recall that $\scp_{2,\epsilon n}$ is defined in \eqref{eq:scp2nDef} and $s$ and $t$ are defined in \eqref{eq:sandtDef}).
Observe that \eqref{eq:scLDef} is the product of a Dirichlet polynomial and the Dirichlet $L$-function 
$L\(\nu + \frac{1}{2}, \(\frac{t}{\cdot}\)\)$.  If $t \neq 1$ then $L\(\nu + \frac{1}{2}, \(\frac{t}{\cdot}\)\)$ has holomorphic continuation
to all of $\C$.  If $t=1$ then $L\(\nu + \frac{1}{2}, \(\frac{t}{\cdot}\)\)$ has holomorphic continuation to all of $\C$ except for a simple
pole $\nu = \frac{1}{2}$.  Therefore, by Proposition \ref{prop:SimpFourCoeff}, we see that for $n \neq 0$, $a_{\epsilon,\nu}(n)$ has holomorphic continuation
to all of $\C$ except for a simple pole at $\nu = \frac{1}{2}$ if $\epsilon n$ is a square. 
We record this result along with explicit formulas for $a_{\epsilon,\nu}(n)$ in the following theorem.

\begin{thm}
\label{thm:SimplFourCoeff}
The Fourier coefficients $a_{\epsilon,\nu}(n)$ for $n \neq 0$ have holomorphic continuation to $\C$ except
for a simple pole at $\nu = \frac{1}{2}$ when $\epsilon n$ is a square.
Furthermore, $a_{\epsilon,\nu}(n) = \scc_{\epsilon,\nu}(n) \mathcal{L}(\epsilon n; \nu)$
where $\scc_{\epsilon,\nu}(n)$ is defined accordingly:
\begin{itemize}
\item  If $\epsilon n \equiv 3 (\mod 4)$ or $n=2 n'$ where $\gcd(n',2)=1$ then
\begin{equation*}
\scc_{\epsilon,\nu}(n) = -(1+\epsilon i) 4^{-\nu-1}.
\end{equation*}
\item If $\epsilon n \equiv 1 (\mod 4)$ then
\begin{equation*}
\scc_{\epsilon,\nu}(n) = (1 + \epsilon i) \(4^{-\nu-1}+8^{-\nu-\frac{1}{2}}\).
\end{equation*}
\item If $n=2^\ell n'$ where $\ell > 1$ and $\gcd(n',2)=1$, then
\begin{align*}
\scc_{\epsilon,\nu}(n) &= (1+\epsilon i) \Bigg( \sum_{\substack{k=0\\ k \text{ even}}}^{\ell-2} 2^{k} \(2^{k+2}\)^{-\nu-1}
- \delta_{\ell \equiv 1 (2)} 2^{\ell-1} \(2^{\ell+1}\)^{-\nu-1}\\
&+ \delta_{\ell \equiv 0 (2)} \chi_4(\epsilon n') 2^\ell \(2^{\ell+2}\)^{-\nu-1} \\
&+ \delta_{\ell \equiv 0 (2)} \delta_{n' \equiv 1 (4)} \chi_8(\epsilon n') 2^{\ell+\frac{3}{2}} \(2^{\ell+3}\)^{-\nu-1}
\Bigg)
\end{align*}
where $\delta_{\ell \equiv j(m)} = \begin{cases}1 &\text{if } \ell \equiv j (\mod m),\\ 0 &\text{otherwise},\end{cases}$ for  $j, \ell, m \in \Z$.
\end{itemize}
\end{thm}

In the following proposition we obtain a formula for the Fourier coefficient $a_{\epsilon,\nu}(0)$.
This formula shows that $a_{\epsilon,\nu}(0)$ also has a holomorphic continuation to all of $\C$
except for a pole at $\nu = \frac{1}{2}$.  This formula for $a_{\epsilon,\nu}(0)$ will play a vital role in establishing 
a functional equation for metaplectic Eisenstein distributions in Section \ref{sec:EisenDistFunctionalEq}.

\begin{prop}
\label{prop:EisenInfConst}
For $\Re(\nu) > 1$,
\begin{equation*}
a_{\epsilon,\nu}(0) = (1+\epsilon i) 2^{-2\nu-2} \zeta(2\nu).
\end{equation*}
\end{prop}

\begin{proof}
By \eqref{eq:EisenInfFourCoeff} and \eqref{eq:KlooDef},
\begin{align*}
&a_{\epsilon,\nu}(0) = \epsilon i 4^{-\nu-1} \zeta_2(2\nu+1) \sum_{c \in \Z_{>0}} c^{-\nu-1} \scC_{-\epsilon}\(K_{-1}(0;4c)\) \\
&= \epsilon i 4^{-\nu-1} \zeta_2(2\nu+1) \sum_{c \in \Z_{>0}} c^{-\nu-1} \scC_{-\epsilon}\(\sum_{d \in \Z/4c\Z} \Delta_d \(\frac{4c}{d}\) \).
\end{align*}
If $c$ is a square then
\begin{equation*}
\sum_{d \in \Z/4c\Z} \Delta_d \(\frac{4c}{d}\) =  \sum_{\substack{d \in \Z/4c\Z\\d \text{ odd} \\\gcd(c,d)=1}} \Delta_d = \frac{1+i}{2} \phi(4c),
\end{equation*}
where $\phi$ is the Euler totient function. If $c$ is not a square then we can select $d'' \in (\Z / 4 c \Z)^*$ such that $\(\frac{4c}{d''}\)=-1$.
This follows from the law of quadratic reciprocity (Proposition \ref{prop:KronProps}) and \cite[Theorem 5.2.3]{Ireland84}.
Notice that $d'' \equiv \pm 1 (\mod 4)$.
If $d'' \equiv -1 (\mod 4)$ then $\(\frac{4c}{-d''}\) = -1$ since $\(\frac{4c}{-1}\) = 1$.
Thus we can always find $d' \in (\Z / 4 c \Z)^*$ such that $d' \equiv 1 (\mod 4)$ and $\(\frac{4c}{d'}\)=-1$.
Since
\begin{equation*}
\(\frac{4c}{d'}\) \sum_{\substack{d (\mod 4c)\\d \equiv 1 (4)}} \(\frac{4c}{d}\) = \sum_{\substack{d (\mod 4c)\\d \equiv 1 (4)}} \(\frac{4c}{d}\) \text{ it follows that } \sum_{\substack{d (\mod 4c)\\d \equiv 1 (4)}} \(\frac{4c}{d}\)=0.
\end{equation*}
By the same argument, it also follows that
\begin{equation*}
\sum_{d (\mod 4c)} \(\frac{4c}{d}\)=0, \text{ and thus } \sum_{\substack{d (\mod 4c)\\d \equiv 3 (4)}} \(\frac{4c}{d}\)=0.
\end{equation*}
Therefore if $c$ is not a square then
\begin{equation*}
\sum_{d \in \Z/4c\Z} \Delta_d \(\frac{4c}{d}\) = 0.
\end{equation*}
Thus
{\allowdisplaybreaks
\begin{align*}
&a_{\epsilon,\nu}(0) = \epsilon i 4^{-\nu-1} \zeta_2(2\nu+1) \sum_{\substack{c \in \Z_{>0}\\ c \text{ a square}}} c^{-\nu-1} \frac{1-\epsilon i}{2} \phi(4c)\\
&= \frac{1+\epsilon i}{2} 4^{-\nu-1} \zeta_2(2\nu+1) \sum_{c \in \Z_{>0}} c^{-2\nu-2} \phi(4c^2)\\
&= \frac{1+\epsilon i}{2} 4^{-\nu-1} \zeta_2(2\nu+1) \(\sum_{\substack{c \in \Z_{>0}\\c \text{ odd}}} \sum_{k \in \Z_{\geq 0}}(2^k c)^{-2\nu-2} \phi(2^{2k+2} c^2) \)\\
&= \frac{1+\epsilon i}{2} 4^{-\nu-1} \zeta_2(2\nu+1) \(\sum_{\substack{c \in \Z_{>0}\\c \text{ odd}}} \sum_{k \in \Z_{\geq 0}}(2^k)^{-2\nu-2} c^{-2\nu-2}  2^{2k+1} \phi(c^2) \)\\
&= \frac{1+\epsilon i}{2} 4^{-\nu-1} \zeta_2(2\nu+1) \(\sum_{\substack{c \in \Z_{>0}\\c \text{ odd}}} 2 \sum_{k \in \Z_{\geq 0}} 2^{-2 k \nu} c^{-2\nu-2} \phi(c^2) \)\\
&= (1+\epsilon i) 2^{-2\nu-2} (1-2^{-2\nu})^{-1} \zeta_2(2\nu+1) \(\sum_{\substack{c \in \Z_{>0}\\c \text{ odd}}} c^{-2\nu-2} \phi(c^2) \).
\end{align*}}%
Since $\phi(c^2)$ is multiplicative,
\begin{align*}
a_{\epsilon,\nu}(0) &= (1+\epsilon i) 2^{-2\nu - 2} \(1 - 2^{-2\nu}\)^{-1} \zeta_2(2\nu+1) \prod_{p \text{ odd prime}} \sum_{k \in \Z_{\geq 0}} (p^k)^{-2\nu-2} \phi(p^{2k}).
\end{align*}
Observe that for odd primes $p$,
\begin{align}
\label{eq:eulerPhiSum}
&\sum_{k \in \Z_{\geq 0}} (p^k)^{-2\nu-2} \phi(p^{2k}) = 1 + \sum_{k \in \Z_{\geq 1}} p^{-2k\nu-2k} (p^{2k} - p^{2k-1})
= \frac{1 - p^{-2\nu-1}}{1-p^{-2\nu}}.
\end{align}
Using this expression and the Euler product expansion for $\zeta_2(2\nu+1)$, we find that
\begin{align*}
&a_{\epsilon,\nu}(0) = (1+\epsilon i) 2^{-2\nu-2} (1-2^{-2\nu})^{-1} \zeta_2(2\nu+1) \prod_{p \text{ odd prime}} \frac{1 - p^{-2\nu-1}}{1-p^{-2\nu}}\\
&= (1+\epsilon i) 2^{-2\nu-2} (1-2^{-2\nu})^{-1} \prod_{p \text{ odd prime}} (1-p^{-2\nu})^{-1}\\
&= (1+\epsilon i) 2^{-2\nu-2} \zeta(2\nu).\qedhere
\end{align*}
\end{proof}

\section{The Meromorphic Continuation of \texorpdfstring{$\tlE_{\epsilon,\nu}^{(\infty)}$}{the Metaplectic Eisenstein Distribution at Infinity}}
\label{sec:MeroContEisenInf}

Now that we have established the meromorphic continuation of the Fourier coefficients $a_{\epsilon,\nu}(n)$, it remains to show that
$\tlE_{\epsilon,\nu}^{(\infty)}$ also has a meromorphic continuation.  
The first step to accomplishing this is to establish some bounds for $|a_{\epsilon,\nu}(n)|$.
We begin by finding bounds for $|\mathcal{L}(\epsilon n, \nu)|$, which is defined in \eqref{eq:scLDef}.

{\allowdisplaybreaks
For $\Re(\nu) \geq 0$, we have the following (crude) bounds (recall that $\scp_{2,\epsilon n}$ was defined in \eqref{eq:scp2nDef},
 and $s$ and $t$ are defined in \eqref{eq:sandtDef}):
\begin{align}
&\left|\scp_{2,\epsilon n} \cdot \(1-\(\frac{\epsilon n}{2}\) 2^{-\nu-\frac{1}{2}}\)\right| \leq \(1+ 2^{-\Re(\nu)-\frac{1}{2}}\)^2 \leq \(1+ 2^{-\frac{1}{2}}\)^2 < 4,\\[.2in]
&\left|\prod_{\substack{p \text{ odd prime} \\ p^e || t, \, e >0}}  \sum_{\substack{j=0\\ j \text{ even}}}^{e-1} p^{-j \nu}\right| \leq \prod_{\substack{p \text{ odd prime} \\ p^e || t, \, e>0}}  \sum_{\substack{j=0\\ j \text{ even}}}^{e-1} 1 < \prod_{\substack{p \text{ odd prime} \\ p^e || t, \, e > 0}} p^{e} \leq |n|,\\[.2in]
&\left| \prod_{\substack{p \text{ odd prime}\\ p^e || s, \, e > 0}}  \( \(1 - \(\frac{t}{p}\) p^{-\nu -\frac{1}{2}}\) \sum_{\substack{j=0\\ j \text{ even}}}^{e-2} p^{-j \nu} + p^{-e \nu}\) \right|\\
&= \left| \prod_{\substack{p \text{ odd prime}\\ p^e || s, \, e > 0}}  \( \sum_{\substack{j=0\\ j \text{ even}}}^{e} p^{-j \nu}  - \(\frac{t}{p}\) p^{-\nu -\frac{1}{2}} \sum_{\substack{j=0\\ j \text{ even}}}^{e-2} p^{-j \nu}\)\right|\notag\\
&\leq \prod_{\substack{p \text{ odd prime}\\ p^e || s, \, e > 0}}\( \left| \sum_{\substack{j=0\\ j \text{ even}}}^{e} p^{-j \nu} \right| +  \left| \(\frac{t}{p}\) p^{-\nu -\frac{1}{2}} \sum_{\substack{j=0\\ j \text{ even}}}^{e-2} p^{-j \nu}\right|\)\notag\\
&\leq \prod_{\substack{p \text{ odd prime}\\ p^e || s, \, e > 0}} \( \sum_{\substack{j=0\\ j \text{ even}}}^{e} 1 +  p^{-\frac{1}{2}} \sum_{\substack{j=0\\ j \text{ even}}}^{e-2} 1 \)\notag\\
&< \prod_{\substack{p \text{ odd prime}\\ p^e || s, \, e > 0}}\(p^{e} + p^e\) < \prod_{\substack{p \text{ odd prime}\\ p^e || s, \, e > 0}} p^{e + 1} 
\leq |n|^2.\notag
\end{align}
For $\Re(\nu) \leq 0$, we have the following bounds:
\begin{align}
&\left|\scp_{2,\epsilon n} \cdot \(1-\(\frac{\epsilon n}{2}\) 2^{-\nu-\frac{1}{2}}\)\right| \leq \(1+ 2^{-\Re(\nu)-\frac{1}{2}}\)^2\\
&< \(2 \cdot 2^{-\Re(\nu)}\)^2 = 4^{-\Re(\nu)+1}, \notag\\[.2in]
&\left|\prod_{\substack{p \text{ odd prime} \\ p^e || t, \, e > 0}}  \sum_{\substack{j=0\\ j \text{ even}}}^{e-1} p^{-j \nu}\right|
< \prod_{\substack{p \text{ odd prime} \\ p^e || t, \, e > 0}}  e \cdot p^{-e \Re(\nu)}
< \prod_{\substack{p \text{ odd prime} \\ p^e || t, \, e > 0}}  p^{e (-\Re(\nu)+1)}\\[.1in]
&\leq |n|^{-\Re(\nu)+1},\notag\\[.2in]
&\left| \prod_{\substack{p \text{ odd prime}\\ p^e || s, \, e > 0}}  \( \(1 - \(\frac{t}{p}\) p^{-\nu -\frac{1}{2}}\) \sum_{\substack{j=0\\ j \text{ even}}}^{e-2} p^{-j \nu} + p^{-e \nu}\) \right|\\
&= \left| \prod_{\substack{p \text{ odd prime}\\ p^e || s, \, e > 0}}  \( \sum_{\substack{j=0\\ j \text{ even}}}^{e} p^{-j \nu}  - \(\frac{t}{p}\) p^{-\nu -\frac{1}{2}} \sum_{\substack{j=0\\ j \text{ even}}}^{e-2} p^{-j \nu}\)\right|\notag\\
&\leq \prod_{\substack{p \text{ odd prime}\\ p^e || s, \, e > 0}} \( \left| \sum_{\substack{j=0\\ j \text{ even}}}^{e} p^{-j \nu} \right| +  \left| \(\frac{t}{p}\) p^{-\nu -\frac{1}{2}} \sum_{\substack{j=0\\ j \text{ even}}}^{e-2} p^{-j \nu}\right|\)\notag\\
&\leq \prod_{\substack{p \text{ odd prime}\\ p^e || s, \, e > 0}} \( \sum_{\substack{j=0\\ j \text{ even}}}^{e} p^{-e \Re(\nu)}  +  p^{-\Re(\nu) -\frac{1}{2}} \sum_{\substack{j=0\\ j \text{ even}}}^{e-2} p^{-e \Re(\nu)}\)\notag\\
&\leq \prod_{\substack{p \text{ odd prime}\\ p^e || s, \, e > 0}} \(2  p^{-\Re(\nu)}  \sum_{\substack{j=0\\ j \text{ even}}}^{e} p^{-e \Re(\nu)}\)
< \prod_{\substack{p \text{ odd prime}\\ p^e || s, \, e > 0}} p \cdot p^{-\Re(\nu)} e p^{-e \Re(\nu)}\notag\\
&<\prod_{\substack{p \text{ odd prime}\\ p^e || s, \, e > 0}}  p \cdot p^{-\Re(\nu)} p^e p^{-e \Re(\nu)} \notag\\
&=\prod_{\substack{p \text{ odd prime}\\ p^e || s, \, e > 0}} p^{(e+1)(-\Re(\nu)+1)}\notag
\leq |n^2|^{-\Re(\nu)+1}.
\end{align}
By these inequalities, it follows from \eqref{eq:scLDef} that
\begin{align}
\label{eq:LfuncBound1}
&\left| \mathcal{L}(\epsilon n, \nu) \right| \\
&<\begin{cases}
4 |n|^3 \left|L\(\nu + \frac{1}{2}, \(\frac{t}{\cdot}\)\)\right|  &\text{if } \Re(\nu) \geq 0\\
4^{-\Re(\nu)-1} |n|^{-\Re(\nu)+1} |n^2|^{-\Re(\nu)+1} \left|L\(\nu + \frac{1}{2}, \(\frac{t}{\cdot}\)\)\right|  &\text{if } \Re(\nu) < 0
\end{cases}\notag\\
&<\begin{cases}
4 |n|^3 \left|L\(\nu + \frac{1}{2}, \(\frac{t}{\cdot}\)\)\right|  &\text{if } \Re(\nu) \geq 0\\
|4 n^3|^{-\Re(\nu)+1} \left|L\(\nu + \frac{1}{2}, \(\frac{t}{\cdot}\)\)\right|  &\text{if } \Re(\nu) < 0.
\end{cases}\notag
\end{align}
}%
By the standard convexity bound for $L$-functions \cite[(2)]{Friedlander95},
for $\nu \in \C_{\neq \frac{1}{2}}$,
\begin{equation}
\label{eq:LfuncBound3}
\left| L\(\nu+\frac{1}{2},\(\frac{t}{\cdot}\)\)\right| \ll |t|^m
\end{equation}
with both the implied constant and $m$ having continuous dependence upon $\nu$.
This, along with \eqref{eq:LfuncBound1}, shows that
\begin{equation}
\label{eq:LfuncBound4}
\left| \mathcal{L}(\epsilon n,\nu) \right| \ll |4n|^m,
\end{equation}
where both the implied constant and $m$ are continuously dependent upon $\nu$.

In Theorem \ref{thm:SimplFourCoeff} we see that
for $n \neq 0$, $a_{\epsilon,\nu}(n)$ is a finite linear combination with 
bounded scalar coefficients and terms of the form
\begin{equation}
\label{eq:FourCoeffTerms}
2^k \(2^{k+2}\)^{-\nu-1} \mathcal{L}(\epsilon n,\nu)
\end{equation}
with $k$ such that $k \leq \ell + 1$ where
$n = 2^{\ell}n'$ with $\gcd(2,n') = 1$.
Also observe that in Theorem \ref{thm:SimplFourCoeff}, the number of summands of the form \eqref{eq:FourCoeffTerms} 
for a given $a_{\epsilon,\nu}(n)$ is bounded by  $|2n|$.
Thus for $\nu \neq \frac{1}{2}$,
\begin{equation}
\label{eq:FourCoeffBound}
|a_{\epsilon,\nu}(n)| \ll |4n|^m,
\end{equation}
where both the implied constant and $m$ are continuously dependent upon $\nu$.
We will utilize \eqref{eq:FourCoeffBound} to justify the convergence of $\tlE_{\epsilon,\nu}^{(\infty)}$ later in this section.

To prove that $\tlE^{(\infty)}_{\epsilon,\nu}$ has meromorphic continuation to $\C$,
we will need to make reference to the notion of a distribution \textit{vanishing to a given order at a point}.
See \cite{Miller04} for the relevant definitions and results.
We will also use the notation we established in Section \ref{sec:DblCvrsSL2}; namely,
 we let $f_0$ and $f_\infty$ denote the restrictions of $f \in V_{\epsilon,\nu}^{-\infty}$ to
$\tlN$ and $\tls^{-1} \tlN$ (respectively), where $\tls$ is defined in \eqref{eq:tlsdef}.
By \eqref{eq:sigmaIndRepUnbddIso}, we define $\dse_{\epsilon,\nu,n} \in V_{\epsilon,\nu}^\infty$
such that
\begin{equation}
\label{eq:dseDef}
\(\dse_{\epsilon,\nu,n}\)_0(x) = e(n x) \text{ and } \(\dse_{\epsilon,\nu,n}\)_\infty(x) = \sgn(-x)^{\epsilon/2} |x|^{\nu-1} \(\dse_{\epsilon,\nu,n}\)_0(x).
\end{equation}
By Propositions 2.19 and 2.26 of \cite{Miller04}, it follows that
\begin{equation}
\label{eq:dseVanish}
\(\dse_{\epsilon,\nu,n}\)_\infty \text{ vanishes to infinite order at $0$.}
\end{equation}
In light of this, the element $\dse_{\epsilon,\nu,n}$ is referred to as the \textit{canonical extension} of $e(nx)$ to infinity.
The aforementioned results of \cite{Miller04} in conjunction with \eqref{eq:FourCoeffBound}
show that 
\begin{align}
\label{eq:EisenFourSerConv}
&\sum_{n \in \Z_{\neq 0}} a_{\epsilon,\nu}(n) \dse_{\epsilon,\nu,n} \text{ converges in } V_{\epsilon,\nu}^{-\infty}, \text{ and}\\
&\(\sum_{n \in \Z_{\neq 0}} a_{\epsilon,\nu}(n) \dse_{\epsilon,\nu,n}\)_\infty \text{ vanishes to infinite order at } 0 
\text{ for all } \nu \in \C_{\neq \frac{1}{2}}.\notag
\end{align}
Likewise, for $\nu \in \C - \Z_{\leq 0}$, $\epsilon \in \{\pm 1\}$, we claim that
the constant function can be identified with the distribution
$\onebb_{\epsilon,\nu} \in V_{\epsilon,\nu}^{-\infty}$, where 
\begin{equation}
\label{eq:onebbDef}
(\onebb_{\epsilon,\nu})_0(x) = 1 \text{ and } 
\(\onebb_{\epsilon,\nu}\)_\infty(x) = \sgn(-x)^{\epsilon/2} |x|^{\nu - 1}
\end{equation}
See \eqref{eq:sigmaIndRepUnbddIso} for why, at least formally, this defines an element of $V_{\epsilon,\nu}^{-\infty}$.
At first glance, it appears that $\onebb_{\epsilon,\nu}$ is not defined for  $\nu$ with $\Re(\nu) \leq 0$,
since integration of a smooth function of compact support on $\R$ against $\(\onebb_{\epsilon,\nu}\)_\infty$
does not necessarily converge.
However, a formal application of integration by parts
allows us to extend the definition of $\(\onebb_{\epsilon,\nu}\)_\infty$ to all $\nu \in \C - \Z_{\leq 0}$.
Furthermore, one can show that as a function of $\nu$, $\onebb_{\epsilon,\nu}$ has simple poles at $\nu \in \Z_{\leq 0}$.
We adopt the convention that $\dse_{\epsilon,\nu,0} = \onebb_{\epsilon,\nu}$.

In the following lemma we utilize the distribution $\delta_{\epsilon,\nu,\infty}$, whose relevant properties we recall from \eqref{eq:tlsdeltprop}.

\begin{lem}
\label{lem:ComplEisenFourSeries}
For $\Re(\nu) > 1$,
\begin{equation}
\label{eq:ComplEisenFourSeries}
\tlE^{(\infty)}_{\epsilon,\nu} = \sum_{n \in \Z} a_{\epsilon,\nu}(n) \dse_{\epsilon,\nu,n} + \zeta_2(2 \nu + 1) \delta_{\epsilon,\nu,\infty}.
\end{equation}
\end{lem}

\begin{proof}
For $\Re(\nu) > 1$, we have already have that
$\(\tlE^{(\infty)}_{\epsilon,\nu}\)_0 = \(\sum_{n \in \Z} a_{\epsilon,\nu}(n) \dse_{\epsilon,\nu,n}\)_0$.
Since $\(\delta_{\epsilon,\nu,\infty}\)_0 = 0$, it also follows that 
\begin{equation}
\label{eq:EisenFourDiff1}
\(\tlE^{(\infty)}_{\epsilon,\nu}\)_0 = \(\sum_{n \in \Z} a_{\epsilon,\nu}(n) \dse_{\epsilon,\nu,n} + \zeta_2(2 \nu + 1) \delta_{\epsilon,\nu,\infty}\)_0.
\end{equation}
This proves that \eqref{eq:ComplEisenFourSeries} holds when restricted to $\tlN$.
It remains to establish the same equality when we restrict to $\tls^{-1} \tlN$.

By \eqref{eq:EisenFourDiff1} and \eqref{eq:sigmaIndRepUnbddEq}, it follows that
\begin{equation}
\label{eq:EisenFourDiff2}
\(\tlE^{(\infty)}_{\epsilon,\nu} - \sum_{n \in \Z} a_{\epsilon,\nu}(n) \dse_{\epsilon,\nu,n} - \zeta_2(2\nu + 1) \delta_{\epsilon,\nu,\infty} \)_\infty
\end{equation}
is the zero distribution when restricted $\R_{\neq 0}$, and therefore as a distribution on $\R$ has support contained in $\{0\}$.
By \cite[Lemma 3.1]{Miller04}, we have that $\(\onebb_{\epsilon,\nu}\)_\infty$ vanishes to order $0$ at $0$.
Since $\(\sum_{n \in \Z_{\neq 0}} a_{\epsilon,\nu}(n) \dse_{\epsilon,\nu,n}\)_\infty$ vanishes to infinite order at $0$, it follows then that
$\(\sum_{n \in \Z} a_{\epsilon,\nu}(n) \dse_{\epsilon,\nu,n}\)_\infty$ vanishes to order $0$ at $0$.
Later we will show that 
\begin{equation}
\label{eq:vanishZeroExp}
\(\tlE^{(\infty)}_{\epsilon,\nu} - \zeta_2(2 \nu + 1) \delta_{\epsilon,\nu,\infty}\)_\infty \text{ vanishes to order 0 at 0}.
\end{equation}
Assuming that this is the case, it then follows that \eqref{eq:EisenFourDiff2} vanishes to order $0$ at $0$, in addition to having
support contained in $\{0\}$.
Hence by \cite[Lemma 2.8]{Miller04}, it follows that \eqref{eq:EisenFourDiff2} is the zero distribution on $\R$,
which then proves the lemma.
Thus it remains to prove \eqref{eq:vanishZeroExp}.

By \eqref{eq:EisenInfDistInf} and \eqref{eq:tlsdeltprop},
\begin{subequations}
\begin{align}
&\(\tlE^{(\infty)}_{\epsilon,\nu} - \zeta_2(2 \nu + 1) \delta_{\epsilon,\nu,\infty}\)_\infty \notag\\
&= \epsilon i \zeta_2(2\nu + 1) \sum_{\substack{(c,d)\in \Z_{\neq 0} \times \Z \\ \gcd(c,d)=1 \\ d \equiv 1 (\mod 4) \\ c \equiv 0 (\mod 4)}}
(-c,d)_H \(\frac{c}{d}\)|d|^{-\nu-1} \sgn(d)^{\epsilon/2} \delta_{\frac{c}{d}}\notag \\
\label{eq:eisendposExp}
&= \epsilon i  \zeta_2(2\nu+1) \sum_{\substack{(c,d)\in \Z_{\neq 0} \times \Z_{>0} \\ \gcd(c,d)=1 \\ d \equiv 1 (\mod 4) \\ c \equiv 0 (\mod 4)}} 
\(\frac{c}{d}\) d^{-\nu-1}  \delta_{\frac{c}{d}}\\
\label{eq:eisendnegExp}
&\hh - \zeta_2(2\nu+1) \sum_{\substack{(c,d)\in \Z_{\neq 0} \times \Z_{<0} \\ \gcd(c,d)=1 \\ d \equiv 1 (\mod 4) \\ c \equiv 0 (\mod 4)}} \sgn(-c) 
\(\frac{c}{d}\)|d|^{-\nu-1} \delta_{\frac{c}{d}},
\end{align}
\end{subequations}
where \eqref{eq:eisendposExp} is a summation with $d > 0$ and \eqref{eq:eisendnegExp} is a summation with $d<0$.
As the comments subsequent to \eqref{eq:EisenInfDistInf} indicate, the above series converge uniformly and absolutely
against compactly supported test functions on $\R$.

We begin by showing that \eqref{eq:eisendposExp} vanishes to order $0$ at $0$.
Towards this end, observe 
\begin{align*}
&\sum_{\substack{(c,d)\in \Z_{\neq 0} \times \Z_{>0} \\ \gcd(c,d)=1 \\ d \equiv 1 (\mod 4) \\ c \equiv 0 (\mod 4)}} 
\(\frac{c}{d}\) d^{-\nu-1}  \delta_{\frac{c}{d}}
=\sum_{\substack{(c,d)\in \Z_{> 0} \times \Z_{>0} \\ \gcd(c,d)=1 \\ d \equiv 1 (\mod 4) \\ c \equiv 0 (\mod 4)}} 
\(\(\frac{c}{d}\) d^{-\nu-1}  \delta_{\frac{c}{d}} + \(\frac{-c}{d}\) d^{-\nu-1}  \delta_{\frac{-c}{d}}\)\\
&=\sum_{\substack{(c,d)\in \Z_{> 0} \times \Z_{>0} \\ \gcd(c,d)=1 \\ d \equiv 1 (\mod 4) \\ c \equiv 0 (\mod 4)}} 
\(\frac{c}{d}\) d^{-\nu-1}  (\delta_{\frac{c}{d}} + \delta_{\frac{-c}{d}}),
\end{align*}
where in the last line we used the fact that $\(\frac{-1}{d}\) = 1$ for $d \equiv 1 (\mod 4)$.
Observe that for $\phi \in C_c^\infty(\R)$,
\begin{equation*}
\int_\R  (\delta_{\frac{c}{d}}(x) + \delta_{\frac{-c}{d}}(x)) \phi(x) \, dx = -\int_\R h_{\frac{c}{d}}(x) \phi'(x) \, dx,
\end{equation*}
where
\begin{equation*}
h_{\frac{c}{d}}(x) = \begin{cases} 1 & \text{if } \frac{c}{d} < x\\ 0 & \text{if } -\frac{c}{d} \leq x \leq \frac{c}{d}\\ -1 & \text{if } x < -\frac{c}{d}.\end{cases}
\end{equation*}
For $d > 0$ and $d \equiv 1 (\mod 4)$, let
\begin{equation}
q_d(x) = \sum_{\substack{c \in \Z_{> 0} \\ \gcd(c,d)=1 \\ c \equiv 0 (\mod 4)}} \(\frac{c}{d}\) h_{\frac{c}{d}}(x).
\end{equation}
The above analysis shows that the antiderivative of \eqref{eq:eisendposExp} is equal to
\begin{equation}
\label{eq:antiDeriv1}
\epsilon i  \zeta_2(2\nu+1) \sum_{\substack{d \in \Z_{>0}\\d \equiv 1 (\mod 4)}} d^{-\nu-1} q_d(x)
= \epsilon i  \zeta_2(2\nu+1) \sum_{\substack{d \in \Z_{>0}\\d \equiv 1 (\mod 4)}} d^{-\nu} \frac{q_d(x)}{d}.
\end{equation}

If $0 \leq x \leq \frac{4}{d}$ then $q_d(x) = 0$.
If $\frac{4}{d} < x$ then there exists maximal $m \in \Z$ such that $m \equiv 0 (\mod 4)$, $\gcd(m,d) = 1$,
and $\frac{m}{d} < x$. For such $x$, 
\begin{align*}
h_{\frac{c}{d}}(x) = 0 &\text{ for } \frac{c}{d} > \frac{m}{d}, \text{ and}\\
h_{\frac{c}{d}}(x) = 1 &\text{ for } \frac{c}{d} \leq \frac{m}{d},
\end{align*}
where $\gcd(c,d)=1$ and $c \equiv 0 (\mod 4)$.
Thus for such $x$,
\begin{equation*}
q_d(x) = \sum_{\substack{c \in \Z_{> 0} \\ \gcd(c,d)=1 \\ c \equiv 0 (\mod 4)}} \(\frac{c}{d}\) h_{\frac{c}{d}}(x) = \sum_{\substack{0 < c \leq m\\ \gcd(c,d)=1 \\ c \equiv 0 (\mod 4)}} \(\frac{c}{d}\) \cdot 1  = \sum_{\substack{0 < c \leq m\\ \gcd(c,d)=1 \\ c \equiv 0 (\mod 4)}} \(\frac{c}{d}\).
\end{equation*}
By this equality, it follows that for any $x \in \R_{> 0}$, we have
$\left|\frac{q_d(x)}{d}\right| \leq \frac{m}{d} < x$.
Since $q_d(x)$ is an odd function, it follows that 
$\left| \frac{q_d(x)}{d} \right| \leq |x|$ for all $x \in \R$.
Thus
\begin{align*}
\left|\epsilon i  \zeta_2(2\nu+1) \sum_{\substack{d \in \Z_{>0}\\d \equiv 1 (\mod 4)}} d^{-\nu} \frac{q_d(x)}{d} \right|
\leq |x| \cdot \left|\epsilon i  \zeta_2(2\nu+1) \sum_{\substack{d \in \Z_{>0}\\d \equiv 1 (\mod 4)}} d^{-\nu} \right|,
\end{align*}
which implies \eqref{eq:antiDeriv1} vanishes (at least) to order $1$ at $0$ by \cite[Lemma 3.1(c)]{Miller04} (recall that we are assuming $\Re(\nu)>1$,
so the above series converges).
Therefore \eqref{eq:eisendposExp}, which is the derivative of \eqref{eq:antiDeriv1}, must vanish to order $0$ at $0$.
An almost identical argument shows that \eqref{eq:eisendnegExp} also vanishes to order $0$ at $0$,
proving \eqref{eq:vanishZeroExp} and thereby completing the proof.
\end{proof}

By \eqref{eq:EisenFourSerConv} and Theorem \ref{thm:SimplFourCoeff},
\begin{equation*}
\sum_{n \in \Z_{\neq 0}} a_{\epsilon,\nu}(n) \dse_{\epsilon,\nu,n} + \zeta_2(2\nu+1) \delta_{\epsilon,\nu,\infty}
\end{equation*}
is well-defined for $\nu \in \C - \{0,\frac{1}{2}\}$.
Notice that $a_{\epsilon,\nu}(0) \onebb_{\epsilon,\nu}$ is missing from this summation.
Although $\onebb_{\epsilon,\nu}$ has simple poles at $\nu \in \Z_{\leq 0}$ (see discussion following \eqref{eq:onebbDef}),
since by Proposition \ref{prop:EisenInfConst}, $a_{\epsilon,\nu}(0)$ has simple zeros at $\nu \in \Z_{\leq 0}$,
it follows that
\begin{equation*}
\sum_{n \in \Z} a_{\epsilon,\nu}(n) \dse_{\epsilon,\nu,n} + \zeta_2(2\nu+1) \delta_{\epsilon,\nu,\infty}
\end{equation*}
is well-defined for $\nu \in \C - \{0,\frac{1}{2}\}$.
Thus by Lemma \ref{lem:ComplEisenFourSeries}
and the uniqueness of meromorphic continuation, we obtain the following proposition.

\begin{prop}
\label{prop:EisenInfMeroCont}
The metaplectic Eisenstein distribution $\tlE_{\epsilon,\nu}^{(\infty)} \in V^{-\infty}_{\epsilon,\nu}$ has holomorphic continuation
to $\C$ except for simple poles on $\{0,\frac{1}{2}\}$.
In particular,
\begin{equation*}
\tlE^{(\infty)}_{\epsilon,\nu} = \sum_{n \in \Z} a_{\epsilon,\nu}(n) \dse_{\epsilon,\nu,n} + a_{\epsilon,\nu}(\infty) \delta_{\epsilon,\nu,\infty},
\end{equation*}
where 
\begin{equation}
\label{eq:aInfCoeff}
a_{\epsilon,\nu}(\infty) = \zeta_2(2\nu+1) = (1-2^{-2\nu-1}) \zeta(2\nu + 1).
\end{equation}
\end{prop}

In the prior proposition we introduced the notation $a_{\epsilon,\nu}(\infty)$ for the coefficient of $\delta_{\epsilon,\nu,\infty}$.
Throughout this paper, we will encounter Fourier series of the form
\begin{equation*}
\sum_{n \in \Z} c_{\epsilon,\nu}(n) \dse_{\epsilon,\nu,n} + h(\epsilon,\nu) \delta_{\epsilon,\nu,\infty},
\end{equation*}
where $h(\epsilon,\nu)$ is a meromorphic function in $\nu$.
When this occurs, we denote $h(\epsilon,\nu)$ by $c_{\epsilon,\nu}(\infty)$.

\section{The Metaplectic Eisenstein Distribution at \texorpdfstring{$0$}{Zero}}
\label{sec:MetaAtZero}

Recall that $\tlE^{(\infty)}_{\epsilon,\nu}$ is the distributional analogue of the usual metaplectic Eisenstein series based at the 
cusp $\infty$.
Next we will define $\tlE^{(0)}_{\epsilon,\nu}$, which will be a distributional analogue of the metaplectic Eisenstein series based
at the cusp $0$.
This Eisenstein distribution will appear in the functional equation (Theorem \ref{thm:FuncEq}).

Let
\begin{equation}
\label{eq:tlxiZeroDef}
\tlxi_{0} =\tla_2^{-1} \tls = \tls \tla_2 = \(\mm{0}{-2^{-1}}{2}{0},1\)
\end{equation}
where $\tls = \(\mm{0}{-1}{1}{0},1\)$ and $\tla_{2} = \(\mm{2}{0}{0}{2^{-1}},1\)$ 
(as defined in \eqref{eq:tlsdef} and \eqref{eq:sltl2elems} respectively).
Let 
\begin{equation}
\label{eq:tlGamma0Def}
\tlGamma_{0} = \tlxi_0^{-1} \tlGamma_\infty \tlxi_0 =  \tlxi_{0} \tlGamma_\infty \tlxi_{0}^{-1} = \left\{\(\mm{1}{0}{-4n}{1},1\):n \in \Z\right\},
\end{equation}
which one can check is a subgroup of $\tlGamma_1(4)$.
We define the \textit{metaplectic Eisenstein distribution at $0$} to be the following
distribution in $V_{\epsilon,\nu}^{-\infty}$:
\begin{equation}
\label{eq:MetaEisenAtZeroDef}
\tlE^{(0)}_{\epsilon,\nu} = \zeta_2(2 \nu + 1) \sum_{\tlgamma \in \tlGamma_1(4) / \tlGamma_{0}} \pi\(\tlgamma \tlxi_{0}\) \delta_{\epsilon,\nu,\infty},
\end{equation}
where $\Re(\nu) > 1$.
Since $\delta_{\epsilon,\nu,\infty}$ is invariant under left translation by $\tlGamma_\infty$,
it follows that $\pi(\tlxi_{0}) \delta_{\epsilon,\nu,\infty}$
is invariant under left translation by $\tlGamma_{0}$ (by virtue of \eqref{eq:tlGamma0Def}),
justifying the appearance of $\tlGamma_1(4) / \tlGamma_{0}$ in the definition of $\tlE^{(0)}_{\epsilon,\nu}$.
By construction, we see that $\tlE^{(0)}_{\epsilon,\nu}$ is formally $\tlGamma_1(4)$-invariant.
We will justify the convergence of the series defining $\tlE^{(0)}_{\epsilon,\nu}$ momentarily.

Since $\tlGamma_1(4)$ is generated by 
\begin{align}
&\(\mm{1}{1}{0}{1},1\) = \tlxi_{0} \(\mm{1}{0}{-4}{1},1\) \tlxi_{0}^{-1} \text{ and }\\
&\(\mm{1}{0}{-4}{1},1\) = \tlxi_{0} \(\mm{1}{1}{0}{1},1\) \tlxi_{0}^{-1},\notag
\end{align}
it follows 
that $\tlxi_{0} \tlGamma_1(4) \tlxi_{0}^{-1} = \tlGamma_1(4)$.
Since $\tlxi_{0} \tlGamma_1(4) \tlxi_{0}^{-1} = \tlGamma_1(4)$ and 
$\tlxi_{0} \tlGamma_\infty \tlxi_{0}^{-1} = \tlGamma_{0}$,
\begin{equation}
\label{eq:EZeroToEInf}
\tlE^{(0)}_{\epsilon,\nu} = \zeta_2(2 \nu + 1) \sum_{\tlgamma \in \tlGamma_1(4) / \tlGamma_{\infty}} \pi\(\tlxi_{0} \tlgamma\) \delta_{\epsilon,\nu,\infty} = \pi\(\tlxi_{0}\) \tlE^{(\infty)}_{\epsilon,\nu}.
\end{equation}
Therefore the series defining $\tlE^{(0)}_{\epsilon,\nu}$ converges for $\Re(\nu)>1$.
Furthermore, we see that $\tlE^{(0)}_{\epsilon,\nu}$ inherits the meromorphic continuation of $\tlE^{(\infty)}_{\epsilon,\nu}$
(stated in Proposition \ref{prop:EisenInfMeroCont}),
and therefore $\tlE^{(0)}_{\epsilon,\nu}$ is holomorphic on $\C$ except for simple poles on $\{0,\frac{1}{2}\}$.

Recall that in Section \ref{sec:DblCvrsSL2}, we let $f_0$ and $f_\infty$ denote the restrictions of $f \in V_{\epsilon,\nu}^{-\infty}$ to
$\tlN$ and $\tls^{-1} \tlN$.
Since $\tlE_{\epsilon,\nu}^{(0)}$ is $\tlGamma_1(4)$-invariant, it follows that
$\(\tlE^{(0)}_{\epsilon,\nu}\)_0$ is periodic, and consequently, has the following Fourier series expansion:
\begin{equation}
\(\tlE^{(0)}_{\epsilon,\nu}\)_0(x) = \sum_{n \in \Z} b_{\epsilon,\nu}(n) e(nx),
\end{equation}
where
\begin{equation}
b_{\epsilon,\nu}(n) = \int_0^1 \(\tlE^{(0)}_{\epsilon,\nu}\)_0(x) e(-nx) \, dx.
\end{equation}
In \eqref{eq:dseDef}, we extended the definition of $e(n x)$ to an element $\dse_{\epsilon,\nu,n}$ of $V_{\epsilon,\nu}^{-\infty}$. 
In Proposition \ref{prop:EisenInfMeroCont}, we gave a series
expansion for $\tlE^{(\infty)}_{\epsilon,\nu}$ in terms of $\dse_{\epsilon,\nu,n}$ and $\delta_{\epsilon,\nu,\infty}$.
We claim that a similar series expansion holds for $\tlE^{(0)}_{\epsilon,\nu}$:
\begin{equation}
\label{eq:ComplFourSerEisAtZero}
\tlE^{(0)}_{\epsilon,\nu} = \sum_{n \in \Z} b_{\epsilon,\nu}(n) \dse_{\epsilon,\nu,n},
\end{equation}
where
\begin{equation}
\label{eq:ZeroEisenFourCoeff}
b_{\epsilon,\nu}(n) = \int_0^1 \(\tlE^{(0)}_{\epsilon,\nu}\)_0(x) e(-n x) \, dx.
\end{equation}
Notice that \eqref{eq:ComplFourSerEisAtZero} has no $\delta_{\epsilon,\nu,\infty}$ term (in contrast to
the series given in Proposition \ref{prop:EisenInfMeroCont}).
Thus
\begin{equation}
\label{eq:bInfCoeff}
b_{\epsilon,\nu}(\infty) = 0,
\end{equation}
in terms of the notation of Proposition \ref{prop:EisenInfMeroCont}.

To show that \eqref{eq:ComplFourSerEisAtZero} holds, we must first establish that 
$\displaystyle \sum_{n \in \Z} b_{\epsilon,\nu}(n) \dse_{\epsilon,\nu,n}$ does indeed converge for all
$\nu \in \C - \{0,\frac{1}{2}\}$.
To accomplish this, we need only show that $b_{\epsilon,\nu}(n)$ is bounded polynomially in $n$ for each 
$\nu \in \C - \{0,\frac{1}{2}\}$, just as we did for $a_{\epsilon,\nu}(n)$ in \eqref{eq:FourCoeffBound}. 
Since we already have the meromorphic continuation
of $\tlE^{(0)}_{\epsilon,\nu}$, it follows immediately from  \cite[(9.15)]{Folland99} that $b_{\epsilon,\nu}(n)$ is bounded polynomially 
for such $\nu$.
Having established the convergence of the Fourier series, we then observe that for $\Re(\nu)>1$,
\begin{equation}
\label{eq:MetaEisenZeroSerDiff}
\(\tlE^{(0)}_{\epsilon,\nu}\)_\infty - \(\sum_{n \in \Z} b_{\epsilon,\nu}(n) \dse_{\epsilon,\nu,n} \)_\infty
\end{equation}
has support contained in $\{0\}$. By \cite[Lemma 2.8]{Miller04},
if we can show that both $\(\tlE^{(0)}_{\epsilon,\nu}\)_\infty$ and $\(\sum_{n \in \Z} b_{\epsilon,\nu}(n) \dse_{\epsilon,\nu,n} \)_\infty$
vanish to order $0$ at $0$, then \eqref{eq:MetaEisenZeroSerDiff} equals zero for $\Re(\nu)>0$, whereupon 
the uniqueness of meromorphic continuation then establishes this identity for all $\nu \in \C - \{0,\frac{1}{2}\}$.
This in turn, would immediately imply \eqref{eq:ComplFourSerEisAtZero}.

In \eqref{eq:EisenFourSerConv}, we stated that $\(\sum_{n \in \Z_{\neq 0}} a_{\epsilon,\nu}(n) \dse_{\epsilon,\nu,n} \)_\infty$ 
vanishes to infinite order at $0$. This result relied upon \cite[Proposition 2.19]{Miller04}, the proof of which only uses
the fact that $a_{\epsilon,\nu}(n)$ is polynomially bounded in $n$.
Since we already have such a bound for $b_{\epsilon,\nu}(n)$, a reapplication of this argument shows that
$\(\sum_{n \in \Z_{\neq 0}} b_{\epsilon,\nu}(n) \dse_{\epsilon,\nu,n} \)_\infty$ vanishes to infinite order at $0$,
from which it then follows that $\(\sum_{n \in \Z} b_{\epsilon,\nu}(n) \dse_{\epsilon,\nu,n}\)_\infty$
vanishes to order $0$ at $0$.

Proving that $\(\tlE^{(0)}_{\epsilon,\nu}\)_\infty$ vanishes to order $0$ at $0$ is more involved.
To obtain this result from our previous computations, observe that by \eqref{eq:EZeroToEInf} and the transformation law in \eqref{eq:indRep}
defined by \eqref{eq:sigmaDef} and \eqref{eq:omegaepsnu}, we have
\begin{align*}
&\(\tlE^{(0)}_{\epsilon,\nu}\)_\infty = \(\pi(\tlxi_{0}) \tlE^{(\infty)}_{\epsilon,\nu}\)_\infty
= \(\pi\(\tls \tla_2\) \tlE^{(\infty)}_{\epsilon,\nu}\)_\infty 
= \(\pi\(\tls^2 \tla_2\) \tlE^{(\infty)}_{\epsilon,\nu}\)_0\\
&= \(\pi(\tla_2) \pi(\tlm_{-1,-1}) \tlE^{(\infty)}_{\epsilon,\nu}\)_0
=  \sigma_\epsilon(\tlm_{-1,-1}) \(\pi\(\tla_2\) \tlE^{(\infty)}_{\epsilon,\nu}\)_0\\
&= \epsilon i  \(\pi\(\tla_2\) \tlE^{(\infty)}_{\epsilon,\nu}\)_0.
\end{align*}
From this, we can conclude that $\(\tlE^{(0)}_{\epsilon,\nu}\)_\infty$ vanishes to order $0$ at $0$ if
$\(\tlE^{(\infty)}_{\epsilon,\nu}\)_0$ vanishes to order $0$ at $0$.%
\footnote{The action of $\tla_2$ preserves the order of vanishing at $0$ since 
it simply scales the $x$-variable by a constant factor.}
We will spare the reader a detailed proof that such vanishing occurs,
but this can be done by modifying the proof given in Lemma \ref{lem:ComplEisenFourSeries}.
One can check that the key lies in not having a $\delta_0$ term in $\(\tlE^{(\infty)}_{\epsilon,\nu}\)_0$.
By \eqref{eq:EisenInfDistZero}, we see that this is indeed the case.

To prove the functional equation in Theorem \ref{thm:FuncEq}, 
we need to compute an explicit formula for the Fourier coefficient $b_{\epsilon,\nu}(0)$.
The derivation of this formula makes considerable use of the Kronecker symbol.  We refer the reader  back to Proposition \ref{prop:KronProps},
where the relevant properties of the Kronecker symbol are stated.

\begin{prop}
\label{prop:b0coeff}
For $\Re(\nu) > 1$,
\begin{equation}
\label{eq:b0coeff}
b_{\epsilon,\nu}(0) =  \epsilon i 2^{-\nu-1} \zeta_2(2\nu) = \epsilon  i 2^{-\nu-1} (1-2^{-2\nu}) \zeta(2 \nu).
\end{equation}
\end{prop}

\begin{proof}
Let $\tlgamma^{-1} = \(\mm{a}{b}{c}{d},\(\frac{c}{d}\)\) \in \tlGamma_1(4)$ with $c \neq 0$.
To prove \eqref{eq:b0coeff} we will first prove \eqref{eq:MetaEistlgmAct1b}, which provides
a convenient expression for $\(\pi\(\tlgamma \tlxi_0\) \delta_{\epsilon,\nu,\infty}\)_0$;
recall that $\tlE^{(0)}_{\epsilon,\nu}$, which we defined in \eqref{eq:MetaEisenAtZeroDef}, is a sum of
$\pi\(\tlgamma \tlxi_0\) \delta_{\epsilon,\nu,\infty}$ for such $\tlgamma^{-1}$.
In order to prove \eqref{eq:MetaEistlgmAct1b}, we first need to establish \eqref{eq:tlgmabKron}.

Since $\tlgamma^{-1} \in \tlGamma_1(4)$, $ad \equiv 1 (\mod c)$, which implies $\(\frac{c}{a}\) \(\frac{c}{d}\) = \(\frac{c}{ad}\) = \(\frac{c}{1}\) = 1$
by parts (f), (h), and (c) of Proposition \ref{prop:KronProps}.
Thus 
\begin{equation}
\label{eq:b0coeffStep}
\(\frac{c}{d}\) = \(\frac{c}{a}\).
\end{equation}
Next we evaluate $\(\frac{-bc}{a}\)$ for $b \neq 0$.
Since we are already assuming $c \neq 0$ then $bc \neq 0$.
Since $\tlgamma^{-1} \in \tlGamma_1(4)$, $-bc \equiv 1 (\mod a)$.
If $a>0$ then by Proposition \ref{prop:KronProps}(g),
$\(\frac{-bc}{a}\) = \(\frac{1}{a}\)=1$.
If $a<0$ then by parts (f), (c), and (g) of Proposition \ref{prop:KronProps},
\begin{equation*}
\(\frac{-bc}{a}\) = \(\frac{-bc}{-1}\) \(\frac{-bc}{-a}\) = \sgn(-bc)\(\frac{1}{-a}\) = \sgn(-bc).
\end{equation*}
Therefore
\begin{equation}
\label{eq:b0coeffStep2}
\(\frac{-b}{a}\)\(\frac{c}{a}\) = \(\frac{-bc}{a}\) = (a,-bc)_H.
\end{equation}
By parts (e) and (i) of Proposition \ref{prop:KronProps}, $\(\frac{-b}{a}\) = \(\frac{-1}{a}\) \(\frac{b}{a}\) = \(\frac{b}{a}\)$
since $a \equiv 1 (\mod 4)$.
Thus by \eqref{eq:b0coeffStep} and \eqref{eq:b0coeffStep2},
\begin{equation}
\label{eq:tlgmabKron}
\(\frac{c}{d}\) = \(\frac{b}{a}\)(a,-bc)_H,
\end{equation}
for $\tlgamma^{-1} \in \tlGamma_1(4)$ with $bc \neq 0$.

Having proved \eqref{eq:tlgmabKron}, we turn our attention to proving \eqref{eq:MetaEistlgmAct1b}.
Consider once more $\tlgamma^{-1} \in \tlGamma_1(4)$ with $c \neq 0$.
Observe
\begin{equation}
\tlgamma \tlxi_{0} \tls =  \( \mm{-\frac{d}{2}}{2b}{\frac{c}{2}}{-2a}, \sgn(c) \(\frac{c}{d}\) \) =\( \mm{-2a}{-2b}{-\frac{c}{2}}{-\frac{d}{2}}, \sgn(c) \(\frac{c}{d}\) \)^{-1}.
\end{equation}
Thus by Lemma \ref{lem:MetaUnbddAct}(a),
\begin{align}
\label{eq:tlgammaxiDeltaAct}
&\(\pi\(\tlgamma\tlxi_{0}\)\delta_{\epsilon,\nu,\infty}\)_0(x) = \(\pi\(\tlgamma \tlxi_{0} \tls\) \delta_{\epsilon,\nu,0}\)_0(x)\\
&= \sgn(c) \(\frac{c}{d}\) \(\frac{c}{2}, -\frac{c}{2}x - \frac{d}{2}\)_H \left|-\frac{c}{2}x-\frac{d}{2}\right|^{\nu-1} \sgn\(-\frac{c}{2}x - \frac{d}{2}\)^{\epsilon/2} \delta_0\(\frac{-2a x -2b}{-\frac{c}{2}x -\frac{d}{2}}\)\notag\\
&= \sgn(c)  \(\frac{c}{d}\) (c, -cx - d)_H 2^{-\nu+1}|cx+d|^{\nu-1} \sgn(-cx - d)^{\epsilon/2} \delta_0\(4 \(\frac{a x + b}{cx + d}\)\),\notag
\end{align}
as an equality between distributions on $\R_{\neq \frac{-d}{c}}$ since $c \neq 0$.
Let $\phi$ be a test function of compact support on $\R_{\neq \frac{-d}{c}}$. Observe 
{\allowdisplaybreaks
\begin{align}
\label{eq:MetEisen0Calc}
&\int_{\R_{\neq \frac{-d}{c}}} \sgn(c) \(\frac{c}{d}\) (c, -cx - d)_H 2^{-\nu + 1}|cx+d|^{\nu-1} \sgn(-cx - d)^{\epsilon/2}\\*
&\hh \cdot \delta_0\(4 \(\frac{a x + b}{cx + d}\)\) \phi(x) \, dx\notag\\
&= \int_{\R_{\neq 0}} \sgn(c) \(\frac{c}{d}\)(c, -cx)_H 2^{-\nu + 1}|cx|^{\nu-1} \sgn(-cx)^{\epsilon/2}\notag\\*
&\hh \cdot \delta_0\(4 \(\frac{a x - \frac{ad}{c} + b}{cx}\)\) \phi\(x-\frac{d}{c}\) \, dx\notag\\
&= \int_{\R_{\neq 0}} \sgn(c) \(\frac{c}{d}\) (c, -cx)_H 2^{-\nu + 1}|cx|^{\nu-1} \sgn(-cx)^{\epsilon/2} \delta_0\( \frac{4a}{c} - \frac{4}{c^2 x}\)\notag\\*
&\hh \cdot \phi\(x-\frac{d}{c}\) \, dx\notag\\
&= \int_{\R_{\neq 0}} \sgn(c) \(\frac{c}{d}\)  \(c, \frac{-4x}{c}\)_H 2^{-\nu+1} 4^{\nu} |c|^{-\nu-1} |x|^{\nu-1} \sgn\(\frac{-4x}{c}\)^{\epsilon/2}\notag\\*
&\hh \cdot \delta_0\( \frac{4a}{c} - \frac{1}{x}\) \phi\(\frac{4x}{c^2}-\frac{d}{c}\) \, dx\notag\\
&= \int_{\R_{\neq 0}} \sgn(c)  \(\frac{c}{d}\) \(c, \frac{4}{cx}\)_H 2^{\nu+1} |c|^{-\nu-1} |x|^{-\nu-1} \sgn\(\frac{4}{cx}\)^{\epsilon/2} \delta_0\( \frac{4a}{c} + x\)\notag\\*
&\hh \cdot \phi\(\frac{-4}{c^2 x}-\frac{d}{c}\) \, dx\notag\\
&= \sgn(c) \(\frac{c}{d}\) \(c, \frac{-1}{a}\)_H 2^{\nu+1} |4a|^{-\nu-1} \sgn\(\frac{-1}{a}\)^{\epsilon/2} \phi\(\frac{1}{a c}-\frac{d}{c}\)\notag\\
&= \sgn(c)  \(\frac{c}{d}\)\(c, -a\)_H  |2a|^{-\nu-1} \sgn\(-a\)^{\epsilon/2} \phi\(\frac{-b}{a}\).\notag
\end{align}}%
Therefore by \eqref{eq:tlgammaxiDeltaAct} and \eqref{eq:MetEisen0Calc},
\begin{align}
\label{eq:MetEisen1Calc}
&\(\pi\(\tlgamma\tlxi_{0}\)\delta_{\epsilon,\nu,\infty}\)_0(x)
= \sgn(c)  \(\frac{c}{d}\)\(c, -a\)_H  |2a|^{-\nu-1} \sgn\(-a\)^{\epsilon/2} \delta_{\frac{-b}{a}}(x)
\end{align}
as an equality between distributions on $\R_{\neq \frac{-d}{c}}$ for $\tlgamma^{-1} \in \tlGamma_1(4)$ with $c \neq 0$.

We can further simplify \eqref{eq:MetEisen1Calc} by using \eqref{eq:tlgmabKron}; in particular,
if $b \neq 0$,
\begin{align*}
&\sgn(c)  \(\frac{c}{d}\)\(c, -a\)_H \sgn\(-a\)^{\epsilon/2} = \sgn(c)  (a,-bc)_H \(c, -a\)_H \(\frac{b}{a}\) \sgn\(-a\)^{\epsilon/2}.
\end{align*}
If $a>0$ then 
\begin{equation*}
\sgn(c) (a,-bc)_H \(c, -a\)_H = \sgn(c)^2 = 1,
\end{equation*}
but if $a<0$ then
\begin{equation*}
\sgn(c) (a,-bc)_H (c,-a)_H = \sgn(c) \sgn(-bc) = \sgn(-b).
\end{equation*}
Thus 
\begin{equation*}
\sgn(c) (a,-bc)_h (c,-a)_H = (a,-b)_H,
\end{equation*}
and hence
\begin{equation*}
\sgn(c)  \(\frac{c}{d}\)\(c, -a\)_H \sgn\(-a\)^{\epsilon/2} = \(\frac{b}{a}\)(a,-b)_H \sgn(-a)^{\epsilon/2}.
\end{equation*}
By this equality and \eqref{eq:MetEisen1Calc}, for $b \neq 0$ we have
\begin{equation}
\label{eq:MetaEistlgmAct1}
\(\pi\(\tlgamma\tlxi_{0}\)\delta_{\epsilon,\nu,\infty}\)_0 = \(\frac{b}{a}\) (a,-b)_H \sgn\(-a\)^{\epsilon/2} |2a|^{-\nu-1} \delta_{\frac{-b}{a}},
\end{equation}
as an equality between distributions on $\R_{\neq \frac{-d}{c}}$ for $\tlgamma^{-1} \in \tlGamma_1(4)$ with $c \neq 0$.
If instead $b=0$ then since $b=0$ implies $a=d=1$, we have that
\begin{align*}
&\sgn(c)  \(\frac{c}{d}\)\(c, -a\)_H  = \sgn(c) \sgn(c) = 1 = \(\frac{b}{a}\).
\end{align*}
Thus for all $\tlgamma^{-1} \in \tlGamma_1(4)$ with $c \neq 0$,
\begin{equation}
\label{eq:MetaEistlgmAct1b}
\(\pi\(\tlgamma\tlxi_{0}\)\delta_{\epsilon,\nu,\infty}\)_0 = \lambda(a,b) \sgn\(-a\)^{\epsilon/2} |2a|^{-\nu-1} \delta_{\frac{-b}{a}},
\end{equation}
where
\begin{equation}
\lambda(a,b) = \begin{cases} \(\frac{b}{a}\) (a,-b)_H &  ab \neq 0\\ \(\frac{b}{a}\) & \text{otherwise}.\end{cases}
\end{equation}
Notice $\lambda(a,b) = 0$ if $\gcd(a,b) \neq 1$ since $\(\frac{b}{a}\) = 0$ if $\gcd(a,b) \neq 1$ (by Proposition \ref{prop:KronProps}(a)).

We wish to describe $\(\pi\(\tlgamma \tlxi_{0}\)\delta_{\epsilon,\nu,\infty}\)_0$ about the point $\frac{-d}{c}$.
To do this, observe
\begin{equation}
\tlgamma \tlxi_{0} = \(\mm{-2b}{-\frac{d}{2}}{2a}{\frac{c}{2}}, \(\frac{c}{d}\) (a,c)_H \) = \(\mm{\frac{c}{2}}{\frac{d}{2}}{-2a}{-2b},\(\frac{c}{d}\) (a,c)_H \)^{-1}.
\end{equation}
Since $\(\delta_{\epsilon,\nu,\infty}\)_0(x) = 0$, then by Lemma \ref{lem:MetaUnbddAct}(a),
\begin{equation*}
\(\pi\(\tlgamma \tlxi_{0}\)\delta_{\epsilon,\nu,\infty}\)_0 = 0,
\end{equation*}
as an equality between distributions on $\R_{\neq \frac{-b}{a}}$.
Since $\frac{-d}{c} \neq \frac{-b}{a}$, it follows that $\(\pi\(\tlgamma\tlxi_{0}\)\delta_{\epsilon,\nu,\infty}\)_0$ vanishes about the point $\frac{-d}{c}$.
Thus \eqref{eq:MetaEistlgmAct1b} holds as  an equality between distributions on $\R$.

Since in \eqref{eq:MetaEistlgmAct1b}, $\(\pi\(\tlgamma \tlxi_0\) \delta_{\epsilon,\nu,\infty}\)_0$ is expressed
in terms of $a$ and $b$, and since $\tlE^{(0)}_{\epsilon,\nu}$ is defined as a summation over $\tlGamma_1(4) / \tlGamma_{0}$,
we wish to index the cosets of $\tlGamma_1(4) / \tlGamma_{0}$ in terms of $a$ and $b$ as well.
Observe 
\begin{align}
&\tlgamma \cdot \(\mm{1}{0}{-4n}{1},1 \) = \(\mm{d}{-b}{-c}{a}, \(\frac{c}{d}\)\) \cdot \(\mm{1}{0}{-4n}{1},1 \)\\
&= \(\mm{d+4bn}{-b}{-c-4an}{a},\(\frac{-c}{a}\) \).\notag
\end{align}
From this equality we see that to each coset of $\tlGamma_1(4) / \tlGamma_{0}$ there corresponds $(a,b) \in \Z^2$ such that 
$\gcd(a,b)=1$ and $a \equiv 1(\mod 4)$.  
This correspondence is unique, for if
\begin{equation*}
\tlgamma' = \(\mm{*}{-b}{*}{a},*\),
\end{equation*}
as must be the case if $\tlgamma' \tlGamma_{0}$ and $\tlgamma \tlGamma_{0}$ corresponded to the same $(a,b) \in \Z^2$,
then
\begin{equation*}
(\tlgamma')^{-1} \tlgamma = \(\mm{a}{b}{*}{*},*\) \(\mm{*}{-b}{*}{a},*\) = \(\mm{*}{0}{*}{*},*\).
\end{equation*}
Since the matrix coordinate of this element must have determinant $1$, and since the diagonal entries are each congruent to $1$ modulo $4$,
it follows that $(\tlgamma')^{-1} \tlgamma \in \tlGamma_{0}$.
Therefore the correspondence of cosets of 
$\tlGamma / \tlGamma_{0}$  with such $(a,b) \in \Z^2$ is in fact unique.
Conversely, when given $(a,b) \in \Z^2$ such that $\gcd(a,b)=1$ and $a \equiv 1 (\mod 4)$, it follows that $\gcd(a,4b)=1$. Thus
there exists $c',d \in \Z$ such that $ad - 4bc' = ad - b(4c') =1$.  Since $a \equiv 1 (\mod 4)$ then $d \equiv 1(\mod 4)$.
If we let $c=4c'$ then $c \equiv 0 (\mod 4)$.  Thus we are able to construct $\tlgamma$ which corresponds to such $(a,b)$.
Therefore
\begin{equation}
\label{eq:cosetIndex}
\tlGamma_1(4) / \tlGamma_{0} \cong \{(a,b) \in \Z^2: \gcd(a,b)=1, a \equiv 1 (\mod 4)\}.
\end{equation}

By \eqref{eq:MetaEisenAtZeroDef}, \eqref{eq:MetaEistlgmAct1b}, and \eqref{eq:cosetIndex},
\begin{align*}
&\(\tlE_{\epsilon,\nu}^{(0)}\)_0 = \zeta_2(2\nu+1) \sum_{\substack{(a,b) \in \Z^2 \\ a \equiv 1 (\mod 4)}} \lambda(a,b) |2a|^{-\nu-1} \sgn\(-a\)^{\epsilon/2} \delta_{\frac{-b}{a}}
\end{align*}
(recall $\lambda(a,b) = 0$ if $\gcd(a,b) \neq 1$).
Observe
\begin{align*}
&\frac{b_{\epsilon,\nu}(0)}{\zeta_2(2\nu + 1)} = \int_0^1 \sum_{\substack{(a,b)\in \Z^2 \\ a \equiv 1 (\mod 4)}} \lambda(a,b) |2a|^{-\nu-1} \sgn\(-a\)^{\epsilon/2} \delta_{\frac{-b}{a}}(x) \, dx\\
&= \epsilon i \sum_{\substack{(a,b) \in \Z_{>0} \times \Z \\ 0 \leq -b < a \\ a \equiv 1 (\mod 4)}} \lambda(a,b) |2a|^{-\nu-1} + \sum_{\substack{(a,b) \in \Z_{<0} \times \Z \\ 0 \geq -b > a \\ a \equiv 1 (\mod 4)}} \lambda(a,b) |2a|^{-\nu-1}\\
&=\epsilon i \sum_{\substack{(a,b) \in \Z_{>0} \times \Z \\ 0 \leq b < a \\ a \equiv 1 (\mod 4)}} \lambda(a,-b) |2a|^{-\nu-1} + \sum_{\substack{(a,b) \in \Z_{>0} \times \Z \\ 0 \leq b < a \\ a \equiv 3 (\mod 4)}} 
\lambda(-a,b) |2a|^{-\nu-1}.
\end{align*}
If $a \equiv 1(\mod 4)$ and $a > 0$ then by parts (e) an (i) of Proposition \ref{prop:KronProps},
\begin{equation*}
\lambda(a,-b) = \(\frac{-b}{a}\) = \(\frac{-1}{a}\)\(\frac{b}{a}\) = \(\frac{b}{a}\),
\end{equation*}
both for $b \neq 0$ and $b = 0$.
Likewise, if $a \equiv 3 (\mod 4)$, $a > 0$, $b \geq 0$, then by parts (f) and (c) of Proposition \ref{prop:KronProps},
\begin{equation*}
\lambda(-a,b) = -\(\frac{b}{-a}\) = -\(\frac{b}{-1}\) \(\frac{b}{a}\) = -\(\frac{b}{a}\),
\end{equation*}
both for $b > 0$ and $b = 0$.
Thus
\begin{align*}
&\frac{b_{\epsilon,\nu}(0)}{\zeta_2(2\nu + 1)} = \epsilon i \sum_{a \in \Z_{>0}} \(\sum_{b \in \Z/a\Z} \(\frac{b}{a}\) \) \Delta_a^{\epsilon}  |2a|^{-\nu-1},
\end{align*}
where $\Delta_a$ is defined in \eqref{eq:DeltaDef}.
Observe that if $a$ is a square then 
\begin{equation*}
\sum_{b \in \Z/a\Z} \(\frac{b}{a}\)= \phi(a),
\end{equation*}
but if $a$ is not a square then
\begin{equation*}
\sum_{b \in \Z/a\Z} \(\frac{b}{a}\)= 0.
\end{equation*}
The latter case follows since if $a$ is not a square then there exists $b' \in \Z$ such that $\gcd(a,b')=1$ and $\(\frac{b'}{a}\)=-1$, and thus
\begin{equation*}
-\(\sum_{b \in \Z/a\Z} \(\frac{b}{a}\)\) = \(\frac{b'}{a}\) \(\sum_{b \in \Z/a\Z} \(\frac{b}{a}\)\) =  \(\sum_{b \in \Z/a\Z} \(\frac{b'b}{a}\)\) = \(\sum_{b \in \Z/a\Z} \(\frac{b}{a}\)\).
\end{equation*}
Therefore
\begin{align*}
&b_{\epsilon,\nu}(0) = \epsilon i \zeta_2(2\nu+1) \sum_{a \in \Z_{>0}}  \Delta_{a^2}^{\epsilon} \phi(a^2) |2a^2|^{-\nu-1}\\
&= \epsilon i \zeta_2(2\nu+1) 2^{-\nu-1} \sum_{\substack{a \in \Z_{>0}\\a \text{ odd}}} \phi(a^2) |a|^{-2\nu-2}.
\end{align*}
Since $\phi(a^2)$ is multiplicative, it follows from \eqref{eq:eulerPhiSum} that
\begin{align*}
&b_{\epsilon,\nu}(0) = \epsilon i \zeta_2(2\nu+1) 2^{-\nu-1} \prod_{p \text{ odd prime}} \sum_{k \in \Z_{\geq 0}} \phi(p^{2k}) (p^k)^{-2\nu-2}\\
&= \epsilon i \zeta_2(2\nu+1) 2^{-\nu-1} \prod_{p \text{ odd primes}} \frac{1-p^{-2\nu-1}}{1-p^{-2\nu}} = \epsilon i  2^{-\nu-1} \zeta_2(2\nu).\qedhere
\end{align*}
\end{proof}

\section{Intertwining Operators}
\label{sec:Intertwine}

In this section, we give results concerning intertwining operators on $\tlSL_2(\R)$.
Much of this material is standard \cite{Knapp01,Wallach88a,Wallach88b}, but hard to find in print.
Throughout, we make reference to much of the notation defined in Section \ref{sec:DblCvrsSL2},
in particular \eqref{eq:sltl2elems}.
Recall that \eqref{eq:TransLawsB} gives transformation laws for $f \in V_{\epsilon,\nu}^{-\infty}$.
From this we deduce a similar statement for $f \in V_{\epsilon,-\nu}^{-\infty}$
(notice that we have $-\nu$ instead of $\nu$):
\begin{align}
\label{eq:TransLaws}
&f(\tlg \, \tla_u \tlm_{-1,\kappa} \tln_-) = \epsilon \kappa i u^{\nu+1} f(\tlg),\\
&f(\tlg \, \tla_u \tlm_{1,\kappa} \tln_-) = \kappa u^{\nu+1} f(\tlg),\notag
\end{align}where $\kappa, \epsilon \in \{\pm 1\}$, $\nu \in \C$, and $u \in \R_{>0}$.
This transformation law, along with \eqref{eq:TransLawsB}, will be used repeatedly and implicitly throughout this section.

For $\Re(\nu)>0$, let $I_{\epsilon,\nu}:V_{\epsilon,-\nu}^{\infty} \to V_{\epsilon,\nu}^{\infty}$ denote the
intertwining operator
\begin{equation}
\label{eq:InuDef}
I_{\epsilon,\nu} f(\tilde{g}) = \int_{-\infty}^\infty f(\tilde{g} \, \tls \, \tilde{n}_t) \, dt.
\end{equation}
It is a well-known result in representation theory that the integral in \eqref{eq:InuDef} converges absolutely for $\Re(\nu)>0$
\cite[\S VII.6]{Knapp01}.
It remains to justify our statement regarding the codomain of $I_{\epsilon,\nu}$. Observe that by changing variables and
employing \eqref{eq:TransLaws},
\begin{align*}
&I_{\epsilon,\nu}f(\tlg \, \tla_u \tln_{-,r}) = \int_{-\infty}^\infty f(\tlg \, \tla_u \tln_{-,r} \tls \, \tln_t) \, dt 
= \int_{-\infty}^\infty f\(\tlg \, \tls \, \tln\(\frac{t-r}{u^2}\) \tla(u^{-1})\) \, dt\\
&= (u^{-1})^{\nu+1} \int_{-\infty}^\infty f\(\tlg \, \tls \, \tln\(\frac{t}{u^2}\)\) \, dt 
=  u^{-\nu-1} u^2 \int_{-\infty}^\infty f\(\tlg \, \tls \, \tln_t\) \, dt \\
&= u^{-\nu+1} I_{\epsilon,\nu}f (\tlg).
\end{align*}
Since $\tlM$ is the center of $\tlSL_2(\R)$, then by \eqref{eq:TransLaws},
\begin{align*}
&I_{\epsilon,\nu}f(\tlg \tlm_{-1,\kappa}) = \int_{-\infty}^\infty f(\tlg \, \tls \, \tln_t \tlm_{-1,\kappa}) \, dt
= \epsilon \kappa i I_{\epsilon,\nu}f(\tlg), \text{ and}\\
&I_{\epsilon,\nu}f(\tlg \tlm_{1,\kappa}) = \int_{-\infty}^\infty f(\tlg \, \tls \, \tln_t \, \tlm_{1,\kappa}) \, dt
= \kappa I_{\epsilon,\nu}f(\tlg).
\end{align*}
By these equalities, we see that $I_{\epsilon,\nu}f$ satisfies \eqref{eq:TransLawsB} (or equivalently,
\eqref{eq:TransLaws} for $V_{\epsilon,\nu}^{-\infty}$).
It is straightforward to confirm that $I_{\epsilon,\nu}f$ is smooth, since differentiation can be expressed
in terms of the left regular representation.
Thus $I_{\epsilon,\nu}$ does indeed have $V_{\epsilon,\nu}^{\infty}$ as its codomain.

A well-known result from representation theory states that $I_{\epsilon,\nu}$ can be meromorphically continued to all of $\C$ 
\cite[\S VII]{Knapp01}.
To be more precise, one identifies elements of $V_{\epsilon,\nu}^\infty$ as elements of $C^\infty(\tlK)$ via restriction to $\tlK$ 
(the so-called ``compact picture").  Under this identification,
the intertwining operators $I_{\epsilon,\nu}$ become a family of operators on $C^\infty(\tlK)$ indexed by $\epsilon$ and $\nu$.
The statement of meromorphic continuation asserts the existence of intertwining operators $I_{\epsilon,\nu}$ for all $\nu \in \C$, except
possibly for a subset
of discrete points.
The statement also asserts
that for any fixed $f \in C^\infty(\tlK)$ and $\tlk \in \tlK$,
the map $\nu \mapsto I_{\epsilon,\nu}f(\tlk)$ extends to a meromorphic function from $\C$ to $\C$.

Since $I_{\epsilon,\nu}$ is continuous on $V_{\epsilon,-\nu}^\infty$ \cite{Casselman89} and since $V_{\epsilon,-\nu}^{\infty}$ is dense in 
$V^{-\infty}_{\epsilon,-\nu}$, it follows that $I_{\epsilon,\nu}$ can be extended to a continuous, intertwining operator
from $V^{-\infty}_{\epsilon,-\nu}$ to $V^{-\infty}_{\epsilon,-\nu}$.
In this section, we wish to give a Fourier series expansion for $I_{\epsilon,\nu}\(\tlE^{(\infty)}_{\epsilon,-\nu}\)$ 
similar to the one given for $\tlE^{(\infty)}_{\epsilon,\nu}$ in Proposition \ref{prop:EisenInfMeroCont}.
Recall that the pairing $\<\cdot,\cdot\>_{\epsilon,\nu}$ defined in \eqref{eq:pairingDef} can be extended to 
$V_{\epsilon,\nu}^{-\infty} \times V_{-\epsilon,-\nu}^\infty$.
The following lemma describes how the intertwining operator $I_{\epsilon,\nu}$ behaves with respect to this pairing.

\begin{lem}
\label{lem:AlmostAdjoint}
For $f_1 \in V_{\epsilon,-\nu}^{-\infty}$ and $f_2 \in V_{-\epsilon,-\nu}^{\infty}$,
\begin{equation*}
\<I_{\epsilon,\nu} f_1, f_2\>_{\epsilon,\nu} = -\epsilon i \<f_1, I_{-\epsilon,\nu} f_2\>_{\epsilon,-\nu}.
\end{equation*}
\end{lem}

\begin{proof}
It suffices to prove this lemma for the case where $f_1 \in V_{\epsilon,-\nu}^\infty$.
To see why this is sufficient, recall that the pairing
$\<\cdot,\cdot\>_{\epsilon,\nu}$ is separately continuous on both factors of $V_{\epsilon,\nu}^{-\infty} \times V_{-\epsilon,-\nu}^\infty$
(see \eqref{eq:sepCts}) and that $I_{\epsilon,\nu}$ is continuous on $V_{\epsilon,-\nu}^{-\infty}$.
Once we establish this identity for smooth $f_1$, the lemma then follows for 
$f_1 \in V_{\epsilon,-\nu}^{-\infty}$ by the aforementioned continuity and
the fact that $V_{\epsilon,-\nu}^\infty$ is dense in $V_{\epsilon,-\nu}^{-\infty}$.

Suppose $\Re(\nu)>0$.
By performing a change of variables,
\begin{equation*}
I_{\epsilon,\nu} f_1(\tlg) = \int_{-\infty}^\infty f_1(\tlg \, \tls \, \tln_t) \, dt = \int_{-\pi/2}^{\pi/2} f_1(\tlg \, \tls \, \tln_{\tan(\phi)}) \sec(\phi)^2 \, d\phi.
\end{equation*}
Since
\begin{equation}
\tln_{\tan(\phi)} = \tlk_{-\phi} \tla_{\cos(\phi)} \tln_{-,\cos(\phi) \sin(\phi)} 
\text{ for } \phi \in \(-\frac{\pi}{2},\frac{\pi}{2}\),
\end{equation}
it then follows 
from \eqref{eq:TransLaws} that
\begin{align}
\label{eq:InterTwinedf}
&I_{\epsilon,\nu} f_1(\tlg) = \int_{-\pi/2}^{\pi/2} f_1\(\tlg \tls \tlk_{- \phi}\tla_{\cos(\phi)} \tln_{-,\cos(\phi) \sin(\phi)}\) \sec(\phi)^2 \, d\phi\\
&= \int_{-\pi/2}^{\pi/2} f_1(\tlg \, \tls \, \tlk_{- \phi}) \cos(\phi)^{\nu+1} \sec(\phi)^2 \, d\phi\notag\\
&= \int_{-\pi/2}^{\pi/2} f_1(\tlg \, \tls \, \tlk_{- \phi}) \cos(\phi)^{\nu-1} \, d\phi.\notag
\end{align}
By changing variables and using the fact that $\tlK$ is an abelian subgroup,
\begin{align*}
&\<I_{\epsilon,\nu} f_1, f_2\>_{\epsilon,\nu} = \int_{-\pi/2}^{\pi/2} \int_{-\pi/2}^{\pi/2} f_1(\tlk_\theta \tls  \tlk_{-\phi}) \cos(\phi)^{\nu-1} \, d\phi \, f_2(\tlk_\theta) \, d\theta\\
&= \int_{-\pi/2}^{\pi/2} \int_{-\pi/2}^{\pi/2} f_1(\tlk_\theta \tls^{-1}  \tlk_{-\phi} \tlm_{-1,-1})\cos(\phi)^{\nu-1} \, d\phi \, f_2(\tlk_\theta) \, d\theta\\
&= -\epsilon i \int_{-\pi/2}^{\pi/2} \int_{-\pi/2}^{\pi/2} f_1(\tlk_\theta \tls^{-1}  \tlk_{-\phi}) \cos(\phi)^{\nu-1} \, f_2(\tlk_\theta) \, d\theta \, d\phi\\
&= -\epsilon i \int_{-\pi/2}^{\pi/2} \int_{-\pi/2}^{\pi/2} f_1(\tlk_\theta ) f_2(\tlk_\theta \tls  \tlk_{\phi}) \cos(\phi)^{\nu-1} \, d\theta  \, d\phi\\
&= -\epsilon i \int_{-\pi/2}^{\pi/2} f_1(\tlk_\theta )  \int_{-\pi/2}^{\pi/2} f_2(\tlk_\theta \tls  \tlk_{-\phi}) \cos(\phi)^{\nu-1} \, d\phi \, d\theta.
\end{align*}
The change in the $\theta$ variable does not result in a change in the bounds of integration since the integrand is
$\tlM$-invariant and $d\theta$ is a Haar measure for $\tlK /\tlM$ (as mentioned in the discussion
following \eqref{eq:pairingDef}).
By \eqref{eq:InterTwinedf}, which is also applicable to $f_2 \in V_{-\epsilon,-\nu}^{\infty}$,
\begin{align*}
&\<I_{\epsilon,\nu} f_1, f_2\>_{\epsilon,\nu} = -\epsilon i \int_{-\pi/2}^{\pi/2} f_1(\tlk_\theta ) I_{-\epsilon,\nu}f_2(\tlk_\theta) \, d\theta\\
&=-\epsilon i \<f_1, I_{-\epsilon,\nu}f_2\>_{\epsilon,-\nu}.
\end{align*}
This proves the lemma for the case when $\Re(\nu)> 0$.
Invoking the uniqueness of meromorphic continuation completes the proof.
\end{proof}

By Schur's lemma, we know that $I_{\epsilon,-\nu} \circ I_{\epsilon,\nu}$ is a scalar operator.
The following lemma, which we will use for computing this scalar operator, will also be used in Section \ref{sec:EisenSeriesFunctionalEq} when
we describe  how to recover the classical metaplectic Eisenstein series from the metaplectic Eisenstein distributions.

\begin{lem}
\label{lem:KFiniteLem}
For $\ell \in 2 \Z$, there exists unique $\phi_{\epsilon,\nu,\ell} \in V_{\epsilon,\nu}^\infty$ such
that 
\begin{equation}
\label{eq:phiDef}
\phi_{\epsilon,\nu,\ell}(\tlk_\theta) = \exp\(-\epsilon \(\frac{1}{2}+\ell\) i \theta\).
\end{equation}
Furthermore, 
\begin{equation*}
I_{\epsilon,\nu}\phi_{\epsilon,-\nu,\ell} = i^{-\epsilon \ell} \frac{1-\epsilon i}{\sqrt{2}} \cdot \frac{\pi  2^{1-\nu } \Gamma (\nu )}{\Gamma\left(-\frac{{1}}{2} \left(\ell +\frac{1}{2}\right) + \frac{\nu +1}{2} \right) \Gamma \left(\frac{1}{2} \left(\ell +\frac{1}{2}\right) +\frac{\nu +1}{2}\right)} \phi_{\epsilon,\nu,\ell}
\end{equation*}
for $\nu \in \C - \Z_{\leq 0}$.
\end{lem}

\begin{proof}
Define
\begin{equation}
\label{eq:phiDefB}
\phi_{\epsilon,\nu,\ell}(\tlk_\theta \tla_u \tln_-) = u^{-\nu+1} \exp\(-\epsilon \(\frac{1}{2}+\ell\) i \theta\),
\end{equation}
which determines $\phi_{\epsilon,\nu,\ell}$ completely since $\tlSL_2(\R) = \tlK \tlA \tlN_-$.
Notice that $\phi_{\epsilon,\nu,\ell}|_{\tlK}$ is a character.
Thus if $\tlg = \tlk_{\theta} \tla_u \tln_-$ for $\tlk_\theta \in \tlK$, $\tla_u \in \tlA$ and $\tln_- \in \tlN_-$,
then since $\tlM \subset \tlK$ is the center of $\tlSL_2(\R)$,
\begin{align}
\label{eq:IsotypicPartA}
&\phi_{\epsilon,\nu,\ell}(\tlg \tlm) 
= \phi_{\epsilon,\nu,\ell}(\tlm \tlk_\theta \tla_u \tln_-)
= \phi_{\epsilon,\nu,\ell}(\tlm \tlk_\theta) u^{-\nu+1}\\
&= \phi_{\epsilon,\nu,\ell}(\tlm) \phi_{\epsilon,\nu,\ell}(\tlk_\theta) u^{-\nu+1}
= \phi_{\epsilon,\nu,\ell}(\tlm) \phi_{\epsilon,\nu,\ell}(\tlg),\notag
\end{align}
for all $\tlm \in \tlM$. With this in mind, observe
\begin{align}
\label{eq:IsotypicPartB}
&\phi_{\epsilon,\nu,\ell}(\tlm_{-1,\kappa}) = \phi_{\epsilon,\nu,\ell}(\tlk_{-\kappa \pi}) = \exp\( \frac{\epsilon \kappa i \pi}{2} \) = \epsilon \kappa i,\\
&\phi_{\epsilon,\nu,\ell}(\tlm_{1,\kappa}) = \phi_{\epsilon,\nu,\ell}(\tlk_{(\kappa-1) \pi})= \exp\(-\frac{\epsilon(\kappa-1) i \pi}{2}\) = \kappa, \notag
\end{align}
for $\kappa \in \{\pm 1\}$.
It follows from the definition that
\begin{equation}
\label{eq:IsotypicPartC}
\phi_{\epsilon,\nu,\ell}(\tlg \tla_u \tln_-) = u^{-\nu+1} \phi_{\epsilon,\nu,\ell}(\tlg).
\end{equation}
Thus by \eqref{eq:IsotypicPartA}, \eqref{eq:IsotypicPartB}, and \eqref{eq:IsotypicPartC},
we see that $\phi_{\epsilon,\nu,\ell}$ satisfies \eqref{eq:TransLawsB}
(or equivalently, \eqref{eq:TransLaws} for $V_{\epsilon,\nu}^{-\infty}$).
Since it is also clear that $\phi_{\epsilon,\nu,\ell}$ is smooth, it follows that
$\phi_{\epsilon,\nu,\ell} \in V_{\epsilon,\nu}^\infty$.
Clearly $\phi_{\epsilon,\nu,\ell}$ is the unique element of $V_{\epsilon,\nu}^\infty$ which satisfies
\begin{equation*}
\phi(\tlk_\theta) = \exp\(-\epsilon \(\frac{1}{2}+\ell\) i \theta\),
\end{equation*}
since elements of $V_{\epsilon,\nu}^\infty$ are determined by their restriction to $\tlK$.

Since
\begin{equation}
\tln_x = \tlk\(-\arctan(x)\) \, \tln_{-,x} \, \tla\((1+x^2)^{-1/2}\)
\end{equation}
then for $\Re(\nu)>0$,
\begin{align}
\label{eq:KFiniteLemStep1}
&I_{\epsilon,\nu} \phi_{\epsilon,-\nu,\ell}(\tlk_\theta) = \int_{-\infty}^\infty \phi_{\epsilon,-\nu,\ell}(\tlk_\theta \tls \, \tln_x) \, dx\\
&= \int_{-\infty}^\infty \phi_{\epsilon,-\nu,\ell}\Big(\tlk_\theta \tls \tlk(-\arctan(x))\Big) (1+x^2)^{-\frac{\nu+1}{2}}\, dx\notag\\
&= \phi_{\epsilon,-\nu,\ell}(\tlk_\theta) \phi_{\epsilon,-\nu,\ell}(\tls) \int_{-\infty}^\infty \phi_{\epsilon,-\nu,\ell}\Big(\tlk(-\arctan(x))\Big) (1+x^2)^{-\frac{\nu+1}{2}}\, dx\notag\\
&= \phi_{\epsilon,\nu,\ell}(\tlk_\theta) i^{-\epsilon \ell} \frac{1-\epsilon i}{\sqrt{2}} \int_{-\infty}^\infty \phi_{\epsilon,-\nu,\ell}\Big(\tlk(-\arctan(x))\Big) (1+x^2)^{-\frac{\nu+1}{2}}\, dx,\notag
\end{align}
where in the last equality we used that
\begin{align*}
&\phi_{\epsilon,-\nu,\ell}(\tlk_\theta) = \phi_{\epsilon,\nu,\ell}(\tlk_\theta),\\
&\phi_{\epsilon,-\nu,\ell}(\tls) = \phi_{\epsilon,-\nu,\ell}(\tlk_{\pi/2}) = \exp\(\frac{-\epsilon \ell i \pi}{2}\) \exp\(\frac{-\epsilon i \pi}{4}\)
= i^{-\epsilon \ell}  \frac{1-\epsilon i}{\sqrt{2}}.
\end{align*}
It remains to evaluate the integral
\begin{equation*}
\int_{-\infty}^\infty \phi_{\epsilon,-\nu,\ell}\Big(\tlk(-\arctan(x))\Big) (1+x^2)^{-\frac{\nu+1}{2}}\, dx.
\end{equation*}
Using the identity $\arctan(x) = \frac{i}{2}(\log(1-ix)-\log(1+ix))$, we find
\begin{align}
\label{eq:phiArcTanToI}
&\phi_{\epsilon,-\nu,\ell}\Big(\tlk(-\arctan(x))\Big) = \exp\(\epsilon \(\frac{1}{2}+\ell\) i \arctan(x)\)\\
&=\exp\(-\frac{\epsilon}{2} \(\frac{1}{2}+\ell\)(\log(1-ix)-\log(1+ix)) \)\notag\\
&= (1-i x)^{-\frac{\epsilon}{2} \(\frac{1}{2}+\ell\)} (1+i x)^{\frac{\epsilon}{2}\(\frac{1}{2}+\ell\)}.\notag
\end{align}
Because of our choice of branch cut for $\log$ given in \eqref{eq:branchCut}, we also have
\begin{equation}
\label{eq:branchCutResult}
(1+x^2)^{-\frac{\nu+1}{2}} = (1-ix)^{-\frac{\nu+1}{2}}(1+ix)^{-\frac{\nu+1}{2}}.
\end{equation}
Thus
\begin{align}
\label{eq:InterTwinedKFinite}
&\int_{-\infty}^\infty \phi_{\epsilon,-\nu,\ell}\Big(\tlk(-\arctan(x))\Big) (1+x^2)^{-\frac{\nu+1}{2}}\, dx\\
&= \int_{-\infty}^\infty (1-i x)^{-\frac{\epsilon}{2} \(\frac{1}{2}+\ell\)-\frac{\nu+1}{2}} (1+i x)^{\frac{\epsilon}{2}\(\frac{1}{2}+\ell\)-\frac{\nu+1}{2}} dx\notag\\
&= \frac{\pi  2^{1-\nu } \Gamma (\nu )}{\Gamma\left(-\frac{\epsilon}{2} \left(\ell +\frac{1}{2}\right) + \frac{\nu +1}{2} \right) \Gamma \left(\frac{\epsilon}{2} \left(\ell +\frac{1}{2}\right) +\frac{\nu +1}{2}\right)}\notag
\end{align}
for $\Re(\nu)>0$, where this last equality follows from \cite[(1.3)]{Gasper89}.
Since $\epsilon = \pm 1$, by symmetry we can simply take $\epsilon = 1$ in the last line of \eqref{eq:InterTwinedKFinite}.
Combining \eqref{eq:KFiniteLemStep1} and \eqref{eq:InterTwinedKFinite} proves the assertion of the lemma
when both sides are restricted to $\tlK$ for $\Re(\nu)>0$.
We complete the proof by noting that elements of $V_{\epsilon,\nu}^\infty$ are uniquely determined by their
restriction to $\tlK$ (hence our formula holds for all of $\tlSL_2(\R)$) and by invoking the uniqueness of meromorphic continuation.
\end{proof}

\begin{rem}
Observe that $\{\tlk_\theta \mapsto \exp\(\frac{j}{2} \theta\): j \in \Z\}$ is the set of all characters on $\tlK$.
This follows from the fact that $\theta \mapsto \tlk_\theta$ is an isomorphism from $\R / [-2\pi,2\pi)$ onto $\tlK$.
A character is said to factor over $K$, the maximal compact subgroup of $\SL_2$,
if the map $k_\theta \mapsto \exp\(\frac{j}{2} \theta\)$ is also a character on $K$.
The characters of $\tlK$ that do not factor over $K$ (i.e. the genuine characters on $\tlK$)
are precisely those for which $j \notin 2\Z$.
Of these, one can show via \eqref{eq:TransLawsB} (or equivalently, \eqref{eq:TransLaws} for $V_{\epsilon,\nu}^\infty$) 
that $\{\phi_{\epsilon,\nu,\ell}|_{\tlK}: \ell \in 2\Z\}$ is the set of all genuine characters contained in $V_{\epsilon,\nu}^\infty$.
\end{rem}

By applying Lemma \ref{lem:KFiniteLem} twice and utilizing standard identities for the $\Gamma$ function, we find that
\begin{equation*}
I_{\epsilon,-\nu} I_{\epsilon,\nu} \phi_{\epsilon,-\nu,0} = \frac{2 \pi \epsilon i \cot(\pi \nu)}{\nu} \phi_{\epsilon,-\nu,0}.
\end{equation*}
Since we know  by Schur's lemma that $I_{\epsilon,-\nu} I_{\epsilon,\nu}$ is a scalar operator,
it follows that $I_{\epsilon,-\nu} I_{\epsilon,\nu} f = \frac{2 \pi \epsilon i \cot(\pi \nu)}{\nu}f$ for all $f \in V_{\epsilon,-\nu}^\infty$.
By applying a density argument, we can extend this result to $V_{\epsilon,-\nu}^{-\infty}$:

\begin{lem} 
\label{lem:IntertwinFactor}
For $f \in V_{\epsilon,-\nu}^{-\infty}$,
\begin{equation*}
I_{\epsilon,-\nu} I_{\epsilon,\nu} f = \frac{2 \pi \epsilon i \cot(\pi \nu)}{\nu} f
\end{equation*}
for $\nu \in \C - \Z$.
\end{lem}

Recall that we defined $\onebb_{\epsilon,\nu} \in V^{-\infty}_{\epsilon,\nu}$ in \eqref{eq:onebbDef} and
$\delta_{\epsilon,\nu,\infty} \in V^{-\infty}_{\epsilon,\nu}$ just prior to \eqref{eq:tlsdeltprop}.

\begin{lem}
\label{lem:InterTwinDeltInfOnebb}
\hspace{2em}
\begin{itemize}[leftmargin=.3in,font=\normalfont\textbf]
\item[(a)] $\displaystyle I_{\epsilon,\nu}\delta_{\epsilon,-\nu,\infty} = \onebb_{\epsilon,\nu}$,
\item[(b)] $\displaystyle I_{\epsilon,\nu}\onebb_{\epsilon,-\nu} = \frac{2 \pi \epsilon i \cot(\pi \nu)}{\nu} \delta_{\epsilon,\nu,\infty}$.
\end{itemize}
\end{lem}

\begin{proof}
Let $f \in V_{-\epsilon,-\nu}^{\infty}$ with $\Re(\nu)>0$.  
By Lemma \ref{lem:AlmostAdjoint} and \eqref{eq:cmpctToUnbdd},
\begin{align*}
&\<I_{\epsilon,\nu}\delta_{\epsilon,-\nu,\infty}, f\>_{\epsilon, \nu} 
= -\epsilon i \<\delta_{\epsilon,-\nu,\infty}, I_{-\epsilon,\nu} f\>_{\epsilon,-\nu}\\
&= -\epsilon i \int_{-\pi/2}^{\pi/2}\delta_{\epsilon,-\nu,0}(\tls^{-1} \tlk_\theta) I_{-\epsilon,\nu} f(\tlk_\theta) d\theta \\
&= -\epsilon i \int_{-\pi/2}^{\pi/2}\delta_{\epsilon,-\nu,0}(\tlk_\theta) I_{-\epsilon,\nu} f(\tls \tlk_\theta) d\theta\\
&= -\epsilon i \int_{-\infty}^\infty\delta_{\epsilon,-\nu,0}(\tln_{t}) I_{-\epsilon,\nu} f(\tls \, \tln_{t}) dt = -\epsilon i  I_{-\epsilon,\nu} f(\tls).
\end{align*}
Once again, the change in the $\theta$ variable does not result in a change in the bounds of integration since the integrand is
$\tlM$-invariant and $d\theta$ is a Haar measure for $\tlK /\tlM$.
Since $\tls^{\,2} = \tlm_{-1,-1} \in \tlM$ and since $\tlM$ is the center of $\tlSL_2(\R)$,
\begin{align*}
&I_{-\epsilon,\nu} f(\tls) = \int_{-\infty}^\infty f(\tls^{\,2} \tln_x) \, dt = \int_{-\infty}^\infty f(\tln_x \tlm_{-1,-1}) \, dt = \epsilon i \int_{-\infty}^\infty f(\tln_t) \, dt.
\end{align*}
Thus
\begin{equation*}
\<I_{\epsilon,\nu} \delta_{\epsilon,-\nu,\infty}, f\>_{\epsilon,\nu} = -\epsilon i \cdot \epsilon i \int_{-\infty}^\infty f(\tln_t) \, dt = \int_{-\infty}^\infty f(\tln_t) \, dt.
\end{equation*}
This last equality shows that $\int_{-\infty}^\infty \onebb_{\epsilon,\nu}(\tln_t) f(\tln_t) \, dt$ converges for any 
$f \in V_{-\epsilon,-\nu}^{\infty}$.%
\footnote{This can also be seen from Proposition \ref{prop:pairingToCondConvIntegral}.}
Since the pairing $\< \cdot, \cdot \>_{\epsilon,\nu}$ is non-degenerate, it follows that for $\Re(\nu)>0$,
$I_{\epsilon,\nu} \delta_{\epsilon,-\nu,\infty}  = \onebb_{\epsilon,\nu}$.
Since $\nu \mapsto \onebb_{\epsilon,\nu}$ is a meromorphic function on $\C$, the lemma then follows from the uniqueness
of meromorphic continuation.

Lemma \ref{lem:IntertwinFactor} applied to $\delta_{\epsilon,-\nu,\infty}$ shows that
\begin{equation*}
I_{\epsilon,-\nu} I_{\epsilon,\nu}\delta_{\epsilon,-\nu,\infty} = \frac{2 \pi \epsilon i \cot(\pi \nu)}{\nu}\delta_{\epsilon,-\nu,\infty}.
\end{equation*}
Therefore by part (a), 
\begin{equation*}
I_{\epsilon,-\nu} \onebb_{\epsilon,\nu} = \frac{2 \pi \epsilon i \cot(\pi \nu)}{\nu} \delta_{\epsilon,-\nu,\infty},
\end{equation*}
which yields part (b) when we replace $\nu$ with $-\nu$.
\end{proof}

Let
\begin{align}
\label{eq:GammaFactorDefs}
&G_0(\nu) = \frac{2 \Gamma(\nu)}{(2\pi)^\nu} \cos\(\frac{\pi \nu}{2}\),\\
&G_1(\nu) = \frac{2 i \Gamma(\nu)}{(2 \pi)^{\nu}} \sin\(\frac{\pi \nu}{2}\),\notag \\
&G_{(\epsilon_1,\epsilon_2)}(\nu) = \frac{(1-\epsilon_1 i) \Gamma(\nu)}{(2\pi)^\nu} \(\cos\(\frac{\pi \nu}{2}\) + \epsilon_1 \epsilon_2 \sin\(\frac{\pi \nu}{2}\)\), \notag
\end{align}
where $\epsilon_1, \epsilon_2 \in \{\pm 1\}$.  
It follows directly from the definition of $G_{(\epsilon_1,\epsilon_2)}(\nu)$ that
\begin{equation}
\label{eq:Geps1eps2Id1}
G_{(\epsilon_1,\epsilon_2)}(\nu) = -\epsilon_1 i G_{(-\epsilon_1,-\epsilon_2)}(\nu).
\end{equation}
Notice that $G_0(\nu)$ has simple poles at the non-positive even integers, $G_1(\nu)$ has simple poles at the non-positive odd integers,
and $G_{(\epsilon_1,\epsilon_2)}(\nu)$ has simple poles at all the non-positive integers.

For $0 < \Re(\nu) < 1$, it is well-known that
\begin{equation}
\label{eq:GammaCondConvIntegral}
G_\eta(\nu) = \lim_{m_1, m_2 \to \infty} \int_{-m_1}^{m_2}  \sgn(t)^\eta |t|^{\nu-1} e(t) \, dt
\end{equation}
for $\eta \in \{0,1\}$.%
\footnote{We obtain \eqref{eq:GammaCondConvIntegral} by noting that
\begin{equation*}
\int_{-m_1}^{m_2}  \sgn(t)^\eta |t|^{\nu-1} e(t) \, dt 
= \int_{0}^{m_2}  |t|^{\nu-1} e(t) \, dt + (-1)^\eta \int_{0}^{m_1} |t|^{\nu-1} e(-t) \, dt,
\end{equation*}
and recalling that
$\displaystyle \lim_{m \to \infty} \int_0^m t^{\nu-1} e(t) \, dt = \frac{\Gamma(\nu)}{(2\pi)^\nu} e(\nu/4)$
for $0 < \Re(\nu) < 1$,
the latter of which follows by changing variables in the definition of 
$\displaystyle \Gamma(\nu) = \lim_{m \to \infty} \int_0^m t^{\nu-1}e^{-t} \, dt$.}
The following lemma shows that a similar integral representation holds for $G_{(\epsilon_1,\epsilon_2)}(\nu)$.

\begin{lem}
\label{lem:InterIdentGeps1eps2}
For $0 < \Re(\nu) <1$,
\begin{equation*}
G_{(\epsilon_1,\epsilon_2)}(\nu) = \lim_{m_1, m_2 \to \infty} \int_{-m_1}^{m_2} \sgn(-\epsilon_2 t)^{-\epsilon_1 / 2} |t|^{\nu-1} e(t) \, dt.
\end{equation*}
\end{lem}

\begin{proof}
If $\epsilon_2 = 1$ then
\begin{equation*}
\sgn(-\epsilon_2 t)^{-\epsilon_1/2} = \sgn(-t)^{-\epsilon_1/2} = \frac{1}{2}(1-\sgn(t)) - \frac{\epsilon_1 i}{2}(1 + \sgn(t)).
\end{equation*}
If $\epsilon_2 = -1$ then
\begin{equation*}
\sgn(-\epsilon_2 t)^{-\epsilon_1/2} = \sgn(t)^{-\epsilon_1/2} = \frac{1}{2}(1+\sgn(t)) - \frac{\epsilon_1 i}{2}(1 - \sgn(t)).
\end{equation*}
Thus in either case,
\begin{equation*}
\sgn(-\epsilon_2 t)^{-\epsilon_1/2} = \frac{1}{2}(1- \epsilon_2 \sgn(t)) - \frac{\epsilon_1 i}{2}(1 + \epsilon_2 \sgn(t)).
\end{equation*}
Therefore by \eqref{eq:GammaCondConvIntegral}, for $0 < \Re(\nu) < 1$,
\begin{align*}
&\lim_{m_1,m_2 \to \infty} \int_{-m_1}^{m_2} \sgn(-\epsilon_2 t)^{-\epsilon_1 / 2} e(t) |t|^{\nu-1} \, dt\\
&= \lim_{m_1,m_2 \to \infty} \int_{-m_1}^{m_2} \( \frac{1}{2}(1- \epsilon_2 \sgn(t)) - \frac{\epsilon_1 i}{2}(1 + \epsilon_2 \sgn(t)) \) |t|^{\nu-1} e(t) \, dt\\
&=\frac{1}{2} \(G_0(\nu) - \epsilon_2 G_1(\nu)\) - \frac{\epsilon_1 i}{2} \(G_0(\nu) + \epsilon_2 G_1(\nu)\)\\
&=\(\(\cos\(\frac{\pi \nu}{2}\) - \epsilon_2 i \sin\(\frac{\pi \nu}{2}\)\) - \epsilon_1 i \(\cos\(\frac{\pi \nu}{2}\) + \epsilon_2 i \sin\(\frac{\pi \nu}{2}\)\)\)\frac{\Gamma(\nu)}{(2\pi)^\nu}\\
&=(1-\epsilon_1 i)\(\cos\(\frac{\pi\nu}{2}\)+\epsilon_1 \epsilon_2 \sin\(\frac{\pi\nu}{2}\)\) \frac{\Gamma(\nu)}{(2\pi)^\nu}\\
&=G_{(\epsilon_1,\epsilon_2)}(\nu).\qedhere
\end{align*}
\end{proof}

\begin{lem}
\label{lem:InterTwinDse}
For $n \neq 0$, 
\begin{equation*}
I_{\epsilon,\nu} \dse_{\epsilon,-\nu,n} =  G_{(\epsilon,\sgn(n))}(\nu) |n|^{-\nu} \dse_{\epsilon,\nu,n},
\end{equation*}
for all $\nu \in \C - \Z_{\leq 0}$.
Recall that $\dse_{\epsilon,\nu,n}$ is defined in \eqref{eq:dseDef}.
\end{lem}

\begin{proof}
Assume $0 < \Re(\nu) < 1$.
Once we prove the identity for $\nu$ in this range,
the lemma then follows by invoking the uniqueness of meromorphic continuation.

By Lemma \ref{lem:AlmostAdjoint},
\begin{equation}
\label{eq:InterDseStep1}
\<I_{\epsilon,\nu} \dse_{\epsilon,-\nu,n}, \phi\>_{\epsilon,\nu} =  - \epsilon i \<\dse_{\epsilon,-\nu,n}, I_{-\epsilon,\nu} \phi\>_{\epsilon,-\nu}
\end{equation}
for $\phi \in V^\infty_{-\epsilon,-\nu}$.
Suppose $\supp(\phi) \subset \tlN \tlB$.
Consequently, $\phi_0 \in C_c^\infty(\R)$ (see \eqref{eq:Sub0SubInfDef} for the definition $\phi_0$).
Since for $t \neq 0$,
\begin{equation}
\tln_x \tls\,\tln_t = \tln(-t^{-1}+x) \tla(|t|^{-1}) \tlm_{\sgn(t),\sgn(t)} \tln_-(t^{-1}),
\end{equation}
it then follows from \eqref{eq:InuDef} and \eqref{eq:TransLaws} for $V_{-\epsilon,-\nu}^\infty$,
\begin{align}
\label{eq:InterDseStep2}
&\(I_{-\epsilon,\nu} \phi\)_0(x) = \int_{-\infty}^\infty \sgn(t)^{\epsilon / 2} \phi_0(-t^{-1}+x) |t|^{-\nu-1} \, dt\\
&= \int_{-\infty}^\infty \sgn(-t)^{\epsilon / 2} |t|^{\nu-1}  \phi_0(t+x) \, dt  \notag\\
&= \int_{-\infty}^\infty \sgn(x-t)^{\epsilon / 2} |x-t|^{\nu-1}  \phi_0(t) \, dt,\notag
\end{align}
with the above integrals converging absolutely.

By \eqref{eq:dseDef}, we see that $\dse_{\epsilon,-\nu,n}$ satisfies the hypotheses of Proposition \ref{prop:pairingToCondConvIntegral}.
Thus by \eqref{eq:InterDseStep1}, \eqref{eq:InterDseStep2}, and Proposition \ref{prop:pairingToCondConvIntegral},
\begin{align}
\label{eq:InterDseStep3}
&\<I_{\epsilon,\nu} \dse_{\epsilon,-\nu,n}, \phi\>_{\epsilon,\nu}
= -\epsilon i \lim_{m \to \infty} \int_{-m}^m \(I_{-\epsilon,\nu} \phi\)_0(x)  \(\dse_{\epsilon,-\nu,n}\)_0(x) \, dx \\
&= -\epsilon i \lim_{m \to \infty} \int_{-m}^m \int_{-\infty}^\infty \sgn(x-t)^{\epsilon / 2} |x-t|^{\nu-1}  \phi_0(t) \, dt \, e(nx) \, dx.\notag
\end{align}
Since we've assumed $\phi_0$ has compact support, we can think of the integration over $\R$ 
in the last line of \eqref{eq:InterDseStep3} as
actually occurring over a finite interval, say $[-R,R]$ for some $R>0$.
Hence by Fubini's Theorem, we can swap the order of integration, and then perform the following changes of variables:
\begin{align}
\label{eq:InterDseStep4}
&\<I_{\epsilon,\nu} \dse_{\epsilon,-\nu,n}, \phi\>_{\epsilon,\nu}\\
&= -\epsilon i \lim_{m \to \infty} \int_{-\infty}^\infty \int_{-m}^m \sgn(x-t)^{\epsilon / 2} |x-t|^{\nu-1}  \phi_0(t) e(nx) \, dx \, dt\notag\\
&= -\epsilon i \lim_{m \to \infty} \int_{-\infty}^\infty \int_{-m-t}^{m-t} \sgn(x)^{\epsilon / 2} |x|^{\nu-1}  \phi_0(t) e(nx + nt) \, dx \, dt\notag\\
&= -\epsilon i \lim_{m \to \infty} \int_{-\infty}^\infty \phi_0(t) e(nt) \int_{-m-t}^{m-t} \sgn(x)^{\epsilon / 2} |x|^{\nu-1} e(nx) \, dx \, dt\notag\\
&= -\epsilon i \lim_{m \to \infty} \int_{-\infty}^\infty \phi_0(t) e(nt) |n|^{-\nu} \int_{-|n|(m+t)}^{|n|(m-t)} \sgn(nx)^{\epsilon / 2} |x|^{\nu-1} e(x) \, dx \, dt\notag\\
&= -\epsilon i |n|^{-\nu} \lim_{m \to \infty} \int_{-\infty}^\infty \phi_0(t) e(nt) h_m(t) \, dt,\notag
\end{align}
where
\begin{equation}
h_m(t) = \int_{-|n|(m+t)}^{|n|(m-t)} \sgn(nx)^{\epsilon / 2} |x|^{\nu-1} e(x) \, dx.
\end{equation}
For $m>2R$ and $t \in [-R,R]$, we have that $m-t$, $m+t > R$.
Thus
\begin{equation}
\label{eq:hmPrimeBound}
|h_m'(t)| \leq |n|^\nu \(|m-t|^{\nu-1} + |m+t|^{\nu-1}\) \ll 1,
\end{equation}
with the implied constant independent of $m$, but dependent upon $n$ and $\nu$.
This bound follows since $\Re(\nu)<1$.
By Lemma \ref{lem:InterIdentGeps1eps2},
\begin{equation}
\label{eq:hmLimit}
\lim_{m \to \infty} h_m(t) = G_{(-\epsilon,-\sgn(n))}(\nu).
\end{equation}
Since $m \mapsto h_m(0)$ is a continuous function, $(m \mapsto h_m(0)) \ll 1$,
with the implied constant independent of $m$.
By this and \eqref{eq:hmPrimeBound}, we have that
\begin{equation}
\label{eq:hmBound}
|h_m(t)| \ll |t|,
\end{equation}
with the implied constant independent of $m$.

By \eqref{eq:hmBound} and \eqref{eq:hmLimit},
we can pass the limit in $m$ past the integral given in the last line of \eqref{eq:InterDseStep4}.
Hence
\begin{equation}
\label{eq:InterDseStep5}
\<I_{\epsilon,\nu} \dse_{\epsilon,-\nu,n}, \phi\>_{\epsilon,\nu}
= -\epsilon i  G_{(-\epsilon,-\sgn(n))}(\nu) |n|^{-\nu} \int_{-\infty}^\infty \phi_0(t) e(nt)  \, dt.
\end{equation}
We can further simplify \eqref{eq:InterDseStep5} by invoking \eqref{eq:Geps1eps2Id1}.
This, together with the fact that
$\<\dse_{\epsilon,\nu,n}, \phi\>_{\epsilon,\nu} = \int_{-\infty}^\infty \phi_0(t) e(nt)  \, dt$,
shows that
\begin{equation}
\label{eq:InterDseStep6}
\<I_{\epsilon,\nu} \dse_{\epsilon,-\nu,n}, \phi\>_{\epsilon,\nu}  = G_{\epsilon,\sgn(n)}(\nu) |n|^{-\nu} \<\dse_{\epsilon,\nu,n}, \phi\>_{\epsilon,\nu}
\end{equation}
for all $\phi \in V_{-\epsilon,-\nu}^\infty$ such that $\supp(\phi) \subset \tlN \tlB$.

By suitably modifying the above argument, one can show that \eqref{eq:InterDseStep6} also holds for
$\phi$ with $\supp(\phi) \subset \tls^{-1} \tlN \tlB$.%
\footnote{One merely replaces $\phi_0$ with $\phi_\infty$ and verifies that \eqref{eq:InterDseStep2} holds for $\phi_\infty$.}
Thus by \eqref{eq:phi1phi2breakup} (and the surrounding discussion), \eqref{eq:InterDseStep6} holds for all $\phi \in V_{-\epsilon,-\nu}^\infty$.
The lemma then follows from the non-degeneracy of the pairing $\<\cdot,\cdot\>_{\epsilon,\nu}$.
\end{proof}

The functional equation for the metaplectic Eisenstein distributions (given in Theorem \ref{thm:FuncEq})
will include $I_{\epsilon,\nu} \tlE^{(\infty)}_{\epsilon,-\nu}$ as one of its terms. 
The following proposition gives a series expansion for $I_{\epsilon,\nu} \tlE^{(\infty)}_{\epsilon,-\nu}$ similar to the
one given for $\tlE_{\epsilon,\nu}^{(\infty)}$ in Proposition \ref{prop:EisenInfMeroCont}.

\begin{prop}
\label{prop:InterEisenInfMeroCont}
$\displaystyle I_{\epsilon,\nu} \tlE^{(\infty)}_{\epsilon,-\nu} = \sum_{n \in \Z} c_{\epsilon,\nu}(n) \dse_{\epsilon,\nu,n}
+ c_{\epsilon,\nu}(\infty) \delta_{\epsilon,\nu,\infty},$
where
\begin{align*}
&c_{\epsilon,\nu}(n) =  G_{(\epsilon,\sgn(n))}(\nu) |n|^{-\nu} a_{\epsilon,-\nu}(n) \text{ for } n \neq 0,\\
&c_{\epsilon,\nu}(0) = (1-2^{2\nu-1}) G_0(2\nu) \zeta(2 \nu), \text{ and}\\
&c_{\epsilon,\nu}(\infty) =  (1-\epsilon i) 2^{2\nu-1} G_0(2\nu) \zeta(2\nu+1).
\end{align*}
\end{prop}

\begin{proof}
We apply $I_{\epsilon,\nu}$ to the series expansion for $\tlE^{(\infty)}_{\epsilon,-\nu}$ given in Proposition \ref{prop:EisenInfMeroCont}.
By Lemma \ref{lem:InterTwinDeltInfOnebb} and Lemma \ref{lem:InterTwinDse},
\begin{equation*}
I_{\epsilon,\nu} \tlE^{(\infty)}_{\epsilon,-\nu} = \sum_{n \in \Z} c_{\epsilon,\nu}(n) \dse_{\epsilon,\nu,n}
+ c_{\epsilon,\nu}(\infty) \delta_{\epsilon,\nu,\infty},
\end{equation*}
where
\begin{align*}
&c_{\epsilon,\nu}(n) =  G_{(\epsilon,\sgn(n))}(\nu) |n|^{-\nu} a_{\epsilon,-\nu}(n) \text{ for } n \neq 0,\\
&c_{\epsilon,\nu}(\infty) = \frac{2 \pi \epsilon i \cot(\pi \nu)}{\nu} a_{\epsilon,-\nu}(0), \text{ and}\\
&c_{\epsilon,\nu}(0) = a_{\epsilon,-\nu}(\infty).
\end{align*}
Notice that $c_{\epsilon,\nu}(n)$ for $n \neq 0$ satisfies the equality asserted
in the statement of the proposition.
However, more justification is required for our statements regarding $c_{\epsilon,\nu}(0)$
and $c_{\epsilon,\nu}(\infty)$.

In what follows, we shall use the following identities, the latter of which is the functional equation for
the Riemann zeta function:
\begin{align}
\label{eq:GammaZetaIds}
&\frac{\pi \cot(\pi \nu)}{\nu} G_0(2\nu + 1) = - G_0(2\nu),\\
&\zeta(1-\nu) = G_0(\nu) \zeta(\nu).\notag
\end{align}
By these identities and Proposition \ref{prop:EisenInfConst}, it follows that
\begin{align*}
c_{\epsilon,\nu}(\infty) &= \frac{2 \pi \epsilon i \cot(\pi \nu)}{\nu} a_{\epsilon,-\nu}(0)\\ 
&=  \frac{2 \pi \epsilon i \cot(\pi \nu)}{\nu} (1+\epsilon i) 2^{2 \nu - 2} \zeta(-2\nu) \notag \\
&= \epsilon i (1+\epsilon i) 2^{2\nu - 1} \frac{\pi \cot(\pi \nu)}{\nu}  G_0(2 \nu + 1)  \zeta(2 \nu + 1) \notag \\
&= - \epsilon i (1+\epsilon i) 2^{2\nu-1} G_0(2\nu) \zeta(2\nu+1) \notag\\
&=  (1-\epsilon i) 2^{2\nu-1} G_0(2\nu) \zeta(2\nu+1). \notag
\end{align*}
Similarly, by \eqref{eq:GammaZetaIds} and Proposition \ref{prop:EisenInfMeroCont}, it follows that
\begin{align*}
c_{\epsilon,\nu}(0) &= a_{\epsilon,-\nu}(\infty) = (1-2^{2\nu-1}) \zeta(1 -2 \nu) \\
&= (1-2^{2\nu-1}) G_0(2\nu) \zeta(2 \nu)\notag.
\end{align*}
From these computations we obtain the formulas for $c_{\epsilon,\nu}(0)$ and $c_{\epsilon,\nu}(\infty)$ given in the lemma.
\end{proof}

We will also need to have explicit formulas for the Fourier coefficients of 
$I_{\epsilon,\nu} \tlE_{\epsilon,-\nu}^{(0)}$.
Recall that $\tlE_{\epsilon,\nu}^{(0)}$ was defined in \eqref{eq:MetaEisenAtZeroDef} and has a Fourier
series given in \eqref{eq:ComplFourSerEisAtZero}.

\begin{prop} 
\label{prop:InterEisenZeroMeroCont}
$\displaystyle I_{\epsilon,\nu} \tlE^{(0)}_{\epsilon,-\nu} = \sum_{n \in \Z} d_{\epsilon,\nu}(n) \dse_{\epsilon,\nu,n}
+ d_{\epsilon,\nu}(\infty) \delta_{\epsilon,\nu,\infty},$
where
\begin{align*}
&d_{\epsilon,\nu}(n) =  G_{(\epsilon,\sgn(n))}(\nu) |n|^{-\nu} b_{\epsilon,-\nu}(n) \text{ for } n \neq 0,\\
&d_{\epsilon,\nu}(0) = 0, \text{ and}\\
&d_{\epsilon,\nu}(\infty) = 2^{\nu} (1-2^{2\nu}) G_0(2\nu) \zeta(2 \nu + 1).
\end{align*}
\end{prop}

\begin{proof}
We shall only justify the statements regarding $d_{\epsilon,\nu}(0)$ and $d_{\epsilon,\nu}(\infty)$ since
the first statement follows just as in the proof of Proposition \ref{prop:InterEisenInfMeroCont}.
Once again, we will use the identities in \eqref{eq:GammaZetaIds} to prove our results.
By Lemma \ref{lem:InterTwinDeltInfOnebb} and Proposition \ref{prop:b0coeff},
\begin{align*}
I_{\epsilon,\nu}(b_{\epsilon,-\nu}(0) \onebb_{\epsilon,-\nu}) &= \frac{2 \pi \epsilon i \cot(\pi \nu)}{\nu} b_{\epsilon,-\nu}(0) \delta_{\epsilon,\nu,\infty}\\ 
&= \frac{2 \pi \epsilon i \cot(\pi \nu)}{\nu}  \epsilon  i 2^{\nu-1} (1-2^{2\nu}) \zeta(-2 \nu) \delta_{\epsilon,\nu,\infty}\\
&= \frac{2 \pi \epsilon i \cot(\pi \nu)}{\nu} \epsilon  i 2^{\nu-1} (1-2^{2\nu}) G_0(2\nu+1) \zeta(2 \nu + 1) \delta_{\epsilon,\nu,\infty}\\
&= (-2 \epsilon i) \epsilon  i 2^{\nu-1} (1-2^{2\nu}) G_0(2\nu) \zeta(2 \nu + 1) \delta_{\epsilon,\nu,\infty}\\
&= 2^{\nu} (1-2^{2\nu}) G_0(2\nu) \zeta(2 \nu + 1) \delta_{\epsilon,\nu,\infty}.
\end{align*}
Similarly, by Lemma \ref{lem:InterTwinDeltInfOnebb} and \eqref{eq:bInfCoeff},
\begin{equation*}
I_{\epsilon,\nu}(b_{\epsilon,-\nu}(\infty) \delta_{\epsilon,-\nu,\infty}) = I_{\epsilon,\nu}(0) = 0 \cdot \onebb_{\epsilon,\nu}.\qedhere
\end{equation*}
\end{proof}

\section{A Functional Equation for  Metaplectic Eisenstein Distribution}
\label{sec:EisenDistFunctionalEq}

The following theorem gives a functional equation for metaplectic Eisenstein distributions,
the proof of which relies upon the calculation of the Fourier coefficients of 
$\tlE_{\epsilon,\nu}^{(\infty)}$  (Proposition \ref{prop:EisenInfConst}, Proposition \ref{prop:EisenInfMeroCont}), 
$\tlE_{\epsilon,\nu}^{(0)}$ (Proposition \ref{prop:b0coeff}, \eqref{eq:bInfCoeff}),
$I_{\epsilon,\nu} \tlE_{\epsilon,-\nu}^{(\infty)}$ (Proposition \ref{prop:InterEisenInfMeroCont}), and
$I_{\epsilon,\nu} \tlE_{\epsilon,-\nu}^{(0)}$ (Proposition \ref{prop:InterEisenZeroMeroCont}).

\begin{thm} \label{thm:FuncEq}
For $\epsilon \in \{\pm 1\}$ and $\nu \in \C$,
\begin{align*}
I_{\epsilon,\nu} \tlE^{(\infty)}_{\epsilon,-\nu} &= \Big((1-\epsilon i) 2^{2\nu-1}(1-2^{-2\nu-1})^{-1} G_0(2 \nu)\Big)\\
&\hh \cdot \(\tlE^{(\infty)}_{\epsilon,\nu} + (1-\epsilon i) 2^{-\nu}(1 - 2^{2\nu}) \tlE^{(0)}_{\epsilon,\nu}\),
\end{align*}
where $\tlE_{\epsilon,\nu}^{(\infty)}$ is defined in \eqref{eq:EisenInfDist}, 
$\tlE_{\epsilon,\nu}^{(0)}$ is defined in \eqref{eq:MetaEisenAtZeroDef}, and
$G_0(\nu)$ is defined in \eqref{eq:GammaFactorDefs}.
\end{thm}

We begin by presenting an outline for the proof of Theorem \ref{thm:FuncEq}.
First, we will show that the distribution
\begin{align}
\label{eq:tauDef}
\tau_{\epsilon,\nu} = I_{\epsilon,\nu} \tlE^{(\infty)}_{\epsilon,-\nu} &- \Big((1-\epsilon i) 2^{2\nu-1}(1-2^{-2\nu-1})^{-1} G_0(2 \nu)\Big)\\
&\hh \cdot \(\tlE^{(\infty)}_{\epsilon,\nu} + (1-\epsilon i) 2^{-\nu}(1 - 2^{2\nu}) \tlE^{(0)}_{\epsilon,\nu}\),\notag
\end{align}
is ``cuspidal'' for all $\nu$ (in a sense that will be defined shortly).
We will then argue from the discreteness of the spectrum of cusp forms that $\tau_{\epsilon,\nu} \equiv 0$ for almost all $\nu$.
The uniqueness of meromorphic continuation then establishes that $\tau_{\epsilon,\nu} \equiv 0$ for all $\nu$, 
from which Theorem \ref{thm:FuncEq} then follows.
To make all this rigorous, we will need to recall the notion of an \textit{automorphic distribution}.
See \cite{Miller06, Miller04, Miller12b, Miller08} for more details on this subject.
Since we will be considering automorphic distributions that are automorphic with respect to $\tlGamma_1(4)$ (defined
in \eqref{eq:tlGammaDef}), we will also need to comment on the relationship between automorphic distributions
and the cusps associated with $\tlGamma_1(4)$.

In what follows, we recall the definition for automorphic distribution as stated in \cite[\S 2]{Miller12b}.
Let $(\pi,V)$ be an admissible, finite length representation of $\tlSL_2(\R)$, where $V$ is a reflexive Banach space.
One can show that the dual representation $(\pi',V')$ is also
an admissible, finite length representation of $\tlSL_2(\R)$, with $V'$ a reflexive Banach space.
Let $(V')^\infty$ denote the space of smooth vectors of $(\pi',V')$, equipped with the topology inherited from $C^\infty(G,V')$
via the inclusion $v' \mapsto (\tlg \mapsto \pi'(\tlg)v')$.
Let $V^{-\infty}$ denote the space of distributions on $(V')^\infty$.
We identify elements of $V$ with elements of $(V')'$ in the usual way.
Thus, as linear functionals on $(V')^\infty$,
\begin{equation}
V^\infty \subset V \subset V^{-\infty}.
\end{equation}
In this way we are able to think of distributions as generalized functions in the sense of Gelfand \textit{et al}.
The dual action of $(\pi',(V')^\infty)$ defines an action of $\tlSL_2(\R)$ on $V^{-\infty}$, which we denote by $\pi$. 
When working with $V^{-\infty}$, one typically equips it with the weak* topology.  
However we shall use the finer topology known as the
\textit{strong distribution topology} \cite[\S 19]{Treves06}
that was previously discussed in Section \ref{sec:DblCvrsSL2}. 
One can check that the action $\pi$ on $V^{-\infty}$
is (strongly) continuous with respect to the strong distribution topology \cite[\S A]{Miller12b} and that
$V^\infty$ (and hence also $V$) is dense in $V^{-\infty}$ with respect to this topology.
Let $(V^{-\infty})^{\tlGamma_1(4)}$ denote the subspace of $V^{-\infty}$ 
consisting of elements that are $\tlGamma_1(4)$-invariant under the action of $\pi$.
Using the terminology established in \cite{Miller12b}, elements of $(V^{-\infty})^{\tlGamma_1(4)}$ are 
called \textit{automorphic distributions}.
This definition of automorphic distribution slightly generalizes the definition of automorphic distribution that
appears in earlier papers on this topic \cite{Miller06, Miller08}.

Recall that $\tlgamma \in \tlGamma_1(4)$, or rather its $\SL_2$-coordinate $\gamma$, acts upon $\R \cup \{\infty\}$
as a linear fractional transformation. To each cusp $\fa \in \Q$ of $\tlGamma_1(4)$, we let
$\tlGamma_\fa$ denote the stabilizer of $\fa$ in $\tlGamma_1(4)$ under this action.
It is well-known that $\tlGamma_\fa$ is a cyclic group, and hence $\tlGamma_\fa$ has a generator $\tlgamma_\fa \in \tlGamma_1(4)$.
The cusps of $\tlGamma_1(4)$ are $0$, $\frac{1}{2}$, and $\infty$.  
For each cusp $\fa$, there exists $\tlxi_\fa \in \tlSL_2(\R)$ such that 
\begin{equation}
\label{eq:scalingMatrix}
\xi_\fa \cdot \infty = \fa \text{ and } \tlxi_\fa^{-1} \tlgamma_\fa^r \tlxi_\fa = \tln_1 = \(\mm{1}{1}{0}{1},1\),
\end{equation}
for some $r \in \Z_{> 0}$.
For $\fa = \infty$, we have $\tlxi_\infty = \(\id,1\)$.
For $\fa = 0$, we use $\tlxi_{0}$ as defined in \eqref{eq:tlxiZeroDef}.
It is straightforward to check that $\tlxi_\infty$ and $\tlxi_0$ both satisfy \eqref{eq:scalingMatrix},
and that in particular, $r = 1$ for both of these cusps.
As for $\fa = \frac{1}{2}$, one can compute that
\begin{equation*}
\tlGamma_{\frac{1}{2}} = \left\{\(\mm{1-4n}{2n}{-8n}{1+4n},\(\frac{-8n}{1+4n}\)\),n \in \Z\right\} = \left\<\(\mm{-3}{2}{-8}{5},-1\)\right\>.
\end{equation*}
In this case, \eqref{eq:scalingMatrix} holds with
\begin{equation}
\tlxi_{\frac{1}{2}} = \(\mm{2}{0}{4}{2^{-1}},1\),\h \tlgamma_{\frac{1}{2}} = \(\mm{-3}{2}{-8}{5},-1\), \text{ and } r = 2.
\end{equation}
One can show that $r=2$ is the minimal value of $r$ for which \eqref{eq:scalingMatrix} holds for $\fa = \frac{1}{2}$.%

Recall that $V_{\epsilon,\nu}$ comes equipped with
the action of $\tlSL_2(\R)$ by the left regular representation, which we denote by $\pi$.
Observe $\(\pi,V_{\epsilon,\nu}\)$ is an admissible representation of finite length, with $V_{\epsilon,\nu}$ a reflexive Banach space.
Let $\tau \in V_{\epsilon,\nu}^{-\infty}$ such that $\tau$ is $\tlGamma_1(4)$-invariant,
so that $\tau$ is an automorphic distribution for $\tlGamma_1(4)$.
We say that $\tau$ is \textit{cuspidal} at $\fa$ if
\begin{equation}
\label{eq:CuspDistIntegralCond1}
\int_0^1  \pi\(\tln_x^{-1} \tlxi_\fa^{-1}\) \tau \, dx = 0.
\end{equation}
Notice that the integral in \eqref{eq:CuspDistIntegralCond1} is an integral of a $V_{\epsilon,\nu}^{-\infty}$-valued function.
This is well-defined since $V_{\epsilon,\nu}^{-\infty}$ is a complete, locally convex,
Hausdorff topological vector space with $\tlSL_2(\R)$ acting strongly continuously \cite[\S A]{Miller12b}.
Recall that the dual representation of $\(\pi,V_{\epsilon,\nu}\)$ can be identified via the pairing 
$\< \cdot , \cdot \>_{\epsilon,\nu}$
(defined in \eqref{eq:pairingDef}) as $\(\pi',V_{-\epsilon,-\nu}\)$,
where $\pi'$ denotes the left regular representation.
Thus by \cite[(2.15)]{Miller12b},
\begin{equation}
F_{\tau,\phi}(\tlg) = \<\pi(\tlg^{-1})\tau, \phi\>_{\epsilon,\nu} = \<\tau, \pi'(\tlg)\phi\>_{\epsilon,\nu}
\end{equation}
is a smooth automorphic form for $\tlGamma_1(4)$, for any $\phi \in V_{-\epsilon,-\nu}^\infty$.
Recall that a smooth automorphic form $F$ is cuspidal at $\fa$ if
\begin{equation}
\int_0^1 F(\tlxi_\fa \tln_x \tlg) \, dx = 0,
\end{equation}
for all $\tlg \in \tlSL_2(\R)$.
Thus if $\tau$ is cuspidal at $\fa$ then so also $F_{\tau,\phi}$ must be cuspidal at $\fa$.
We say that $\tau$ is a \textit{cuspidal automorphic distribution} if
\eqref{eq:CuspDistIntegralCond1} holds for all cusps $\fa$.

In the classical setting, a cusp $\fa$ is said to be \textit{singular} relative to a multiplier system,
if and only if the multiplier system is trivial on the stabilizer for that cusp \cite[(2.65)]{Iwaniec97}.
Translated to the present setting,
a cusp $\fa$ is \textit{singular} if and only if the second coordinate of $\tlgamma_\fa$ is equal to $1$.
Thus for $\tlGamma_1(4)$, we find that $\infty$ and $0$ are singular cusps (see \eqref{eq:tlGamma0Def}), while $\frac{1}{2}$ is non-singular.
In the classical setting, every automorphic form (relative to a multiplier system) is automatically cuspidal at the non-singular cusps
\cite[\S 2.7]{Iwaniec97}.
The same is true in our setting, and hence in particular, $\tau$ is automatically cuspidal at $\fa = \frac{1}{2}$.
This claim can also be seen directly.
Indeed, since
\begin{equation}
\tlxi_{\frac{1}{2}} \tln_{\frac{1}{2}} = \tlgamma_{\frac{1}{2}} \tlxi_{\frac{1}{2}} \tlm_{1,-1},
\end{equation}
it follows from \eqref{eq:TransLawsB} and the fact that $\tlm_{-1,-1} = (-\id,-1) \in \tlM$, which
is the center of $\tlSL_2(\R)$, that
\begin{align}
&\int_0^1 \tau(\tlxi_{\frac{1}{2}} \tln_x \tlg) \, dx
= \int_0^{\frac{1}{2}} \tau(\tlxi_{\frac{1}{2}} \tln_x \tlg) \, dx + \int_0^{\frac{1}{2}} \tau(\tlxi_{\frac{1}{2}} \tln_{\frac{1}{2}} \tln_x \tlg) \, dx \\
&= \int_0^{\frac{1}{2}} \tau(\tlxi_{\frac{1}{2}} \tln_x \tlg) \, dx - \int_0^{\frac{1}{2}} \tau(\tlxi_{\frac{1}{2}} \tln_x \tlg) \, dx = 0. \notag
\end{align}
Hence $\tau$ is automatically cuspidal at $\fa = \frac{1}{2}$ as claimed.

In \eqref{eq:dseDef}, we defined $\dse_{\epsilon,\nu,n} \in V_{\epsilon,\nu}^{-\infty}$
such that $\(\dse_{\epsilon,\nu,n}\)_0 = e(nx)$ for all $n \in \Z_{\neq 0}$.%
\footnote{Recall that if $f \in V_{\epsilon,\nu}^{-\infty}$ then $f_0$ simply denotes the restriction of $f$ to $\tlN$.}
Since $\(\pi\(\tlxi_\fa^{-1}\) \tau\)_0$ is periodic by \eqref{eq:scalingMatrix},
it follows that
\begin{equation*}
\(\pi\(\tlxi_\fa^{-1}\) \tau\)_0 = \(\sum_{n \in \Z} p_{\epsilon,\nu,\fa}(n) \dse_{\epsilon,\nu,n}\)_0,
\end{equation*}
where
\begin{equation*}
p_{\epsilon,\nu,\fa}(n) = \int_0^1 \(\pi_{\epsilon,\nu}\(\tlxi_\fa^{-1}\) \tau \)_0(x) e\(-n x\) \, dx.
\end{equation*}
Thus
\begin{equation}
\pi\(\tlxi_\fa^{-1}\) \tau = \sum_{n \in \Z} p_{\epsilon,\nu,\fa}(n) \dse_{\epsilon,\nu,n} + \sigma_{\epsilon,\nu,\fa}
\end{equation}
where
$\sigma_{\epsilon,\nu,\fa}$ is a distribution supported in $\tls^{-1} \tlB$.
Since $\pi(\tln_{x}^{-1}) \dse_{\epsilon,\nu,n} = e(nx) \dse_{\epsilon,\nu,n}$,%
\footnote{To see that this is the case, observe that  $\(\pi(\tln_{x}^{-1}) \dse_{\epsilon,\nu,n}\)_0 = \(e(nx) \dse_{\epsilon,\nu,n}\)_0$
on $\R$, and hence by \eqref{eq:sigmaIndRepUnbddEq}, $\(\pi(\tln_{x}^{-1}) \dse_{\epsilon,\nu,n}\)_\infty = \(e(nx) \dse_{\epsilon,\nu,n}\)_\infty$
on $\R_{\neq 0}$. One can check that both distributions in this latter equality vanish to non-negative order at $0$.
Thus by \cite[Lemma 2.8]{Miller04}, they must in fact be equal on all of $\R$.}
it follows that
\begin{equation}
\int_0^1 \pi\(\tln_x^{-1}\) \dse_{\epsilon,\nu,n} \, dx = \( \int_0^1 e(nx) \, dx \) \dse_{\epsilon,\nu,n} = 0 
\text{ for } n \neq 0.
\end{equation}
Also observe that $\int_0^1 \pi\(\tln_x^{-1}\) \onebb_{\epsilon,\nu} \, dx$ is supported on $\tlSL_2(\R)$ while
$\int_0^1 \pi\(\tln_x^{-1}\) \sigma_{\epsilon,\nu,\fa} \, dx$ is supported on $\tls^{-1} \tlB$.
Thus by \eqref{eq:CuspDistIntegralCond1},
\begin{equation}
\label{eq:CuspDistIntegralCond3}
p_{\epsilon,\nu,\fa}(0)=0 \text{ and } \sigma_{\epsilon,\nu,\fa} = 0
\end{equation}
if and only if $\tau$ is cuspidal at a singular cusp $\fa$.
As was mentioned earlier, we will prove Theorem \ref{thm:FuncEq} by showing that $\tau_{\epsilon,\nu}$ 
(defined in \eqref{eq:tauDef}) is a cuspidal automorphic distribution.
By the above discussion, it suffices to show that $\tau_{\epsilon,\nu}$ is cuspidal at the cusps $0$ and $\infty$.

\begin{proof}[Proof of Theorem \ref{thm:FuncEq}]
We begin by showing that $\tau_{\epsilon,\nu}$ is cuspidal at the cusp $\infty$.
Since $\tlxi_\infty = \tlm_{1,1} = (\id,1)$, it follows from Proposition \ref{prop:EisenInfMeroCont},
Proposition \ref{prop:EisenInfConst}, \eqref{eq:ComplFourSerEisAtZero}, 
 Proposition \ref{prop:b0coeff}, and Proposition \ref{prop:InterEisenInfMeroCont},
\begin{align*}
&p_{\epsilon,\nu,\infty}(0)\\
&= c_{\epsilon,\nu}(0) - \Big((1-\epsilon i) 2^{2\nu-1}(1-2^{-2\nu-1})^{-1} G_0(2 \nu)\Big)\notag\\
&\hh \cdot \(a_{\epsilon,\nu}(0) + (1-\epsilon i) 2^{-\nu}(1 - 2^{2\nu}) b_{\epsilon,\nu}(0)\)\notag\\
&=(1-2^{2\nu-1}) G_0(2\nu) \zeta(2 \nu) - \Big((1-\epsilon i) 2^{2\nu-1}(1-2^{-2\nu-1})^{-1} G_0(2 \nu)\Big)\notag\\
&\hh \cdot \((1+\epsilon i) 2^{-2\nu-2} \zeta(2\nu) + (1-\epsilon i) 2^{-\nu}(1 - 2^{2\nu})\epsilon i 2^{-\nu-1} (1-2^{-2\nu}) \zeta(2\nu)\)
= 0.
\end{align*}
Likewise, by Proposition \ref{prop:EisenInfMeroCont},  \eqref{eq:bInfCoeff}, and Proposition \ref{prop:InterEisenInfMeroCont},
\begin{align*}
&\sigma_{\epsilon,\nu,\infty}\\
&= c_{\epsilon,\nu}(\infty) \delta_{\epsilon,\nu,\infty} - \Big((1-\epsilon i) 2^{2\nu-1}(1-2^{-2\nu-1})^{-1} G_0(2 \nu)\Big)\notag\\
&\hh \cdot \(a_{\epsilon,\nu}(\infty) + (1-\epsilon i) 2^{-\nu}(1 - 2^{2\nu}) b_{\epsilon,\nu}(\infty)\) \delta_{\epsilon,\nu,\infty} \notag\\
&= (1-\epsilon i) 2^{2\nu-1} G_0(2\nu) \zeta(2\nu+1) \delta_{\epsilon,\nu,\infty}\\
&\hh - \Big((1-\epsilon i) 2^{2\nu-1} (1-2^{-2\nu-1})^{-1} G_0(2 \nu)\Big) \cdot (1-2^{-2\nu-1}) \zeta(2\nu + 1) \delta_{\epsilon,\nu,\infty}
= 0.
\end{align*}
Thus $\tau_{\epsilon,\nu}$ is cuspidal at the cusp $\infty$ by \eqref{eq:CuspDistIntegralCond3}.

It remains to show that $\tau$ is cuspidal at the cusp $0$.  
Recall that in \eqref{eq:tlxiZeroDef} we defined
\begin{equation*}
\tlxi_{0} =\tla_2^{-1} \tls = \tls \, \tla_2 = \(\mm{0}{-2^{-1}}{2}{0},1\).
\end{equation*}
Since $\tlxi_0 (\tlm_{-1,-1})^{-1} = \tlxi_0 (-\id,-1)^{-1} = \tlxi_0^{-1}$,
then by \eqref{eq:EZeroToEInf} and \eqref{eq:TransLawsB},
\begin{equation}
\pi(\tlxi_0^{-1}) \tlE^{(\infty)}_{\epsilon,\nu} = -\epsilon i \tlE^{(0)}_{\epsilon,\nu} \text{ and } \pi(\tlxi_0^{-1}) \tlE^{(0)}_{\epsilon,\nu} = \tlE^{(\infty)}_{\epsilon,\nu}.
\end{equation}
Thus
\begin{align}
\pi\(\tlxi_0^{-1}\) \tau = -\epsilon i I_{\epsilon,\nu} \tlE^{(0)}_{\epsilon,-\nu}  &- \Big((1-\epsilon i) 2^{2\nu-1}(1-2^{-2\nu-1})^{-1} G_0(2 \nu)\Big)\\
&\hh \cdot \(-\epsilon i \tlE^{(0)}_{\epsilon,\nu} + (1-\epsilon i) 2^{-\nu}(1 - 2^{2\nu}) \tlE^{(\infty)}_{\epsilon,\nu}\).\notag
\end{align}
By Proposition \ref{prop:EisenInfConst}, Proposition \ref{prop:b0coeff}, and Proposition \ref{prop:InterEisenZeroMeroCont}, 
\begin{align*}
&p_{\epsilon,\nu,0}(0)\\
&= -\epsilon i d_{\epsilon,\nu}(0) - \Big((1-\epsilon i) 2^{2\nu-1}(1-2^{-2\nu-1})^{-1} G_0(2 \nu)\Big)\\
&\hh \cdot \(- \epsilon ib_{\epsilon,\nu}(0) + (1-\epsilon i) 2^{-\nu}(1 - 2^{2\nu}) a_{\epsilon,\nu}(0)\)\\
&= -\Big((1-\epsilon i) 2^{2\nu-1}(1-2^{-2\nu-1})^{-1} G_0(2 \nu)\Big)\\
&\hh \cdot \(2^{-\nu-1} (1-2^{-2\nu}) \zeta(2 \nu) + (1-\epsilon i) 2^{-\nu}(1 - 2^{2\nu})  (1+\epsilon i) 2^{-2\nu-2} \zeta(2\nu)\)
= 0.
\end{align*}
Likewise, by Proposition \ref{prop:EisenInfMeroCont},  \eqref{eq:bInfCoeff}, and Proposition \ref{prop:InterEisenZeroMeroCont},
\begin{align*}
&\sigma_{\epsilon,\nu,0}\\
&= - \epsilon i d_{\epsilon,\nu}(\infty) \delta_{\epsilon,\nu,\infty} - \Big((1-\epsilon i) 2^{2\nu-1}(1-2^{-2\nu-1})^{-1} G_0(2 \nu)\Big)\\
&\hh \cdot \(-\epsilon i b_{\epsilon,\nu}(\infty) + (1-\epsilon i) 2^{-\nu}(1 - 2^{2\nu}) a_{\epsilon,\nu}(\infty)\) \delta_{\epsilon,\nu,\infty}\\
&=-\epsilon i 2^{\nu} (1-2^{2\nu}) G_0(2\nu) \zeta(2 \nu + 1) \delta_{\epsilon,\nu,\infty}\\
&\hh - \Big((1-\epsilon i) 2^{2\nu-1}(1-2^{-2\nu-1})^{-1} G_0(2 \nu)\Big)\\
&\hh \cdot \((1-\epsilon i) 2^{-\nu}(1 - 2^{2\nu})(1-2^{-2\nu-1}) \zeta(2\nu+1) \) \delta_{\epsilon,\nu,\infty}
= 0.
\end{align*}
Thus $\tau_{\epsilon,\nu}$ is a cuspidal automorphic distribution.
Since for almost all $\nu$, the only cuspidal automorphic distribution is $0$,
it follows that for almost all $\nu$ we have $\tau_{\epsilon,\nu} \equiv 0$.
The uniqueness of meromorphic continuation then shows that $\tau_{\epsilon,\nu} \equiv 0$ for all $\nu$.
\end{proof}

\section{A Functional Equation for  Metaplectic Eisenstein Series}
\label{sec:EisenSeriesFunctionalEq}

For $\ell \in 2 \Z$, the weight $\frac{1}{2} - \ell$ (normalized) classical metaplectic Eisenstein series is
\begin{align}
\label{eq:ClassicEisenDef}
&\dsE_{\fa,\ell}(z,\nu)\\
&= \zeta_2(4\nu - 1) \sum_{\gamma \in \Gamma_\fa \bs \Gamma_1(4)} \vartheta(\gamma) w_{\frac{1}{2}-\ell}(\xi_\fa,\xi_\fa^{-1} \gamma) \(\frac{j_{\tlxi_\fa^{-1} \gamma}(z)}{\left|j_{\tlxi_\fa^{-1} \gamma}(z)\right|}\)^{\ell-\frac{1}{2}} \(\Im \, \tlxi_\fa^{-1} \gamma  z\)^\nu,\notag
\end{align}
where 
$\fa$ is either the cusp $\infty$ or $0$, 
$\xi_\infty = \id$ and $\xi_0 = \mm{0}{-2^{-1}}{2}{0}$ (in compatibility with \eqref{eq:tlxiZeroDef}),
$z \in \dsH = \{z \in \C: \Im(z)>0\}$, $\nu \in \C$ with $\Re(\nu) > 1$,
$w_{\frac{1}{2}-\ell}(\xi_\fa,\xi_\fa^{-1} \gamma)$ is defined
in \cite[(2.48)]{Iwaniec97}, $\vartheta(\gamma) = \(\frac{c}{d}\)$,
and $j_\gamma(z) = cz + d$ for $\gamma = \mm{a}{b}{c}{d}$.
This definition is compatible with the definition given in \cite[(13.4)]{Iwaniec97} except for the fact that we have a normalizing factor of
$\zeta_2(4 \nu - 1)$, that $\vartheta$ in \eqref{eq:ClassicEisenDef} is the $\bar{\eta}$ in \cite[(13.4)]{Iwaniec97}, and that the
$\vartheta$  in \cite[(13.4)]{Iwaniec97} is set as the trivial multiplier system in \eqref{eq:ClassicEisenDef}.
It is well-known that $\dsE_{\fa,\ell}(z,\nu)$ can be meromorphically continued to the entire complex plane \cite{Goldfeld85,Moeglin95}.
It is easy to infer that $w_{\frac{1}{2}-\ell}(\xi_\fa,\xi_\fa^{-1} \gamma) = 1$ when $\fa = \infty$ \cite[(13.8)]{Iwaniec97}.
For $\fa = 0$, one can show that $w_{\frac{1}{2}-\ell}(\xi_\fa,\xi_\fa^{-1} \gamma) = 1$
whenever we have $\gamma$ such that $c>0$.%
\footnote{For such $\gamma$, there always exist $z \in \dsH$ such that $|\arg(cz+d)|, |\arg(-az-b)| < \frac{\pi}{2}$.
From this, one can show that $\arg(j_{\xi_0}(\xi_0^{-1} \gamma z))$, $\arg(j_{\xi_0^{-1} \gamma}(z))$, $\arg(j_\gamma(z))$
are all in $\(-\frac{\pi}{2},\frac{\pi}{2}\)$.
From the definitions of $\omega$ and $w_{\frac{1}{2} - \ell}$ given in \cite[\S 2.6]{Iwaniec97},
it then follows from \cite[(2.40)]{Iwaniec97} that $\omega(\xi_0,\xi_0^{-1} \gamma) = 0$
for all such $\gamma$.
This implies (by \cite[(2.48)]{Iwaniec97}) that $w_{\frac{1}{2}-\ell}(\xi_0,\xi_0^{-1} \gamma)= 1$
for all such $\gamma$.}
Given that \eqref{eq:ClassicEisenDef} is a defined as as summation over $\Gamma_0 \bs \Gamma_1(4)$,
we can always arrange for this to be the case.
Consequently, we find that \eqref{eq:ClassicEisenDef} simplifies to the following:
\begin{align}
\label{eq:ClassicEisenInfDef}
&\dsE_{\infty,\ell}(z,\nu) = \zeta_2(4\nu - 1) \sum_{\gamma \in \Gamma_\infty \bs \Gamma_1(4)} \vartheta(\gamma)  \(\frac{j_{\gamma}(z)}{\left|j_{\gamma}(z)\right|}\)^{\ell-\frac{1}{2}} \(\Im \, \gamma  z\)^\nu,\\
\label{eq:ClassicEisenZeroDef}
&\dsE_{0,\ell}(z,\nu) = \zeta_2(4\nu - 1) \sum_{\substack{\gamma \in \Gamma_0 \bs \Gamma_1(4)\\c > 0}} \vartheta(\gamma) \(\frac{j_{\tlxi_0^{-1} \gamma}(z)}{\left|j_{\tlxi_0^{-1} \gamma}(z)\right|}\)^{\ell-\frac{1}{2}} \(\Im \, \tlxi_0^{-1} \gamma  z\)^\nu,
\end{align}
where in the summation of \eqref{eq:ClassicEisenZeroDef}, $\gamma = \mm{a}{b}{c}{d}$.

We will show that Theorem \ref{thm:FuncEq} implies a functional equation for the classical
metaplectic Eisenstein series.  To do this, we will use Lemma \ref{lem:DistToSeries}, which relates metaplectic
Eisenstein distributions to these classical metaplectic Eisenstein series.
The following lemma will be needed in the proof of Lemma \ref{lem:DistToSeries}.

\begin{lem} \label{lem:deltaInvN}
For $\tln_t \in \tlN$, 
$\pi(\tln_t) \delta_{\epsilon,\nu,\infty} = \delta_{\epsilon,\nu,\infty}$.
\end{lem}

\begin{proof}
Consider $t \neq 0$.
By Lemma \ref{lem:MetaUnbddAct}(b) and \eqref{eq:tlsdeltprop}, we see that 
\begin{equation}
\(\pi(\tln_t) \delta_{\epsilon,\nu,\infty}\)_0(x) = \(\delta_{\epsilon,\nu,\infty}\)_0(x) = 0.
\end{equation}
By this and \eqref{eq:SmoothSigmaIndRepUnbddEq}, we then know that $\(\pi(\tln_t) \delta_{\epsilon,\nu,\infty}\)_\infty(x) = 0$ on $\R_{\neq 0}$.
Thus it remains to determine $\(\pi(\tln_t) \delta_{\epsilon,\nu,\infty}\)_\infty(x)$ on a neighborhood about $0$.
Let $\phi$ a test function on $\R_{\neq \frac{-1}{t}}$.
By Lemma \ref{lem:MetaUnbddAct}(c), 
\begin{equation*}
\(\pi(\tln_t) \delta_{\epsilon,\nu,\infty}\)_\infty(x)
=  (t,tx+1)_H |tx+1|^{\nu-1} \sgn(tx+1)^{\epsilon/2}  \(\delta_{\epsilon,\nu,\infty}\)_\infty\(\frac{x}{tx+1}\),
\end{equation*}
as an equality between distributions on $\R_{\neq \frac{-1}{t}}$.
Thus by \eqref{eq:tlsdeltprop} and changing variables,
\begin{align*}
&\int_{\R_{\neq \frac{-1}{t}}} \(\pi(\tln_t) \delta_{\epsilon,\nu,\infty}\)_\infty(x) \phi(x) \, dx\\
&=\int_{\R_{\neq \frac{-1}{t}}} (t,tx+1)_H |tx+1|^{\nu-1} \sgn(tx+1)^{\epsilon/2} \epsilon i \delta_0\(\frac{x}{tx+1}\) \phi(x) \, dx \\
&=\int_{\R_{\neq 0}} (t,tx)_H |tx|^{\nu-1} \sgn(tx)^{\epsilon/2} \epsilon i \delta_0\(\frac{x-t^{-1}}{tx}\) \phi(x-t^{-1}) \, dx \\
&=\int_{\R_{\neq 0}} (t,tx)_H |tx|^{\nu-1} \sgn(tx)^{\epsilon/2} \epsilon i \delta_0\(t^{-1} - t^{-2} x^{-1}\) \phi(x-t^{-1}) \, dx \\
&=\int_{\R_{\neq 0}} (t,-tx^{-1})_H |t|^{\nu-1} |x|^{-\nu-1} \sgn\(-tx^{-1}\)^{\epsilon/2} \epsilon i \delta_0\(t^{-1} + t^{-2} x\) \phi(-x^{-1}-t^{-1}) \, dx \\
&=\int_{\R_{\neq 0}} (t,-t^{-1}x^{-1})_H |t|^{-\nu-1} |x|^{-\nu-1} \sgn\(-t^{-1}x^{-1}\)^{\epsilon/2} \epsilon i \delta_0\(t^{-1} + x\)\\
&\hhh \cdot \phi(-t^{-2}x^{-1}-t^{-1}) \, dx \\
&=(t,t^{-2})_H |t|^{-\nu-1} |t^{-1}|^{-\nu-1} \sgn\(1\)^{\epsilon/2} \epsilon i \phi(0) = \epsilon i \phi(0)\\
&=\int_{\R_{\neq \frac{-1}{t}}} \(\delta_{\epsilon,\nu,\infty}\)_\infty(x) \phi(x) \, dx.
\end{align*}
This shows that
\begin{equation}
\(\pi(\tln_t) \delta_{\epsilon,\nu,\infty}\)_\infty(x) = \(\delta_{\epsilon,\nu,\infty}\)_\infty(x),
\end{equation}
thereby completing the proof.
\end{proof}

\begin{lem}
\label{lem:DistToSeries}
Let $\phi_{\epsilon,\nu,\ell} \in V_{\epsilon,\nu}^\infty$ such that \eqref{eq:phiDef} holds.
For $\phi = \phi_{-1,-\nu,-\ell}$, $z = x + iy \in \dsH$, and $\fa = 0$ or $\infty$,
the following identities hold:
\begin{itemize}[leftmargin=.3in,font=\normalfont\textbf]
\item[(a)] $\displaystyle \int_{-\pi/2}^{\pi/2} \tlE_{1,\nu}^{(\fa)}\(\tln_x \tla(y^{1/2}) \tlk_\theta \) \phi(\tlk_\theta) \, d\theta = i^{-\ell} \frac{1+i}{\sqrt{2}} \dsE_{\fa,\ell}\(z,\frac{\nu+1}{2} \)$,

\item[(b)] $\displaystyle \begin{aligned}[t]&\int_{-\pi/2}^{\pi/2} I_{1,\nu} \tlE_{1,-\nu}^{(\infty)}\(\tln_x \tla(y^{1/2}) \tlk_\theta \) \phi(\tlk_\theta) \, d\theta\\ &= \frac{\pi  2^{1-\nu} \Gamma (\nu )}{\Gamma\left( \frac{3}{4} - \frac{\ell}{2} + \frac{\nu}{2} \right) \Gamma \left( \frac{1}{4} + \frac{\ell}{2} + \frac{\nu}{2} \right)} \dsE_{\infty,\ell}\(z,\frac{-\nu+1}{2} \).\end{aligned}$
\end{itemize}
\end{lem}

\begin{proof}
By \eqref{eq:EisenInfDist} and \eqref{eq:MetaEisenAtZeroDef},
\begin{equation*}
\tlE_{1,\nu}^{(\fa)}(\tlg) = \zeta_2(2 \nu + 1) \sum_{\tlgamma \in \tlGamma_1(4) / \tlGamma_{\fa}} \pi\(\tlgamma \tlxi_{\fa}\) \delta_{1,\nu,\infty}(\tlg)
\end{equation*}
where $\fa = 0$ or $\infty$ and
\begin{equation*}
\tlxi_{0} = \(\mm{0}{-2^{-1}}{2}{0},1\) \text{ and } \tlxi_{\infty} = \tlm_{1,1} = (\id,1).
\end{equation*}
For $\Re(\nu)>1$, observe
\begin{align}
\label{eq:EisDistToSer1}
&\int_{-\pi/2}^{\pi/2} \tlE_{1,\nu}^{(\fa)} \(\tln_x \tla(y^{1/2}) \tlk_\theta \) \phi(\tlk_\theta) \, d\theta\\
&= \zeta_2(2\nu+1) \sum_{\tlgamma \in \tlGamma_1(4) / \tlGamma_\fa} \int_{-\pi/2}^{\pi/2} \pi(\tlgamma \tlxi_\fa) \delta_{1,\nu,\infty} \(\tln_x \tla(y^{1/2}) \tlk_\theta \) \phi(\tlk_\theta) \, d\theta\notag\\
&= \zeta_2(2\nu+1) \sum_{\tlgamma \in \tlGamma_1(4) / \tlGamma_\fa} \int_{-\pi/2}^{\pi/2} \delta_{1,\nu,\infty} \(\tlxi_\fa^{-1} \tlgamma^{-1} \tln_x \tla(y^{1/2}) \tlk_\theta \) \phi(\tlk_\theta) \, d\theta\notag\\
&= \zeta_2(2\nu+1) \sum_{\tlgamma \in \tlGamma_\fa \bs \tlGamma_1(4)} f_{\nu}\(\tlxi_\fa^{-1} \tlgamma \tln_x \tla(y^{1/2})\),\notag
\end{align}
where
\begin{equation}
\label{eq:fDef}
f_{\nu}(\tlg) =  \int_{-\pi/2}^{\pi/2} \delta_{1,\nu,\infty} \(\tlg \tlk_\theta \) \phi(\tlk_\theta) \, d\theta \text{ with }
\phi = \phi_{-1,-\nu,-\ell}.
\end{equation}
In order to evaluate $f_{\nu}\(\tlxi_\fa^{-1} \tlgamma \tln_x \tla(y^{1/2})\)$ we first evaluate 
the simpler expression $f_{\nu}\(\tln_x \tla(y^{1/2})\)$.
Since $\delta_{1,\nu,\infty}$ is under $\tlN$ by Lemma \ref{lem:deltaInvN},
it follows that 
\begin{equation}
\label{eq:fEvalStep1}
f_{\nu}\(\tln_x \tla(y^{1/2})\) = f_{\nu}\(\tla(y^{1/2})\).
\end{equation}
Since
$\tls^{-1} \tla(y^{1/2}) = \tla(y^{-1/2}) \tls^{-1}$
then it follows from \eqref{eq:fEvalStep1}, the fact that $\delta_{\epsilon,\nu,\infty} = \pi(\tls) \delta_{\epsilon,\nu,0}$,
and a change of variables, that
\begin{align}
\label{eq:fEvalStep1b}
&f_{\nu}\(\tln_x \tla(y^{1/2})\) = f_{\nu}\(\tla(y^{1/2})\) = \int_{-\pi/2}^{\pi/2} \delta_{1,\nu,0} \(\tls^{-1} \tla(y^{1/2}) \tlk_\theta \) \phi(\tlk_\theta) \, d\theta\\
&= \int_{-\pi/2}^{\pi/2} \delta_{1,\nu,0} \(\tla(y^{-1/2}) \tls^{-1} \tlk_\theta \) \phi(\tlk_\theta) \, d\theta
= \int_{-\pi/2}^{\pi/2} \delta_{1,\nu,0} \(\tla(y^{-1/2}) \tlk_\theta \) \phi(\tls \tlk_\theta) \, d\theta.\notag
\end{align}
In this last equality we changed variables in $\theta$ (so as to remove $\tls^{-1}$ into the argument of $\delta_{1,\nu,0}$), but
we did not change the domain of integration. This is justified by the discussion following \eqref{eq:pairingDef}.

Next, observe
\begin{align}
\label{eq:fEvalStep2}
&\tla(y^{-1/2}) \tlk_\theta = \tln\(\frac{-\tan(\theta)}{y}\) \tla\(\frac{|\sec(\theta)|}{y^{1/2}}\) \tln_-(\tan(\theta)),
\end{align}
for $\theta \in \(-\frac{\pi}{2}, \frac{\pi}{2}\)$.
It follows from \eqref{eq:fEvalStep2} and \eqref{eq:TransLawsB}, that
\begin{align}
\label{eq:fEvalStep2b}
&\delta_{1,\nu,0} \( \tla(y^{-1/2}) \tlk_\theta \) = \delta_{1,\nu,0} \( \tln\(\frac{-\tan(\theta)}{y}\) \) \left|\frac{\sec(\theta)}{\sqrt{y}}\right|^{-\nu+1},\\
&\phi(\tls \tlk_\theta) =  \phi\Big(\tls \, \tln(-\tan(\theta))\Big) \left| \sec(\theta) \right|^{\nu+1}.\notag
\end{align}
Thus by \eqref{eq:fEvalStep1b}, \eqref{eq:fEvalStep2b}, and a change of variables,
\begin{align}
\label{eq:fEval}
&f_{\nu}\(\tln_x \tla(y^{1/2})\)\\
&= \int_{-\pi/2}^{\pi/2} \delta_{1,\nu,0} \( \tln\(\frac{-\tan(\theta)}{y}\) \) \left|\frac{\sec(\theta)}{\sqrt{y}}\right|^{-\nu+1} \phi\Big(\tls \, \tln(-\tan(\theta))\Big) \left| \sec(\theta) \right|^{\nu+1} \, d\theta\notag\\
&= y^{\frac{\nu-1}{2}} \int_{-\infty}^{\infty} \delta_{1,\nu,0} \( \tln\(\frac{-t}{y}\) \) \left|\sqrt{1+t^2}\right|^{-\nu+1} \phi\(\tls \, \tln(-t)\) \left| \sqrt{1+t^2} \right|^{\nu+1} \, \frac{dt}{1+t^2}\notag\\
&= y^{\frac{\nu-1}{2}} \int_{-\infty}^{\infty} \delta_{1,\nu,0} \( \tln\(\frac{t}{y}\) \) \phi\(\tls \, \tln(t)\) \, dt\notag\\
&= y^{\frac{\nu-1}{2}+1} \int_{-\infty}^{\infty} \delta_{1,\nu,0} \( \tln_t \) \phi\(\tls \, \tln(t y) \) \, dt 
= y^{\frac{\nu+1}{2}} \phi(\tls)\notag\\
&= y^{\frac{\nu+1}{2}} \phi_{-1,-\nu,-\ell}(\tlk_{\frac{\pi}{2}}) = i^{-\ell} e^{\frac{\pi i}{4}} y^{\frac{\nu+1}{2}}\notag\\
&= i^{-\ell} \frac{1+i}{\sqrt{2}} \Im(z)^{\frac{\nu+1}{2}}, \notag
\end{align}
where $\tls$ is defined in \eqref{eq:tlsdef}.

Let
\begin{equation*}
\tlh = \(\mm{a_0}{b_0}{c_0}{d_0}, \kappa\) \in \tlSL_2(\R);
\end{equation*}
later we will specialize $\tlh$ to be $\tlxi_{\fa}^{-1} \tlgamma$.
Observe
\begin{equation}
\tlh \tln_x \tla(y^{1/2}) = \tln\( \Re\( h z \) \) \tla\( \Im\( h z \)^{1/2} \) \scK_{\tlh,z},
\end{equation}
where $h$ is the $\SL_2$ component of $\tlh$, $z = x + i y$, and
\begin{equation}
\scK_{\tlh,z} = \(\mm{\frac{c_0 x+d_0}{|c_0 z+d_0|}}{\frac{-c_0 y}{|c_0 z+d_0|}}{\frac{c_0 y}{|c_0 z+d_0|}}{\frac{c_0x+d_0}{|c_0z+d_0|}},\kappa\) \in \tlK.
\end{equation}
Thus by changing variables in \eqref{eq:fDef}, utilizing the fact that $\phi|_{\tlK}$ is a character,
and applying \eqref{eq:fEval}, we have
\begin{align}
\label{eq:fEval2}
&f_{\nu}\(\tlh \tln_x \tla(y^{1/2})\) = f_{\nu}\(\tln\( \Re\( h z \) \) \tla\( \Im\( h z \)^{1/2} \) \scK_{\tlh,z}\)\\
&=f_{\nu}\(\tln\( \Re\( h z \) \) \tla\( \Im\( h z \)^{1/2} \) \) \phi\(\scK_{\tlh,z}\)^{-1} \notag \\
&= i^{-\ell} \frac{1+i}{\sqrt{2}} \Im\( h z \)^{\frac{\nu+1}{2}}  \phi\(\scK_{\tlh,z}\)^{-1}. \notag
\end{align}
Since $\scK_{\tlh,z} \in \tlK$ then there exists $\theta \in [-2\pi,2\pi)$ such that $\scK_{\tlh,z} = \tlk_\theta$.
In particular, there exists $\theta' \in (-\pi, \pi)$ such that
\begin{equation*}
\sin(\theta') = \frac{c_0 y}{|c_0 z+d_0|} \text{ and } \cos(\theta') = \frac{c_0 x+d_0}{|c_0z+d_0|},
\end{equation*}
provided $\tlh = \tlxi_{\fa}^{-1} \tlgamma$ 
where $\fa = 0$ or $\infty$.%
\footnote{Notice that we have excluded $-\pi$ from range of possible values for $\theta'$. 
If we have $\tlh = \tlxi_\infty^{-1} \tlgamma$ for $\displaystyle \tlgamma = \(\mm{a}{b}{c}{d},\(\frac{c}{d}\)\) \in \tlGamma_1(4)$
then $c_0 = c$. In this case, if $\theta' = -\pi$ then we must have $c=0$ which then implies $a=d=1$ since $ad-bc=1$ and $a \equiv d \equiv 1 (\mod 4)$.
We would then have that $\scK_{\tlh,z} = \(\id,*\)$, but this clearly does not equal $\tlk_{-\pi}$.
If on the other hand we have $\tlh = \tlxi_0^{-1} \tlgamma$ then $c_0  = -2a \neq 0$ since $a \equiv 1 (\mod 4)$, so clearly we
cannot have $\theta' = -\pi$ in this case either.}
It then follows that $\theta = (1 - \kappa)\pi + \theta'$ by \eqref{eq:epsDef}.
Since $\phi(\tlk_\theta) = \phi_{-1,-\nu,-\ell}(\tlk_\theta) = \exp\(\(\frac{1}{2}-\ell\) i \theta\)$, $\ell \in 2\Z$, 
and $\kappa \in \{\pm 1\}$,
\begin{align*}
&\phi(\scK_{\tlh,z}) = \phi(\tlk_\theta) = \exp\(\(\frac{1}{2}-\ell\) i  (1 - \kappa)\pi) \) \exp\(\(\frac{1}{2}-\ell\) i \theta' \)\\
&= \kappa \exp\(\frac{i \theta'}{2}\) \exp\(-\ell i \theta' \).
\end{align*}
By \eqref{eq:branchCut},
\begin{equation*}
\exp\(i \theta' \)^{\frac{1}{2}} = \exp\(\frac{1}{2} \log(\exp\(i \theta' \) ) \) = \exp\(\frac{i \theta'}{2} \).
\end{equation*}
Thus
\begin{align}
&\phi(\scK_{\tlh,z}) = \kappa \exp\(i \theta' \)^{\frac{1}{2}} \exp\(i \theta' \)^{-\ell} =  \kappa \exp\(i \theta' \)^{\frac{1}{2} - \ell}\\
&= \kappa \(\cos(\theta') + i \sin(\theta')\)^{\frac{1}{2}-\ell} = \kappa \(\frac{c_0z+d_0}{|c_0z+d_0|}\)^{\frac{1}{2}-\ell}\notag
= \kappa \(\frac{j_{h}(z)}{|j_h(z)|}\)^{\frac{1}{2}-\ell}.\notag
\end{align}
Therefore by \eqref{eq:fEval2},
\begin{align}
\label{eq:fEval3}
&f_{\nu}\(\tlh \tln_x \tla(y^{1/2})\) = i^{-\ell} \frac{1+i}{\sqrt{2}} \Im\( h z \)^{\frac{\nu+1}{2}} \kappa \(\frac{j_{h}(z)}{|j_h(z)|}\)^{\ell - \frac{1}{2}},
\end{align}
where $j_h(z) = c_0 z + d_0$.

Let $\tlgamma = \(\mm{a}{b}{c}{d},\(\frac{c}{d}\)\)$.
By \eqref{eq:EisDistToSer1} and \eqref{eq:fEval3} for $\tlh = \tlxi_\infty^{-1} \tlgamma = \tlgamma$,
\begin{align}
\label{eq:EisDistToSer1b}
&\int_{-\pi/2}^{\pi/2} \tlE_{1,\nu}^{(\infty)}\(\tln_x \tla(y^{1/2}) \tlk_\theta \) \phi(\tlk_\theta) \, d\theta\\
&= i^{-\ell} \frac{1+i}{\sqrt{2}} \zeta_2(2\nu+1) \sum_{\tlgamma \in \tlGamma_\infty \bs \tlGamma_1(4)} \(\frac{c}{d}\) \(\frac{j_{\gamma}(z)}{|j_\gamma(z)|}\)^{\ell - \frac{1}{2}} \Im\( \gamma z \)^{\frac{\nu+1}{2}}\notag\\
&= i^{-\ell} \frac{1+i}{\sqrt{2}} \dsE_{\infty,\ell}\(z, \frac{\nu+1}{2}\)\notag
\end{align}
for $\Re(\nu) > 1$. The uniqueness of meromorphic continuation proves part (a) for $\fa = \infty$.

For the case of $\fa = 0$, observe
\begin{equation}
\tlxi_0^{-1} \tlgamma = \(\mm{\frac{c}{2}}{\frac{d}{2}}{-2a}{-2b},(a,c)_H \(\frac{c}{d}\)\).
\end{equation}
Thus by \eqref{eq:EisDistToSer1} and \eqref{eq:fEval3} for $\tlh = \tlxi_0^{-1} \tlgamma$,
\begin{align}
&\int_{-\pi/2}^{\pi/2} \tlE_{1,\nu}^{(\infty)}\(\tln_x \tla(y^{1/2}) \tlk_\theta \) \phi(\tlk_\theta) \, d\theta\notag\\
&= i^{-\ell} \frac{1+i}{\sqrt{2}} \zeta_2(2\nu+1) \sum_{\tlgamma \in \tlGamma_0 \bs \tlGamma_1(4)} (a,c)_H \(\frac{c}{d}\) \(\frac{j_{\tlxi_0^{-1} \gamma}(z)}{|j_{\tlxi_0^{-1} \gamma}(z)|}\)^{\ell - \frac{1}{2}}  \Im\( \tlxi_0^{-1} \gamma z \)^{\frac{\nu+1}{2}}\notag\\
&= i^{-\ell} \frac{1+i}{\sqrt{2}} \dsE_{0,\ell}\(z, \frac{\nu+1}{2}\),\notag
\end{align}
for $\Re(\nu) > 1$ and $\phi = \phi_{-1,-\nu,-\ell}$.
For this last equality we used the fact that for any coset of $\tlGamma_0 \bs \tlGamma_1(4)$, we can always choose
$\tlgamma \in \tlGamma_1(4)$ such that $c > 0$.
For such $\tlgamma$ it follows that $(a,c)_H = 1$.
The uniqueness of meromorphic continuation then proves part (a) for $\fa = 0$.

The proof of part (b) will largely follow from \eqref{eq:EisDistToSer1b}.
Observe that by \eqref{eq:pairingDef},
\begin{align}
&\int_{-\pi/2}^{\pi/2} I_{1,\nu} \tlE_{1,-\nu}^{(\infty)}\(\tln_x \tla(y^{1/2}) \tlk_\theta \) \phi(\tlk_\theta) \, d\theta \\
&= \left\< I_{1,\nu} \(\pi\(\tln_x \tla(y^{1/2})\)^{-1} \tlE_{1,-\nu}^{(\infty)}\), \phi_{-1,-\nu,-\ell} \right\>_{1,\nu}.\notag
\end{align}
Thus by Lemma \ref{lem:AlmostAdjoint} and Lemma \ref{lem:KFiniteLem},
\begin{align}
&\int_{-\pi/2}^{\pi/2} I_{1,\nu} \tlE_{1,-\nu}^{(\infty)}\(\tln_x \tla(y^{1/2}) \tlk_\theta \) \phi(\tlk_\theta) \, d\theta \\
&= -i \left\< \pi\(\tln_x \tla(y^{-1/2})\)^{-1} \tlE_{1,-\nu}^{(\infty)}, I_{-1,\nu} \phi_{-1,-\nu,-\ell} \right\>_{1,-\nu} \notag\\
&= -i i^{\ell} \frac{1 + i}{\sqrt{2}} \cdot \frac{\pi  2^{1-\nu } \Gamma (\nu )}{\Gamma\left(-\frac{1}{2} \left(-\ell +\frac{1}{2}\right)+ \frac{\nu +1}{2} \right) \Gamma \left(\frac{1}{2} \left(-\ell +\frac{1}{2}\right) +\frac{\nu +1}{2}\right)}\notag\\
&\hh \cdot \left\< \pi\(\tln_x \tla(y^{1/2})\)^{-1} \tlE_{1,-\nu}^{(\infty)}, \phi_{-1,\nu,-\ell} \right\>_{1,-\nu} \notag\\
&= -i^{\ell+1} \frac{1 + i}{\sqrt{2}} \cdot \frac{\pi  2^{1-\nu } \Gamma (\nu )}{\Gamma \left( \frac{1}{4} + \frac{\ell}{2} + \frac{\nu}{2} \right) \Gamma\left( \frac{3}{4} - \frac{\ell}{2} + \frac{\nu}{2} \right)}\notag\\
&\hh \cdot \int_{-\pi/2}^{\pi/2} \tlE_{1,-\nu}^{(\infty)}\(\tln_x \tla(y^{1/2}) \tlk_\theta \) \phi_{-1,\nu,-\ell}(\tlk_\theta) \, d\theta. \notag
\end{align}
For $\Re(\nu)<-1$, it follows from \eqref{eq:EisDistToSer1b} that
\begin{align}
&\int_{-\pi/2}^{\pi/2} I_{1,\nu} \tlE_{1,-\nu}^{(\infty)}\(\tln_x \tla(y^{1/2}) \tlk_\theta \) \phi(\tlk_\theta) \, d\theta \\
&= -i^{\ell+1} \frac{1 + i}{\sqrt{2}} \cdot \frac{\pi  2^{1-\nu } \Gamma (\nu )}{\Gamma \left( \frac{1}{4} + \frac{\ell}{2} + \frac{\nu}{2} \right) \Gamma\left( \frac{3}{4} - \frac{\ell}{2} + \frac{\nu}{2} \right)}
i^{-\ell} \frac{1+i}{\sqrt{2}} \dsE_{\infty,\ell}\(z, \frac{-\nu+1}{2}\)\notag\\
&= \frac{\pi  2^{1-\nu } \Gamma (\nu )}{\Gamma \left( \frac{1}{4} + \frac{\ell}{2} + \frac{\nu}{2} \right) \Gamma\left( \frac{3}{4} - \frac{\ell}{2} + \frac{\nu}{2} \right)} \dsE_{\infty,\ell}\(z,\frac{-\nu+1}{2}\).\notag
\end{align}
The uniqueness of meromorphic continuation then proves part (b).
\end{proof}

By Theorem \ref{thm:FuncEq},
\begin{align*}
&\int_{-\pi/2}^{\pi/2} I_{1,\nu} \tlE_{1,-\nu}^{(\infty)}\(\tln_x \tla(y^{1/2}) \tlk_\theta \) \phi_{-1,-\nu,-\ell}(\tlk_\theta) \, d\theta\\
&= \Big((1-i) 2^{2\nu-1}(1-2^{-2\nu-1})^{-1} G_0(2 \nu)\Big) \\
&\hh \int_{-\pi/2}^{\pi/2} \(\tlE^{(\infty)}_{1,\nu}\(\tln_x \tla(y^{1/2}) \tlk_\theta\)  + (1-i) 2^{-\nu}(1 - 2^{2\nu}) \tlE^{(0)}_{1,\nu}\(\tln_x \tla(y^{1/2}) \tlk_\theta\)  \)\\
&\hh\hh \cdot \phi_{-1,-\nu,-\ell}(\tlk_\theta) \, d\theta.
\end{align*}
After applying Lemma \ref{lem:DistToSeries}, this equation becomes
\begin{align}
\label{eq:preFuncEq1}
&\frac{\pi  2^{1-\nu } \Gamma (\nu)}{\Gamma \left( \frac{1}{4} + \frac{\ell}{2} + \frac{\nu}{2} \right) \Gamma\left( \frac{3}{4} - \frac{\ell}{2} + \frac{\nu}{2} \right)} \dsE_{\infty,\ell}\(z,\frac{-\nu+1}{2} \)\\
&= i^{-\ell} \frac{1+i}{\sqrt{2}} \Big((1-i) 2^{2\nu-1}(1-2^{-2\nu-1})^{-1} G_0(2 \nu)\Big)\notag\\
&\hh \cdot \(\dsE_{\infty,\ell}\(z,\frac{\nu+1}{2} \) + (1-i) 2^{-\nu}(1 - 2^{2\nu}) \dsE_{0,\ell}\(z,\frac{\nu+1}{2} \)\).\notag
\end{align}
By the duplication formula for $\Gamma(\nu)$,
\begin{equation}
\frac{G_0(2\nu)}{\Gamma(\nu)} =  \pi^{-\frac{1}{2}-2\nu} \cos(\pi \nu) \Gamma\(\frac{1}{2} + \nu\).
\end{equation}
By applying this identity and simplifying \eqref{eq:preFuncEq1}, we find that
\begin{align}
&\dsE_{\infty,\ell}\(z,\frac{-\nu+1}{2} \)\\
&= i^{-\ell} \frac{2^{3 \nu -\frac{3}{2}} \pi ^{-2 \nu -\frac{3}{2}}}{1-2^{-2\nu-1}}
\cos(\pi \nu) \Gamma\(\nu + \frac{1}{2}\) \Gamma \( \frac{1}{4} + \frac{\ell}{2} + \frac{\nu}{2} \) \Gamma\( \frac{3}{4} - \frac{\ell}{2} + \frac{\nu}{2} \)\notag\\
&\hh \cdot \(\dsE_{\infty,\ell}\(z,\frac{\nu+1}{2} \) + (1-i) 2^{-\nu}(1 - 2^{2\nu}) \dsE_{0,\ell}\(z,\frac{\nu+1}{2} \)\).\notag
\end{align}
Replacing $\nu$ with $2s-1$ yields a more traditional form of the functional equation:

\begin{thm}
For $z \in \dsH$, $s \in \C$, $\ell \in 2\Z$,
\begin{align*}
&\dsE_{\infty,\ell}\(z, 1-s \)\\
&= - i^{-\ell} \frac{ 2^{6 s-\frac{9}{2}} \pi ^{\frac{1}{2}-4 s} }{ 1 - 2^{1 - 4 s} } \cos(2 \pi s)
\Gamma\(2s - \frac{1}{2}\) \Gamma \( s - \frac{1}{4} + \frac{\ell}{2} \) \Gamma\( s + \frac{1}{4} - \frac{\ell}{2} \)\notag\\
&\hh \cdot \Big(\dsE_{\infty,\ell}\(z,s\) + (1-i) 2^{-2 s-1} \(4-16^s\) \dsE_{0,\ell}\(z, s\)\Big).\notag
\end{align*}
\end{thm}

The following proposition will allow us to express $\dsE_{\fa,\ell}$ in a Fourier series expansion.

\begin{prop} 
\label{prop:DistTermsToSeriesTerms}
For $x \in \R$ and $y \in \R_{>0}$,
\begin{itemize}[leftmargin=.3in,font=\normalfont\textbf]
\item[(a)] $\displaystyle \left\<\pi\(\tln_x \tla(y^{1/2})\)^{-1} \onebb_{1,\nu}, \phi_{-1,-\nu,-\ell} \right\>_{1,\nu} \\
\hhh = \frac{\pi  2^{1-\nu}  \Gamma (\nu )}{\Gamma\left(\frac{1}{2} \left(\frac{1}{2} - \ell\right) + \frac{\nu +1}{2} \right) \Gamma \left(-\frac{1}{2} \left(\frac{1}{2} - \ell\right) +\frac{\nu +1}{2}\right)} y^{\frac{-\nu + 1}{2}}$,

\item[(b)] $\displaystyle \left\<\pi\(\tln_x \tla(y^{1/2})\)^{-1} \delta_{1,\nu,\infty}, \phi_{-1,-\nu,-\ell} \right\>_{1,\nu} = i^\ell \frac{1+i}{\sqrt{2}} y^{\frac{\nu + 1}{2}}$,

\item[(c)] $\displaystyle \left\<\pi\(\tln_x \tla(y^{1/2})\)^{-1} \dse_{1,\nu,n}, \phi_{-1,-\nu,-\ell} \right\>_{1,\nu} = e(n x) y^{\frac{-\nu+1}{2}} W_\nu(ny)$,
\end{itemize}
where
\begin{equation}
\label{eq:WhittakerFunction}
W_\nu(y) = \int_{-\infty}^\infty e(ty) (1-it)^{\frac{1}{2} \(\frac{1}{2}-\ell\) - \frac{\nu+1}{2}} (1+it)^{-\frac{1}{2}\(\frac{1}{2}-\ell\) - \frac{\nu+1}{2}} \, dt.
\end{equation}
\end{prop}

\begin{proof}
By the uniqueness of meromorphic continuation, it suffices to prove the above identities for $\Re(\nu)>0$.
By Proposition \ref{prop:pairingToCondConvIntegral},
\begin{align}
\label{eq:WhittakerStep1A}
&\left\<\pi\(\tln_x \tla(y^{1/2})\)^{-1} \onebb_{1,\nu}, \phi_{-1,-\nu,-\ell} \right\>_{1,\nu}\\
&= \int_{-\infty}^\infty \onebb_{1,\nu} \(\tln_x \tla(y^{1/2}) \tln_t\) \phi_{-1,-\nu,-\ell}(\tln_t) \, dt.\notag
\end{align}
Since
\begin{equation}
\label{eq:WhitElemId1}
\tln_t = \tlk(-\arctan(t))\tla\( (1+t^2)^{-1/2} \) \tln_-\(\frac{t}{1+t^2}\)
\end{equation}
then by \eqref{eq:phiDefB},
\begin{equation}
\label{eq:phiArcTan}
\phi_{-1,-\nu,-\ell}(\tln_t) = \phi_{-1,-\nu,-\ell}(\tlk(-\arctan(t))) (1+t^2)^{-\frac{\nu+1}{2}}.
\end{equation}
Thus by \eqref{eq:phiArcTanToI} and \eqref{eq:branchCutResult},
\begin{align}
\label{eq:WhittakerStep2A}
&\phi_{-1,-\nu,-\ell}(\tln_t) = (1-i t)^{\frac{1}{2} \(\frac{1}{2}-\ell\)} (1+i t)^{-\frac{1}{2} \(\frac{1}{2}-\ell\)} (1+t^2)^{-\frac{\nu+1}{2}}\\
&= (1-i t)^{\frac{1}{2} \(\frac{1}{2}-\ell\) - \frac{\nu+1}{2}} (1+i t)^{-\frac{1}{2} \(\frac{1}{2}-\ell\) - \frac{\nu+1}{2}}.\notag
\end{align}
Since
\begin{equation}
\label{eq:WhitElemId2}
\tln_x \tla(y^{1/2}) \tln_t = \tln(x+ty) \tla(y^{1/2}),
\end{equation}
then by \eqref{eq:TransLawsB}
\begin{equation}
\label{eq:WhittakerStep3A}
\onebb_{1,\nu}\(\tln_x \tla(y^{1/2}) \tln_t\) = y^{\frac{-\nu + 1}{2}}.
\end{equation}
Thus by \eqref{eq:WhittakerStep1A}, \eqref{eq:WhittakerStep2A}, and \eqref{eq:WhittakerStep3A},
\begin{align*}
&\left\<\pi\(\tln_x \tla(y^{1/2})\)^{-1} \onebb_{1,\nu}, \phi_{-1,-\nu,-\ell} \right\>_{1,\nu}\\
&= y^{\frac{-\nu + 1}{2}} \int_{-\infty}^\infty (1-i t)^{\frac{1}{2} \(\frac{1}{2}-\ell\) - \frac{\nu+1}{2}} (1+i t)^{-\frac{1}{2} \(\frac{1}{2}-\ell\) - \frac{\nu+1}{2}} dt.
\end{align*}
Part (a) then follows from \eqref{eq:InterTwinedKFinite}.

By Lemma \ref{lem:deltaInvN},
\begin{equation}
\pi\(\tln_x \tla(y^{1/2})\)^{-1} \delta_{1,\nu,\infty} = \pi\(\tla(y^{1/2})\)^{-1} \delta_{1,\nu,\infty}.
\end{equation}
Therefore by \eqref{eq:cmpctToUnbdd2} and the definition of $\delta_{1,\nu,\infty}$ (given just above \eqref{eq:tlsdeltprop}),
\begin{align}
\label{eq:WhittakerStep1B}
&\left\<\pi\(\tln_x \tla(y^{1/2})\)^{-1} \delta_{1,\nu,\infty}, \phi_{-1,-\nu,-\ell} \right\>_{1,\nu}\\
&= \int_{-\infty}^\infty \(\pi\(\tln_x \tla(y^{1/2})\)^{-1} \delta_{1,\nu,\infty}\)_\infty(\tln_t) \cdot \(\phi_{-1,-\nu,-\ell}\)_\infty(t) \, dt\notag\\
&= \int_{-\infty}^\infty \delta_{1,\nu,\infty}\(\tla\(y^{1/2}\) \tls^{-1} \tln_t\)\phi_{-1,-\nu,-\ell}(\tls^{-1} \tln_t) \, dt\notag\\
&= \int_{-\infty}^\infty \delta_{1,\nu,0}\(\tls^{-1} \tla\(y^{1/2}\) \tls^{-1} \tln_t\)\phi_{-1,-\nu,-\ell}(\tls^{-1} \tln_t) \, dt.\notag
\end{align}
Since
\begin{equation}
\tls^{-1} \tla\(y^{1/2}\) \tls^{-1} \tln_t = \tln(t/y) \tla\(y^{-1/2}\) \tlm_{-1,1} = \tln(t/y) \tla\(y^{-1/2}\) \(-\id,1\),
\end{equation}
by \eqref{eq:TransLawsB}
\begin{align}
\label{eq:WhittakerStep2B}
&\left\<\pi\(\tln_x \tla(y^{1/2})\)^{-1} \delta_{1,\nu,\infty}, \phi_{-1,-\nu,-\ell} \right\>_{1,\nu}\\
&= \int_{-\infty}^\infty i y^{\frac{\nu - 1}{2}} \delta_{1,\nu,0}\(\tln(t/y)\) \phi_{-1,-\nu,-\ell}(\tls^{-1} \tln_t) \, dt\notag\\
&= i y^{\frac{\nu + 1}{2}} \int_{-\infty}^\infty \delta_{1,\nu,0}\(\tln_t\) \phi_{-1,-\nu,-\ell}(\tls^{-1} \tln(ty)) \, dt\notag\\
&= y^{\frac{\nu + 1}{2}} i \phi_{-1,-\nu,-\ell}(\tls^{-1}).\notag
\end{align}
By \eqref{eq:tlsdef} and \eqref{eq:phiDef},
\begin{align*}
&i\phi_{-1,-\nu,-\ell}(\tls^{-1}) = i \phi_{-1,-\nu,-\ell}\(\tlk\(\frac{-\pi}{2}\)\) = i \exp\(-\(\frac{1}{2} - \ell\) \frac{\pi i}{2}\)\\
&= e(1/4) e\(\frac{1}{4}\(\ell - \frac{1}{2}\)\) = e\(\frac{1}{4}\(\ell + \frac{1}{2}\)\) = i^\ell \frac{1+i}{\sqrt{2}}.
\end{align*}
Part (b) follows from this and \eqref{eq:WhittakerStep2B}.

By Proposition \ref{prop:pairingToCondConvIntegral},
\begin{align}
\label{eq:WhittakerStep1}
&\left\<\pi\(\tln_x \tla(y^{1/2})\)^{-1} \dse_{1,\nu,n}, \phi_{-1,-\nu,-\ell} \right\>_{1,\nu}\\
&= \int_{-\infty}^\infty \dse_{1,\nu,n}\(\tln_x \tla(y^{1/2}) \tln_t\) \phi_{-1,-\nu,-\ell}(\tln_t) \, dt.\notag
\end{align}
By \eqref{eq:WhitElemId2}, \eqref{eq:TransLawsB}, and \eqref{eq:dseDef},
\begin{equation}
\label{eq:WhittakerStep3}
\dse_{1,\nu,n}\(\tln_x \tla(y^{1/2}) \tln_t\) = e(n (x+ty)) y^{\frac{-\nu + 1}{2}}.
\end{equation}
Thus by \eqref{eq:WhittakerStep1} and \eqref{eq:WhittakerStep2A},
\begin{align*}
&\left\<\pi\(\tln_x \tla(y^{1/2})\)^{-1} \dse_{1,\nu,n}, \phi_{-1,-\nu,-\ell} \right\>_{1,\nu}\\
& = \int_{-\infty}^\infty e(n (x+ty)) y^{\frac{-\nu + 1}{2}} (1-i t)^{\frac{1}{2} \(\frac{1}{2}-\ell\) - \frac{\nu+1}{2}} (1+i t)^{-\frac{1}{2} \(\frac{1}{2}-\ell\) - \frac{\nu+1}{2}}\\
&= e(n x) y^{\frac{-\nu + 1}{2}} W_\nu(ny),
\end{align*}
which completes the proof of part (c).
\end{proof}

By Lemma \ref{lem:DistToSeries}(a) and \eqref{eq:pairingDef},
\begin{equation}
\dsE_{\fa,\ell}\(z,\frac{\nu+1}{2}\) = i^\ell \frac{1-i}{\sqrt{2}} \left\<\pi\(\tln_x \tla(y^{1/2})\)^{-1} \tlE_{1,\nu}^{(\fa)}, \phi_{-1,-\nu,-\ell}\right\>_{1,\nu},
\end{equation}
where $z = x+iy \in \dsH$.
By Proposition \ref{prop:EisenInfMeroCont}, Proposition \ref{prop:DistTermsToSeriesTerms}, and the fact that $\ell$ is even,
\begin{align}
\label{eq:InfEisenSeriesFourExp}
&\dsE_{\infty,\ell}\(z,\frac{\nu+1}{2}\)\\
&= i^\ell \frac{1-i}{\sqrt{2}} \frac{\pi  2^{1-\nu}  \Gamma (\nu )}{\Gamma\left(\frac{1}{2} \left(\frac{1}{2} - \ell\right) + \frac{\nu +1}{2} \right) \Gamma \left(-\frac{1}{2} \left(\frac{1}{2} - \ell\right) +\frac{\nu +1}{2}\right)} a_{1,\nu}(0) y^{\frac{-\nu + 1}{2}}\notag\\
&+ a_{1,\nu}(\infty) y^{\frac{\nu+1}{2}} + i^\ell \frac{1-i}{\sqrt{2}}  \sum_{n \in \Z_{\neq 0}} e(nx) a_{1,\nu}(n) y^{\frac{-\nu+1}{2}} W_\nu(ny),\notag
\end{align}
where $a_{1,\nu}(0)$ is given in Proposition \ref{prop:EisenInfConst}, $a_{1,\nu}(\infty)$ is given in 
Proposition \ref{prop:EisenInfMeroCont}, and $a_{1,\nu}(n)$ for $n \in \Z_{\neq 0}$ is given in Theorem \ref{thm:SimplFourCoeff}.
Similarly, by \eqref{eq:ComplFourSerEisAtZero} and Proposition \ref{prop:DistTermsToSeriesTerms},
\begin{align}
\label{eq:ZeroEisenSeriesFourExp}
&\dsE_{0,\ell}\(z,\frac{\nu+1}{2}\)\\
&= i^\ell \frac{1-i}{\sqrt{2}} \frac{\pi  2^{1-\nu}  \Gamma (\nu )}{\Gamma\left(\frac{1}{2} \left(\frac{1}{2} - \ell\right) + \frac{\nu +1}{2} \right) \Gamma \left(-\frac{1}{2} \left(\frac{1}{2} - \ell\right) +\frac{\nu +1}{2}\right)} b_{1,\nu}(0) y^{\frac{-\nu + 1}{2}}\notag\\
&+ i^\ell \frac{1-i}{\sqrt{2}}  \sum_{n \in \Z_{\neq 0}} e(nx) b_{1,\nu}(n) y^{\frac{-\nu+1}{2}} W_\nu(ny),\notag
\end{align}
where $b_{1,\nu}(0)$ is given in Proposition \ref{prop:b0coeff}.
Replacing $\nu$ with $2s-1$ in \eqref{eq:InfEisenSeriesFourExp} and \eqref{eq:ZeroEisenSeriesFourExp} yields a more traditional form of the 
these Fourier series expansions.

The following proposition shows that the Whittaker function $W_\nu(y)$ defined in \eqref{eq:WhittakerFunction}
is compatible with the Whittaker function defined in \cite{Goldfeld85}.

\begin{prop} For $y \in \R_{>0}$,
\begin{equation*}
W_\nu(ny) = i^{-\ell} \frac{1+i}{\sqrt{2}} y^{\nu} \int_{-\infty}^\infty \frac{e(-nt)}{\(y^2 + t^2\)^{-\frac{1}{2}\(\frac{1}{2} - \ell\)+\frac{\nu+1}{2}}
\(t+iy\)^{\frac{1}{2} - \ell}} \, dt.
\end{equation*}
\end{prop}

\begin{proof}
By \eqref{eq:WhittakerFunction}, \eqref{eq:branchCutResult}, and a change of variable,
\begin{align*}
&W_\nu(ny) = \int_{-\infty}^\infty e(nty) (1-it)^{\frac{1}{2} \(\frac{1}{2}-\ell\)} (1+it)^{-\frac{1}{2}\(\frac{1}{2}-\ell\)} (1+t^2)^{-\frac{\nu+1}{2}} \, dt\\
&= \int_{-\infty}^\infty e(-nty) (1+it)^{\frac{1}{2} \(\frac{1}{2}-\ell\)} (1-it)^{-\frac{1}{2}\(\frac{1}{2}-\ell\)} (1+t^2)^{-\frac{\nu+1}{2}} \, dt.
\end{align*}
One can show that for $\alpha \in \C$ and $t \in \R$,
\begin{equation}
\label{eq:SqrtId}
(1+ it)^{\alpha} (1- i t)^{-\alpha}  = e\(\frac{\alpha}{2}\) (1 + t^2)^{\alpha} (t+i)^{-2\alpha}.
\end{equation}
Applying \eqref{eq:SqrtId} with $\alpha = \frac{1}{2}\(\frac{1}{2}-\ell\)$ shows that
\begin{align*}
W_\nu(ny) &= e\(\frac{1}{4}\(\frac{1}{2}-\ell\)\) \int_{-\infty}^\infty e(-nty) (1+t^2)^{\frac{1}{2}\(\frac{1}{2}-\ell\)-\frac{\nu+1}{2}}
(t+i)^{\ell - \frac{1}{2}} \, dt\\
&= i^{-\ell} \frac{1+i}{\sqrt{2}} \int_{-\infty}^\infty \frac{e(-nty)}{(1+t^2)^{-\frac{1}{2}\(\frac{1}{2} - \ell\)+\frac{\nu+1}{2}}
(t+i)^{\frac{1}{2}-\ell}} \, dt.\notag
\end{align*}
Performing the change of variables $t \mapsto \frac{t}{y}$ shows that
\begin{align*}
&W_\nu(ny) = i^{-\ell} \frac{1+i}{\sqrt{2}} \int_{-\infty}^\infty \frac{e(-nt)}{\(1+\frac{t^2}{y^2}\)^{-\frac{1}{2}\(\frac{1}{2} - \ell\)+\frac{\nu+1}{2}}
\(\frac{t}{y}+i\)^{\frac{1}{2}-\ell}} \, \frac{dt}{y}\\
&= i^{-\ell} \frac{1+i}{\sqrt{2}} \int_{-\infty}^\infty \frac{e(-nt)}{y^{\(\frac{1}{2} - \ell\)-(\nu+1)} \(y^2 + t^2\)^{-\frac{1}{2}\(\frac{1}{2} - \ell\)+\frac{\nu+1}{2}}
y^{-\frac{1}{2} + \ell}\(t+iy\)^{\frac{1}{2}-\ell}} \, \frac{dt}{y}\\
&= i^{-\ell} \frac{1+i}{\sqrt{2}} y^{\nu} \int_{-\infty}^\infty \frac{e(-nt)}{\(y^2 + t^2\)^{-\frac{1}{2}\(\frac{1}{2} - \ell\)+\frac{\nu+1}{2}}
\(t+iy\)^{\frac{1}{2} - \ell}} \, dt.\qedhere
\end{align*}
\end{proof}

\sloppy  
\printbibliography
\end{document}